\documentclass{article}
\newcommand{\version}{5 July 2023}
\usepackage{comment}
\usepackage{pgfplots}
\pgfplotsset{compat=1.15}
\usepackage{ifpdf} 
\makeatletter
\edef\texforht{TT\noexpand\fi
  \@ifpackageloaded{tex4ht}
    {\noexpand\iftrue}
    {\noexpand\iffalse}}
\makeatother

\nonstopmode
\newcommand{\ds}{\displaystyle}
\usepackage{upgreek}

\usepackage{graphics}

\usepackage{ulem} %to-be-commented-out
\providecommand{\emph}[1]{{\it #1}}
\renewcommand{\emph}[1]{{\it #1}}%% or else ulem changes \em to \underline

\usepackage{amsthm}
\usepackage{latexsym,epsfig,bm}

\usepackage{bm}
\usepackage{upgreek}

\usepackage{mathrsfs}
\usepackage{times}
\usepackage{amssymb}

\definecolor{MyDarkGreen}{rgb}{0,0.6,0.1}
\usepackage[backref=none]{hyperref}
\hypersetup{pdfborder={0 0 0},
  colorlinks,
  urlcolor={blue},
  linkcolor={blue},
  citecolor={blue},
  breaklinks=true}

\usepackage{doi}
\providecommand{\eprint}[1]{}
\renewcommand{\eprint}[1]{arXiv:\href{http://arxiv.org/abs/#1}{#1}}

\usepackage{amsmath}

\providecommand{\eqref}[1]{{\rm (\ref{#1})}}
\providecommand{\itref}[1]{{\it (\ref{#1})}}

\DeclareSymbolFont{AMSb}{U}{msb}{m}{n}
\DeclareSymbolFontAlphabet{\mathbb}{AMSb}

\DeclareSymbolFont{EUB}{U}{eur}{b}{n}
\SetSymbolFont{EUB}{bold}{U}{eur}{b}{n}
\DeclareSymbolFontAlphabet{\eub}{EUB}

\newcommand{\jj}{\mathrm{i}}

\newcommand{\calP}{\mathcal{P}}

\newcommand{\notyet}[1]{{}}

\newcommand{\sgn}{\mathop{\rm sgn}}

\newcommand{\range}{\mathop{\rm Range\,}}
\newcommand{\supp}{\mathop{\rm supp}}

\newcommand{\p}{\partial}
\newcommand{\at}[1]{\vert\sb{\sb{#1}}}

\newcommand{\R}{\mathbb{R}}
\newcommand{\C}{\mathbb{C}}

\newcommand{\bfA}{\mathbf{A}}
\newcommand{\bfH}{\mathbf{H}}
\newcommand{\bfX}{\mathbf{X}}

\newcommand{\N}{\mathbb{N}}

\newcommand{\abs}[1]{\vert #1 \vert}

\newcommand{\norm}[1]{\Vert #1 \Vert}

\newcommand{\dom}{\mathfrak{D}}

\DeclareMathSymbol{\varGamma}{\mathord}{letters}{"00}
\DeclareMathSymbol{\varDelta}{\mathord}{letters}{"01}
\DeclareMathSymbol{\varTheta}{\mathord}{letters}{"02}
\DeclareMathSymbol{\varLambda}{\mathord}{letters}{"03}
\DeclareMathSymbol{\varXi}{\mathord}{letters}{"04}
\DeclareMathSymbol{\varPi}{\mathord}{letters}{"05}
\DeclareMathSymbol{\varSigma}{\mathord}{letters}{"06}
\DeclareMathSymbol{\varUpsilon}{\mathord}{letters}{"07}
\DeclareMathSymbol{\varPhi}{\mathord}{letters}{"08}
\DeclareMathSymbol{\varPsi}{\mathord}{letters}{"09}
\DeclareMathSymbol{\varOmega}{\mathord}{letters}{"0A}

\theoremstyle{plain}
\newtheorem{lemma}{Lemma}[section]
\newtheorem{theorem}[lemma]{Theorem}

\theoremstyle{definition}
\newtheorem{definition}[lemma]{Definition}
\theoremstyle{remark}
\newtheorem{remark}[lemma]{Remark}
\newcounter{step}

\makeatletter\@addtoreset{equation}{section}
\makeatletter\@addtoreset{lemma}{section}
\makeatother

\def\mod{\mathop{\rm mod}\nolimits}

\renewcommand{\Re}{\mathop{\rm{R\hskip -1pt e}}\nolimits}
\renewcommand{\Im}{\mathop{\rm{I\hskip -1pt m}}\nolimits}

\begin{document}

\title{Spectral stability and instability of solitary waves
of the Dirac equation with concentrated nonlinearity}

\author{
{\sc Nabile Boussa{\"\i}d}
\\
{\it\small Universit\'e Bourgogne Franche-Comt\'e, 25030 Besan\c{c}on CEDEX, France}
\\~\\
{\sc Claudio Cacciapuoti}
\\
{\it\small DiSAT, Sezione di Matematica, Universit\`a dell'Insubria, via Valleggio 11, I-22100 Como, Italy}
% claudio.cacciapuoti@uninsubria.it
\\~\\
{\sc Raffaele Carlone}
\\
{\it\small Dipartimento di Matematica e Applicazioni R. Caccioppoli,  Universit\`a Federico II di Napoli,}
\\
{\it\small MSA, via Cinthia, I-80126, Napoli, Italy}
%% %\email{raffaele.carlone@unina.it}
\\~\\
{\sc Andrew Comech}
\\
{\it\small
Mathematics Department,
Texas A\&M University, College Station, TX 77843, USA}
\\~\\
{\sc Diego Noja}
\\
{\it\small Dipartimento di Matematica e Applicazioni, Universit\`a di Milano Bicocca,}
\\
{\it\small via R. Cozzi 55, I-20126 Milano, Italy}
%\email{diego.noja@unimib.it}
\\~\\
{\sc Andrea Posilicano}
\\
{\it\small DiSAT, Sezione di Matematica, Universit\`a dell'Insubria, via Valleggio 11, I-22100 Como, Italy}
%\email{andrea.posilicano@uninsubria.it}
}

\date{\version}

\maketitle

\begin{abstract}
%\ac{REMEMBER TO UNCOMMENT THE PICTURES!!! :):)}
We consider the nonlinear Dirac equation with Soler-type nonlinearity
concentrated at one point
and present a detailed study of the spectrum of linearization at solitary waves.
We then consider two different perturbations of the nonlinearity which break the $\mathbf{SU}(1,1)$ symmetry:
the first preserving and the second breaking the parity symmetry.
We show that a particular perturbation 
which breaks the $\mathbf{SU}(1,1)$ symmetry
but not the parity symmetry also preserves the spectral stability
of solitary waves.
Then we consider a particular perturbation
which breaks both the $\mathbf{SU}(1,1)$ symmetry
and the parity symmetry and show that this perturbation
destroys the stability of weakly relativistic solitary waves.
This instability is due to the bifurcations of positive-real-part eigenvalues
from the embedded eigenvalues $\pm 2\omega\jj$.
\end{abstract}

%\tableofcontents
%% \begin{keyword}
%% Carleman estimates\sep
%% nonlinear Dirac equation\sep
%% Soler model\sep
%% %solitary waves\sep
%% spectral stability\sep
%% stability of solitary waves\sep
%% spectral theory of nonselfadjoint operators
%% \MSC[2010]
%% 35B10 % Periodic solutions
%% \sep 35B32 % Bifurcation
%% \sep 35B35 % Stability
%% %35B40 % Asymptotic behavior of solutions
%% %35B41 % Attractors
%% \sep 35C08 % Soliton solutions
%% \sep 35Q41 % Time-dependent Schr\"odinger equations, Dirac equations
%% %37K40 % Soliton theory, asymptotic behavior of solutions
%% %37N20 % Dynamical systems in other branches of physics (quantum mechanics, general relativity, laser physics
%% \sep 81Q05 % Closed and approximate solutions to the Schr\"odinger, Dirac, Klein--Gordon and other quantum-mechanical equations
%% %35B32\sep
%% % 00-01\sep  99-00
%% \end{keyword}

\section{Introduction}

In this article, we study a nonlinear Dirac equation (NLD)
in one dimension
with nonlinearity concentrated at a point,
\begin{eqnarray}\label{nld-point-0}%\tab{NLD}
\jj \p_t\psi=D_m\psi
-\delta(x)f(\psi\sp\ast\sigma_3\psi)\sigma_3\psi,
\qquad
\psi(t,x)\in\C^2,
\quad
x\in\R,
\quad
t\in\R,
\end{eqnarray}
and study stability of its solitary wave solutions.
We also consider
how this stability is affected by certain perturbations.
Above, the free Dirac operator
in one spatial dimension
is taken in the form
\[
D_m=\jj\sigma_2\p_x+m\sigma_3
=\begin{bmatrix}m&\p_x\\-\p_x&-m\end{bmatrix},
\qquad
m>0\textup,
\]
where $f$
is a differentiable real-valued function,
while the standard Pauli matrices are given by
\[
\sigma_{1}=\begin {bmatrix}0&1 \\1&0\end {bmatrix},
\qquad 
\sigma_{2}=\begin {bmatrix}0&-\jj \\\jj&0\end {bmatrix},
\qquad
\sigma_{3}=\begin {bmatrix}1&0 \\0&-1\end {bmatrix}.
\]
The name \emph{concentrated nonlinearity} comes
from the presence of the delta distribution
at the
nonlinear term in the equation.

A rigorous definition of the model is given in Section~\ref{sect-model}, while for the accurate general treatment we refer to \cite{cacciapuoti2017one}. In the usual setting (where the Dirac distribution is missing and the nonlinearity is everywhere distributed), the nonlinearity $f(\psi\sp\ast\sigma_3\psi)\sigma_3\psi$ appearing in \eqref{nld-point-0} defines what is known as Soler model (also called Gross--Neveau model in one spatial dimension, $(1+1)$D); by analogy, the above equation describes a Soler-type concentrated nonlinearity.
We mention that the analysis of various PDEs with concentrated nonlinearities is now a well-developed subject.
Rigorous studies have been performed especially, but not only,
in the Schr\"odinger case (see \cite{adami2001class, adami2003cauchy, noja2005wave, carlone2019well} and references therein).
The local
and global
well-posedness of the nonlinear Dirac equation (NLD) with concentrated nonlinearity is given
in the already cited \cite{cacciapuoti2017one}
and the extension to quantum graphs has also been considered
in \cite{borrelli2019nonlinear,Carlone2019extended}.
%% Justification of the use of concentrated nonlinearity in the NLS case
%% is given in \cite{komech2007global,cacciapuoti2014nls},
%% starting from extended nonlinearities and taking the point limit on solutions.
%% The justification of the NLD with concentrated nonlinearity
%% could be treated along similar lines but up to now it is open.
The local and global well-posedness for the NLS with concentrated nonlinearity
is given in \cite{komech2007global,cacciapuoti2014nls},
starting from extended nonlinearities and taking the point limit on solutions.
%%AC A similar well-posedness of the NLD with concentrated nonlinearity
%%AC could be treated along similar lines but up to now it is open.
%% Diego, as of Dec 11:
A similar derivation and global
well-posedness of the one-dimensional NLD
with concentrated nonlinearity
could be treated along similar lines, although up to now this problem is open.
Our interest in this kind of nonlinearity is raised by the possibility of characterizing explicitly the solitary waves of the model and of giving a fairly complete spectral theory of the linearization around solitary waves.
While in the usual Soler model it seems difficult to have complete and definite results on the spectral stability of solitary waves, in the present example, simplified yet nontrivial, spectral stability and instability of some classes of solitary waves can be established.
(We recall that a solitary wave solution $\phi_\omega(x) e^{-\jj\omega t}$ of the NLD is \emph{spectrally stable} if the spectrum of the linearization operator around the solitary wave has no points
in the right half of the complex plane; in the opposite case we say that the solitary wave is linearly unstable.)

Knowledge of the linearization spectrum and in particular the spectral stability of solitary waves is important because it is a fundamental step towards the analysis of their asymptotic stability.
In previous works on asymptotic stability of solitary waves of NLD
\cite{boussaid2006stable,boussaid2008asymptotic,boussaid2012stability,MR2985264,MR3592683}, their spectral stability was either taken as an assumption,
or checked numerically.
For analytical approaches to the spectral stability in NLD, see \cite{berkolaiko2012spectral,boussaid2016spectral,boussaid2017nonrelativistic,boussaid2018spectral,linear-b,aldunate2021results},
and also the monograph \cite{opus}.
Let us mention a recent related article
on the orbital stability of solitary waves
in the Klein--Gordon equation with concentrated
nonlinearity \cite{comech2021orbital},
with the complete analysis of the spectrum
of the linearized equation.

In Section~\ref{sect-model},
we study the solitary waves (Lemma~\ref{lemma-sw} below)
and then treat the spectrum of the linearized system.
Let us give the essence of our Theorem~\ref{theorem-sect2}
on a particular case of a pure power nonlinearity
$f(\tau)=\abs{\tau}^\kappa$, $\kappa>0$.
Considering solitary waves with frequencies in the gap $\omega\in(-m,m)$,
the spectrum of the linearization is as follows:
there are always eigenvalues $\pm 2\omega\jj$
(embedded into the continuous spectrum when $\abs{\omega}>m/3$);
when
$\kappa\in (0,1]$,
the entire spectrum is located on the imaginary axis;
there are two
simple nonzero eigenvalues when $\kappa\in \big(2^{-1/2},1\big]$
and the frequency satisfies $\omega>\mathcal{T}_\kappa$,
with some $\mathcal{T}_\kappa\in(0,m)$.
For $\kappa>1$, these two imaginary eigenvalues collide
at the origin
when $\omega=\Omega_\kappa$,
with $\Omega_\kappa\in(\mathcal{T}_\kappa,m)$ the second threshold value,
and a couple of nonzero real eigenvalues appear from this collision
when $\omega\in(\Omega_\kappa,m)$.
This second threshold value is the one corresponding to algebraic multiplicity
of the null space of the linearization jumping from two to four.
This value satisfies the Kolokolov condition \cite{kolokolov-1973}:
$\p_\omega \norm{\phi_\omega}_{L^2}^2$ vanishes at $\omega=\Omega_\kappa$.
The statement and proof of these results are in Section~\ref{sect-spectrum}.
The detailed structure of the spectrum
of the linearization at a solitary wave
is formulated in Theorem~\ref{theorem-sect2}
(which is proved in Section~\ref{sect-theorem-sect2-proof}).
A relevant part of the analysis relies on the parity symmetry of the
Soler model,
which allows one to split the Hilbert space into two invariant subspaces:
odd-even-odd-even and even-odd-even-odd subspaces.
In the former subspace live the ``trivial'' eigenvalues $\pm 2\omega\jj$
and in the latter subspace live the possibly further ``nontrivial'' eigenvalues.

The presence of real eigenvalues
in the case $\kappa>1$ rules out spectral stability of the corresponding solitary waves. As explained above, for any positive power $\kappa$ and any $\omega\in (-m,m)$, besides eigenvalue $0$ and possible nontrivial eigenvalues referred above, the point spectrum contains purely imaginary eigenvalues $\pm 2\omega\jj$. These eigenvalues are related to the $\mathbf{SU}(1,1)$ symmetry of the Soler model (see \cite{galindo1977remarkable})
and to the existence of bi-frequency solitary waves
(see \cite{berkolaiko2012spectral,boussaid2018spectral} and Remark~\ref{remark-bi-frequency} in the present paper).
By \cite{linear-b},
the spectral stability 
of small amplitude solitary waves of the Soler model
in dimensions $n\ge 1$
heavily relies
on the presence of $\pm 2\omega\jj$ eigenvalues in the spectrum
of the linearized equation.

If the symmetry responsible of the $\pm 2\omega\jj$ eigenvalues is broken,
then
in principle one could expect that the eigenvalues $\pm 2\omega\jj$
bifurcate off the imaginary axis,
either becoming eigenvalues with nonzero real part,
or turning into resonances,
that is, poles of the resolvent on the unphysical sheet of its Riemann surface;
the second part of the paper is dedicated to the analysis of this issue.
We consider examples of perturbations
of the nonlinearity
which destroy the $\mathbf{SU}(1,1)$ symmetry;
we are interested in the fate of the eigenvalues $\pm 2\omega \jj$
associated to the $\mathbf{SU}(1,1)$ symmetry.
In Section~\ref{sect-nld-perturbation},
we consider a perturbation which preserves the parity
(the self-interaction is based on the quantity
$\psi\sp\ast(\sigma_3+\epsilon I_2)\psi$, $\epsilon\ne 0$,
instead of $\psi\sp\ast\sigma_3\psi$)
and show that solitary waves remain spectrally stable (if they were stable in the Soler model).
We say that this class of perturbations preserves the parity in the sense that the
operator corresponding to the linearization at a solitary wave is invariant
in odd-even-odd-even and even-odd-even-odd subspaces.

In Section~\ref{sect-nld-perturbation-broken-parity}
we consider a perturbation
when
the self-interaction is based on the quantity $\psi\sp\ast(\sigma_3+\epsilon\sigma_1)\psi$, $\epsilon\ne 0$,
instead of $\psi\sp\ast\sigma_3\psi$.
This perturbation breaks not only $\mathbf{SU}(1,1)$ symmetry,
but also the parity symmetry (in the above sense).
We show that such a perturbation leads to linear instability
of weakly relativistic solitary waves
(when $\omega<m$ is close enough to $m$).
We point out that in the model under consideration
the $\pm 2\omega\jj$ eigenvalues of the linearized operator,
the ones which are due to the $\mathbf{SU}(1,1)$ symmetry of the model, are simple (in the sense
that they correspond to a one-dimensional eigenspace). Due to the symmetries of the spectrum with respect to the real and imaginary axes,
these two eigenvalues could not bifurcate off the imaginary axis if they were isolated
(this is the case when $\abs{\omega}<m/3$).
The linear instability that we prove in the nonrelativistic regime ($\omega$ is close to $m$)
is only possible since
in the unperturbed case
these two eigenvalues are embedded into the essential spectrum:
under the perturbation,
an eigenvalue corresponding to a one-dimensional eigenspace
can bifurcate to both sides of the imaginary axis.

Let us make one more comment.
The $\mathbf{SU}(1,1)$ symmetry is absent
for the physically relevant Dirac--Maxwell system
(with the nonlinear Dirac equation being its effective reduction).
%in a suitable approximation.
We presently do not know
whether solitary waves in the Dirac--Maxwell system are spectrally stable.
This question was one of the motivations
for the study of the relation of
the broken $\mathbf{SU}(1,1)$ symmetry and spectral stability
undertaken in the present article.

%\end{remark}

%While the general case seems complicated,
%we can deal with the Dirac equation
%with concentrated nonlinearity.
%In this model, we will study the behavior of the stability properties of
%solitary waves
%under the perturbation of the nonlinearity
%which breaks down the symmetries of the Soler model.

\section{The Soler model with concentrated nonlinearity}
\label{sect-model}

We are looking for solitary wave solutions
$\psi(t,x)=\phi(x)e^{-\jj\omega t}$
to the Dirac equation with nonlinear self-interaction
of Soler type
which is concentrated at the origin.
This reads formally as
\begin{eqnarray}\label{nld-point}
\jj \p_t\psi=D_m\psi
-\delta(x)f(\psi\sp\ast\sigma_3\psi)\sigma_3\psi,
\qquad
\psi(t,x)\in\C^2,
\quad
x\in\R,
\quad
t\in\R,
\end{eqnarray}
with
\begin{eqnarray}\label{def-dm}
D_m=\jj\sigma_2\p_x+\sigma_3 m
=
\begin{bmatrix}m&\p_x
\\-\p_x&-m
\end{bmatrix}
%\qquad
%\dom(D_m)=H^1(\R,\C^2)
\end{eqnarray}
and with the nonlinearity represented by
\begin{eqnarray}\label{f-such}
f\in C(\R)\cap C^1(\R\setminus\{0\}),
\qquad
f(0)=0.
\end{eqnarray}
Formally, the model \eqref{nld-point} corresponds to the Lagrangian density
\begin{eqnarray}\label{Lagrangian-density}
\psi^*\sigma_3(\jj\p_t-D_m)\psi+\delta(x)F(\psi^*\sigma_3\psi),
\qquad
F(s):=\int_0^s f(\tau)\,d\tau.
\end{eqnarray}
Let us give a formalized version of \eqref{nld-point}.
%% We set $H:=\jj\sigma_2\p_x+\sigma_3 m$.
%We will also use (in Section 5) the following notation for the projectors:
%\[
%\Pi_1=\begin{bmatrix}1&0\\0&0\end{bmatrix},
%\qquad
%\Pi_2=\begin{bmatrix}0&0\\0&1\end{bmatrix}.
%\]
%% This makes the Dirac operator real
%% and gives it a form closer to the usual Dirac operator in dimension $n=3$.
Denote
$\R_{-}=(-\infty,0)$,
$\R_{+}=(0,\infty)$,
and let
$H_{-}$ and $H_{+}$ be the free Dirac  operators on
$L^{2}(\R_{-})\otimes\C^2$ and $L^{2}(\R_{+})\otimes\C^2$,
formally given by $D_m$,
with domains
\[
\dom(H_{-})=H^{1}(\R_{-})\otimes\C^{2},
\qquad
\dom(H_{+})=H^{1}(\R_{+})\otimes\C^{2}.
\]
Denoting by $H_{\circ}$ the restriction of $D_m$ onto the domain
$\dom(H_{\circ}):=\{\psi\in H^{1}(\R,\C):\psi(0)=0\}$,
one has that $H_{\circ}$ is closed, symmetric, has defect indices $(2,2)$,
and  adjoint $H_{\circ}^{*}=H_{-}\oplus H_{+}$.
The family of selfadjoint extensions $H^{\mathrm{lin}}$ of the operator $H_{\circ}$
has been studied for a long time
\cite{gesztesy1987new,benvegnu1994relativistic,albeverio2005solvable};
we recall that the family is parametrized by the set of hermitian matrices $M$
and that any selfadjoint operator $H^{\mathrm{lin}}$ has the domain
\begin{equation}\label{domH_linear}
\dom(H^{\mathrm{lin}})  = 
\big\{
\psi \in L^2(\R,\C)\otimes \C^2:\psi \in H^1(\R\setminus\{0\}) \otimes \C^{2},\,
\jj\sigma_{2}[\psi ]_{0}-M\hat\psi=0 
\big\},
\end{equation}
where the two-component vector
\begin{eqnarray}\label{psi-hat}
\hat\psi:=(\psi(0^+)+\psi(0^-))/2
\end{eqnarray}
is the ``mean value'' of the spinor $\psi$ at $x=0$
and
\begin{eqnarray}\label{psi-zero}
[\psi]_{0}:=\psi(0^+)-\psi(0^-)
\end{eqnarray}
is the jump of the spinor $\psi$ at $x=0$.
We define a Dirac operator $H^{\mathrm{nl}}_{f}$ with concentrated nonlinearity
so that the coupling between the jump and the mean value of the spinor function
is given by a nonlinear relation (self-interaction); see~\cite{cacciapuoti2017one}.
To this aim we define the nonlinear domain 
\begin{equation}\label{domH}
\dom(H^{\mathrm{nl}}_{f})  = 
\big\{
\psi \in L^2(\R,\C)\otimes \C^2:\psi \in H^1(\R\setminus\{0\}) \otimes \C^{2},\,
\jj\sigma_{2}[\psi ]_{0}-f(\hat\psi\sp\ast\sigma_3\hat\psi)\sigma_3\hat\psi=0
\big\}.
\end{equation}
The operator
$H^{\mathrm{nl}}_{f}$ is then defined as the restriction
of $H^{*}_{\circ}=H_{-}\oplus H_{+}$ to
the domain $\dom(H^{\mathrm{nl}}_{f})$.
Thus, the Hamiltonian system
$\jj\p_t\psi=H^{\mathrm{nl}}_{f}\psi$,
with $\psi(t)\in \dom(H^{\mathrm{nl}}_{f})$,
is a formalized version of the Soler model with point interaction \eqref{nld-point}.
We will refer to the boundary condition defining the operator domain
$\dom(H^{\mathrm{nl}}_{f})$ from \eqref{domH}
as to the \emph{jump condition},
rewriting it in the form
\begin{equation}\label{jumpeq}
%\jj\sigma_2[\psi ]_{0}-f(\hat\psi\sp\ast\sigma_3\hat\psi)\sigma_3\hat\psi=0,
%\qquad
[\psi ]_{0}=f(\hat\psi\sp\ast\sigma_3\hat\psi)\sigma_1\hat\psi.
\end{equation}

%%\ac{Answering Referee Y, comment 15:}
Let us mention that the approach
similar to the one in the present article
would be applicable to
studying spectral stability of solitary waves
in the Soler model in one spatial dimension
with nonlinearity concentrated at several points
$x_1<x_2<\dots<x_N$.
At the same time, one expects that
if the nonlinearity is concentrated at several points,
then, besides usual
one-frequency and
bi-frequency solitary waves
(see Remark~\ref{remark-bi-frequency} below),
there could be nontrivial multifrequency solitary waves,
just like in the case of the nonlinear Klein--Gordon equation
with nonlinearity concentrated at several points;
see e.g. the example
constructed in \cite[Proposition~8.1]{MR2579377}.

\subsection{Solitary waves}
\label{sect-sw}

Below, we will use the following notations
for $\omega\in(-m,m)$:
%% \ac{
%%  Answering Referee Y, comment 1:
%%  added restriction $\omega\in(-m,m)$;
%%  decided not move formula
%%  to Lemma~\ref{lemma-sw}~\itref{lemma-sw-1}:
%%  elsewhere we consider $\omega=0$, too.
%% }
\begin{eqnarray}\label{def-def}
\varkappa(\omega)=\sqrt{m^2-\omega^2},
\qquad
\mu(\omega)=\sqrt{\frac{m-\omega}{m+\omega}},
\qquad
\omega\in(-m,m).
\end{eqnarray}

First
let us describe all
solitary waves to \eqref{nld-point},
which are defined as solutions of the form
\begin{eqnarray}\label{s-w}
\psi(t,x)=\phi(x)e^{-\jj\omega t},
\qquad
\phi\in
%%H^1(\R,\C^2),
\dom(H^{\mathrm{nl}}_{f}),
\quad
\omega\in\R,
\end{eqnarray}
where $\dom(H^{\mathrm{nl}}_{f})$ is defined in \eqref{domH}.

\begin{lemma}\label{lemma-sw}
\begin{enumerate}
\item
\label{lemma-sw-1}
There are no nonzero solitary waves
with $\omega\in\R\setminus(-m,m)$.
\item
\label{lemma-sw-2}
For
$\omega\in(-m,m)\setminus\{0\}$,
there are two types of solitary waves: the even-odd one,
\begin{eqnarray}\label{s-w-1}
\psi(t,x)
=
\begin{bmatrix}\alpha\\\alpha\mu(\omega) \sgn x\end{bmatrix}
e^{-\varkappa(\omega)\abs{x}}e^{-\jj\omega t},
\end{eqnarray}
where $\alpha\in\C$ satisfies the relation
\begin{equation}\label{Eq:ConditionAlpha}
f(\abs{\alpha}^2)=2\mu(\omega);
\end{equation}
and the odd-even one,
\begin{eqnarray}\label{s-w-2}
\psi(t,x)
=
\begin{bmatrix}\frac{\beta}{\mu(\omega)}\sgn x\\\beta\end{bmatrix}
e^{-\varkappa(\omega)\abs{x}}e^{-\jj\omega t},
\end{eqnarray}
where $\beta\in\C$
satisfies the relation
$f(-\abs{\beta}^2)=2/\mu(\omega)$.
\item
\label{lemma-sw-3}
For $\omega=0$, there are solitary waves of the form
\begin{eqnarray}\label{s-w-3}
\psi(x)
=
\begin{bmatrix}
\alpha+\beta\sgn x\\
\beta+\alpha\sgn x
\end{bmatrix}e^{-m\abs{x}},
\end{eqnarray}
with $\alpha,\,\beta\in\C$ satisfying the relation
$f(\abs{\alpha}^2-\abs{\beta}^2)=2$.
\end{enumerate}
\end{lemma}

%\ac{Answering Referee Y, comment 2:}
Thus, by Lemma~\ref{lemma-sw},
there are nonzero solitary wave solutions
if and only if $\range(f)\cap\R_{+}\ne\emptyset$.

\begin{proof}
The amplitude $\phi(x)$
of a solitary wave $\phi(x)e^{-\jj\omega t}$ is to satisfy,
formally,
\[
(D_m-\omega I_2-\delta(x)f\sigma_3)\phi=0,
\]
where $f=f(\hat\phi^*\sigma_3\hat\phi)$.
On $\R\setminus\{0\}$, one has:
%% \ac{Answering Referee Y, comment 3:
%%  ``{\it The proof of Lemma 2.1 is vaguely written.
%%   Rewrite formula (2.14) as a system of the
%%   first order equations with constant coefficients,
%%   find its eigenvalues and eigenvectors to
%%   arrive at (2.17).
%%   The same idea could be applied/mentioned in (2.37) and (3.4).}''
%%  We tried to add several details into the argument;
%%  keeping it in the vector form seems to make the
%%  argument shorter.
%% }
\begin{eqnarray}\label{m-o}
\jj\sigma_2\p_x\phi+m\sigma_3\phi-\omega\phi=0,
\qquad
\mbox{hence}\quad
\begin{cases}
\p_x\phi_2+(m-\omega)\phi_1=0,
\\
-\p_x\phi_1-(m+\omega)\phi_2=0,
\end{cases}
\qquad
x\in\R\setminus\{0\},
\end{eqnarray}
leading to
$(m^2-\omega^2-\p_x^2)\phi_1(x)=0$,
$(m^2-\omega^2-\p_x^2)\phi_2(x)=0$,
$x\ne 0$.
Thus, for $x\in\R_\pm$,
the amplitude $\phi(x)$
is given by $\phi\sb\pm(x)=\bm{v}\sb\pm e^{-\varkappa\sb\pm\abs{x}}$,
$\bm{v}\sb\pm\in\C^2$ and $\varkappa\sb\pm\in \R^\pm$,
which we write as
\begin{eqnarray}\label{phi-p-m}
\phi\sb\pm(x)=\begin{bmatrix}\alpha+b\sgn x\\\beta+a\sgn x\end{bmatrix}
e^{-\varkappa\sb\pm\abs{x}},
\qquad
x\in\mathbb{R}_\pm,
\qquad
\alpha,\,\beta,\,a,\,b\in\C.
\end{eqnarray}
Substituting \eqref{phi-p-m} into \eqref{m-o}
leads to the relations
(corresponding to $x>0$ and $x<0$)
\begin{eqnarray}\label{two-two}
\begin{bmatrix}
m-\omega&-\varkappa_{+}\\\varkappa_{+}&-m-\omega
\end{bmatrix}
\begin{bmatrix}\alpha+b\\a+\beta\end{bmatrix}=0,
\qquad
\begin{bmatrix}
m-\omega&-\varkappa_{-}\\\varkappa_{-}&-m-\omega
\end{bmatrix}
\begin{bmatrix}\alpha-b\\a-\beta\end{bmatrix}=0,
\end{eqnarray}
hence $\varkappa\sb\pm^2=m^2-\omega^2$;
we see that one needs to take
$\varkappa_\pm=\varkappa(\omega)=\sqrt{m^2-\omega^2}>0$,
and that one also needs to assume that
$\omega\in(-m,m)$ (or else the $L^2$-norm of $\phi$ is infinite
unless $\phi=0$).
Taking into account that, as the matter of fact, both matrices
in \eqref{two-two}
coincide, we arrive at
\[
\begin{bmatrix}
m-\omega&-\varkappa(\omega)\\\varkappa(\omega)&-m-\omega
\end{bmatrix}
\begin{bmatrix}\alpha\\a\end{bmatrix}=0,
\qquad
\begin{bmatrix}
m-\omega&-\varkappa(\omega)\\\varkappa(\omega)&-m-\omega
\end{bmatrix}
\begin{bmatrix}b\\\beta\end{bmatrix}=0.
\]
Hence,
$a=\frac{\varkappa(\omega)}{m+\omega}\alpha=\mu(\omega)\alpha$,
$\beta=\frac{\varkappa(\omega)}{m+\omega}b=\mu(\omega)b$,
with $\mu(\omega)$ from \eqref{def-def},
and now
\eqref{phi-p-m} takes the form
\begin{eqnarray}\label{phi-a-b}
\phi(x)=\begin{bmatrix}
\alpha+\frac{\beta}{\mu(\omega)}\sgn x
\\
\beta+\alpha\mu(\omega)\sgn x
\end{bmatrix}e^{-\varkappa(\omega)\abs{x}},
\qquad
x\in\R.
\end{eqnarray}
The jump condition
\eqref{jumpeq}
with $[\phi]_0=2\begin{bmatrix}{\beta}/{\mu(\omega)}\\\alpha\mu(\omega)\end{bmatrix}$
and
$\hat\phi=\begin{bmatrix}\alpha\\\beta\end{bmatrix}$
coming from \eqref{phi-a-b}
takes the form
\begin{eqnarray}\label{str}
2\begin{bmatrix}
{\beta}/{\mu(\omega)}\\\alpha\mu(\omega)
\end{bmatrix}
=f
\begin{bmatrix}
\beta\\\alpha
\end{bmatrix},
\end{eqnarray}
with
$f=f(\tau)$ evaluated at
\begin{eqnarray}\label{at-tau}
\tau:=\hat\phi\sp\ast\sigma_3\hat\phi=\abs{\alpha}^2-\abs{\beta}^2.
\end{eqnarray}
We conclude from \eqref{str} that
if $\mu\ne 1$ (that is, $\omega\ne 0$),
then either
$\beta=0$ and $2\mu(\omega)=f(\abs{\alpha}^2)$,
or $\alpha=0$ and $2/\mu(\omega)=f(-\abs{\beta}^2)$.
These two cases correspond to solutions
\eqref{s-w-1} and \eqref{s-w-2},
respectively.

If $\mu(\omega)=1$ (that is, $\omega=0$), then \eqref{str}
will be satisfied if and only if
$\alpha,\,\beta\in\C$ satisfy
$2=f(\abs{\alpha}^2-\abs{\beta}^2)$;
this corresponds to the solution \eqref{s-w-3}.
\end{proof}

%% \ac{Answering Referee Y, comment 4:
%%  moved the following remark to a suggested place.
%% }

\begin{remark}
If $\omega=m$ and $\omega=-m$,
then, besides solutions
of the form \eqref{phi-p-m},
equation \eqref{m-o}
has solutions
$\phi(x)=\begin{bmatrix}2m x\\-1\end{bmatrix}$
and
$\phi(x)=\begin{bmatrix}-1\\2m x\end{bmatrix}$,
respectively;
we do not consider them
since they do not belong to $L^2(\R)$.
\end{remark}

\begin{remark}\label{remark-bi-frequency}
Just like the standard Soler model
\cite{PhysRevD.1.2766},
equation \eqref{nld-point} has the
$\mathbf{SU}(1,1)$ symmetry:
if $\psi(t,x)$ is a solution, then so is
\[
(A+B\sigma_1\bm{K})\psi(t,x),
\]
where 
$A,\,B\in\C$ satisfy $\abs{A}^2-\abs{B}^2=1$
and $\bm{K}:\,\C^2\to\C^2$ is the complex conjugation.
In particular,
if $\phi(x) e^{-\jj\omega t}$
is a solitary wave solution to \eqref{nld-point}, then
there is also a bi-frequency solitary wave
\begin{eqnarray}\label{bi-frequency-waves}
A\phi(x) e^{-\jj\omega t}
+B\phi\sp{C}(x) e^{\jj\omega t}
\qquad
A,\,B\in\C,
\qquad
\abs{A}^2-\abs{B}^2=1,
\end{eqnarray}
with
$
\phi\sp{C}(x):=\sigma_1\bm{K}\phi(x)
$.
For more details on bi-frequency solitary waves,
see \cite{opus}.
\end{remark}

\begin{remark}\label{remark-omega-zero}
We note that the $\omega=0$ solitary waves
of the form \eqref{s-w-3},
with $\alpha,\,\beta\in\C$ satisfying $f(\abs{\alpha}^2-\abs{\beta}^2)=2$,
with $\abs{\alpha}>\abs{\beta}$,
can be written as
$(A+B\sigma_1\bm{K})
\begin{bmatrix}\alpha_0\\\alpha_0\sgn x\end{bmatrix}e^{-m\abs{x}}$,
with
$A=\alpha/\sqrt{\abs{\alpha}^2-\abs{\beta}^2}$,
$B=\beta/\sqrt{\abs{\alpha}^2-\abs{\beta}^2}$,
and $\alpha_0=\sqrt{\abs{\alpha}^2-\abs{\beta}^2}$
satisfying
$\abs{A}^2-\abs{B}^2=1$
and $f(\abs{\alpha_0}^2)=2$,
hence their stability properties
(both linear and nonlinear)
follow from the corresponding stability properties of
the solitary wave
$\begin{bmatrix}\alpha_0\\\alpha_0\sgn x\end{bmatrix}e^{-m\abs{x}}$
(that is, \eqref{s-w-1} with $\omega=0$).
Similarly, the stability of solitary waves
of the form \eqref{s-w-3}
in the case $\abs{\alpha}<\abs{\beta}$
can be reduced to the stability properties
of a solitary wave \eqref{s-w-2}
with $\omega=0$.
Note that
there are no solitary waves of the form
\eqref{s-w-3} with $\abs{\alpha}=\abs{\beta}$
since $f(0)=0$
(see \eqref{f-such})
while
for solitary wave \eqref{s-w-3}
one needs $f(\abs{\alpha}^2-\abs{\beta}^2)=2$.
\end{remark}

\begin{remark}\label{remark-even}
One can see from  \eqref{nld-point}
that if 
$\psi(t,x)$ is its solution,
then
$\psi\sp{C}(t,x)
=\sigma_1\bm{K}\psi(t,x)$
is a solution to \eqref{nld-point}
with the nonlinearity represented by
the function $\tilde f(\tau)=f(-\tau)$, $\tau\in\R$,
instead of $f$.
Consequently, if
\begin{eqnarray}\label{psi-1}
\psi(t,x)=\begin{bmatrix}
\frac{\beta}{\mu(\omega)}\sgn x\\\beta
\end{bmatrix}
e^{-\varkappa(\omega)\abs{x}}e^{-\jj\omega t},
\qquad
\beta\in\C,
\quad
f(-\abs{\beta}^2)=2/\mu(\omega),
\end{eqnarray}
is a solitary wave solution to
\eqref{nld-point},
then
\begin{eqnarray}\label{psi-1-c}
\psi\sp{C}(t,x)
=\sigma_1\bm{K}\psi(t,x)
=\begin{bmatrix}
\beta\\ \beta\mu(-\omega)\sgn x
\end{bmatrix}
e^{-\varkappa(-\omega)\abs{x}}e^{\jj\omega t},
\qquad
\beta\in\C,
\quad
\tilde f(\abs{\beta}^2)=2/\mu(\omega)=2\mu(-\omega),
\end{eqnarray}
is a solitary wave solution to
\eqref{nld-point}
with $\tilde f(\tau)=f(-\tau)$, $\tau\in\R$,
corresponding to the frequency $-\omega$,
and
the solitary waves \eqref{psi-1} and \eqref{psi-1-c}
have the same stability properties.
In particular, in the case when
the nonlinearity is represented by
the function $f(\tau)$ which is even,
if $\phi(x)e^{-\jj\omega t}$
from \eqref{s-w}
is a solitary wave solution, then so is
$\phi\sp{C}(x)e^{\jj\omega t}$,
with the same stability properties.
Therefore, it is enough to study properties
of solitary waves of the form
\eqref{s-w-1}.
\end{remark}

\subsection{Spectrum of the linearization operator}
\label{sect-spectrum}

In the present work,
we focus on stability of solitary waves
of the form \eqref{s-w-1}:
\begin{eqnarray}\label{solitary-wave}
\psi(t,x)=\phi_\omega(x)e^{-\jj\omega t},
\quad
\mbox{with}
\quad
\phi_\omega(x)=
\begin{bmatrix}\alpha\\\alpha\mu(\omega)\sgn x\end{bmatrix}e^{-\varkappa(\omega)\abs{x}},
\end{eqnarray}
with
$\varkappa(\omega)=\sqrt{m^2-\omega^2}$,
$\mu(\omega)=\sqrt{(m-\omega)/(m+\omega)}$
(see \eqref{def-def})
and with $\alpha\ne 0$ satisfying the
relation
\begin{eqnarray}\label{a-b-f}
f(\abs{\alpha}^2)=2\mu(\omega).
\end{eqnarray}
From now on, we assume that
\[
\alpha>0;
\]
due to
$\mathbf{U}(1)$-invariance of equation \eqref{nld-point},
there is no loss of generality in this assumption.

The spectral stability of solitary waves of the form
\eqref{s-w-2}
is obtained in the same way
(one can use Remark~\ref{remark-even});
for definiteness, we restrict our attention to the solitary
waves of the form \eqref{s-w-1}.

The spectral stability of solitary waves of the form
\eqref{s-w-3}
follows from Remark~\ref{remark-bi-frequency}.

For future use, let us mention that
for a family of solitary waves with $\omega$ from an interval of $(-m,m)$
the relation \eqref{a-b-f} allows us to consider $\alpha>0$
locally as a function of $\omega$ and
to compute $\p_\omega\alpha(\omega)$ when $f$ is $C^1$
and its derivative does not vanish
at a particular value of $\alpha^2$:
\begin{eqnarray}\label{p-alpha}
f'(\alpha^2)\alpha\p_\omega\alpha
=\p_\omega\mu(\omega)
=-\frac{m}{(m+\omega)\varkappa(\omega)}.
\end{eqnarray}

%% \ac{Answering Referee Y, comment 5:
%%  above, replaced $\varkappa$ by $\varkappa(\omega)$.
%% }

Let us consider the spectrum of the operator
corresponding to the linearization at
the solitary wave $\phi_\omega(x)e^{-\jj\omega t}$
from \eqref{solitary-wave}.
We use the Ansatz
\begin{eqnarray}\label{Ansatz-br-0}
\psi(t,x)=(\phi_\omega(x)+r(t,x)+\jj s(t,x))e^{-\jj\omega t},
\qquad
\big(r(t,x),\,s(t,x)\big)\in \R^2\times\R^2.
\end{eqnarray}
A substitution of the Ansatz
\eqref{Ansatz-br-0}
into equation \eqref{nld-point} shows that the perturbation
$(r(t,x),\,s(t,x))$
satisfies in the first order the following system:
\begin{eqnarray}\label{l-p-l-m}
\begin{cases}
-\dot s=(D_m-\omega) r
-\delta(x)f\sigma_3\, r
-\delta(x)(\phi_\omega\sp\ast\sigma_3\, r)2g\sigma_3\,\phi_\omega
,
\\[1ex]
\dot r=(D_m -\omega) s
-\delta(x)f\sigma_3\, s
.
\end{cases}
\end{eqnarray}
Above,
$D_m$ is from \eqref{def-dm}
and $f,\,g\in\R$ are given by
\[
f=f(\alpha^2),
\qquad
g=f'(\alpha^2).
\]
In \eqref{l-p-l-m},
we abuse notation considering $\delta$
as an operator acting on $L^2(\R)$.
We refer to Remark~\ref{Rem:RigorousNotation} below
for a more rigorous formulation.
%%Making use of the assumption \eqref{f-p},
We define
\begin{eqnarray}\label{def-kappa}
\kappa=\alpha^2 f'(\alpha^2)/f(\alpha^2)\in\R.
\end{eqnarray}
By \eqref{Eq:ConditionAlpha}, $f(\alpha^2)=2\mu>0$.
Note that the definition \eqref{def-kappa} is compatible
with the pure power case,
\begin{eqnarray}\label{pure-power}
f(\tau)=\abs{\tau}^\kappa,
\qquad
\tau\in\R,
\qquad
\kappa>0.
\end{eqnarray}
Combining \eqref{a-b-f}, \eqref{p-alpha},
and the definition \eqref{def-kappa},
we can express
\begin{eqnarray}\label{p-alpha-kappa}
2\kappa\mu\p_\omega\alpha=\alpha\p_\omega\mu.
\end{eqnarray}

\begin{remark}
Nonlinearities giving rise to $\kappa<0$ are not covered by the well-posedness results (\cite{cacciapuoti2017one}); in this section for completeness we give the analysis of the linearization operator for any value of $\kappa$,
which makes our results applicable not only to the pure power case
\eqref{pure-power}
but also to a generic nonlinearity
$f\in C(\R)\cap C^1(\R\setminus\{0\})$, $f(0)=0$.
\end{remark}

Using the relation \eqref{a-b-f} and the definition \eqref{def-kappa},
we simplify the system \eqref{l-p-l-m} to
\begin{eqnarray}\label{l-p-l-m-1}
\begin{cases}
-\dot s=\big(D_m-\omega I_2
-2\mu\delta(x)\sigma_3
-4\mu\kappa\delta(x)\Pi_1\big) r
=:L_{+}r,
\\[1ex]
\dot r=\big(D_m -\omega I_2
-2\mu\delta(x)\sigma_3\big) s
=:L_{-}s,
\end{cases}
\end{eqnarray}
with $I_2$ the identity matrix in $\C^2$
and with $\Pi_i$, $i=1,\,2$, defined by
\begin{eqnarray}\label{def-pi1-pi2}
\Pi_1=\begin{bmatrix}1&0\\0&0\end{bmatrix},
\qquad
\Pi_2=\begin{bmatrix}0&0\\0&1\end{bmatrix}.
\end{eqnarray}
%$\Pi_1=\begin{bmatrix}1&0\\0&0\end{bmatrix}$.
%% We used the relation \eqref{a-b-f} and
%% definition \eqref{def-kappa}.
In the matrix form,
the linearized system
\eqref{l-p-l-m-1}
can be written as
\begin{eqnarray}\label{nld-point-linearization}
\p_t
\begin{bmatrix}r(t,x)\\s(t,x)\end{bmatrix}
=\bfA
\begin{bmatrix}r(t,x)\\s(t,x)\end{bmatrix},
\qquad
\bfA=\begin{bmatrix}
0&L_{-}\\-L_{+}&0
\end{bmatrix},
\end{eqnarray}
where the operator $\bfA$
is given explicitly by
\begin{eqnarray}\label{def-a-ch2}
\bfA
=
\begin{bmatrix}
0&D_m-\omega I_2-2\mu(\omega)\delta(x)\sigma_3
\\-D_m+\omega I_2
+2\mu(\omega)\delta(x)\sigma_3
+4\mu(\omega)\kappa\delta(x)\Pi_1
&0
\end{bmatrix}.
\end{eqnarray}
We consider $\bfA$ as an operator-valued function
\[
\bfA=\bfA(\omega,\kappa),
\qquad
\omega\in(-m,m),
\quad
\kappa\in\R.
\]

\begin{remark}\label{Rem:RigorousNotation}
More precisely, one
considers $L_\pm$ as singular perturbations of the Dirac operator acting as $D_m-\omega I_2$
on $H^1(\R\setminus\{0\},\C^2)$
and with domains
\begin{equation}
\label{Eq:DomainOfL_+}
\dom(L_+)=
\left\{
r\in
H^1(\R\setminus\{0\},\C^2):
\,
\jj\sigma_2[r]_0-2\mu(\sigma_3+2\kappa\Pi_1)\hat r=0
\right\},
\end{equation}
with
$\hat r=\big(r(0^{+})+r(0^{-})\big)/2$,
$[r]_0=r(0^{+})- r(0^{-})$,
and
\begin{equation}
\label{Eq:DomainOfL_-}
\dom(L_{-})=
\left\{
s\in
H^1(\R\setminus\{0\},\C^2):
\,
\jj\sigma_2[s]_0
-2\mu\sigma_3\hat s=0
\right\},
\end{equation}
with
$\hat s=\big(s(0^{+})+s(0^{-})\big)/2$,
$[s]_0=s(0^{+})- s(0^{-})$.
%% \dn{Answering Referee X:
%%   ``{\it Just one suggestion -- it would be good to include
%%   some justification for the selfadjointness for $L(\omega,\kappa)$,
%%   before the statement of Lemma 2.7. Specifically, it will
%%   be helpful to emphasize the essentially unique choice of BC,
%%   which makes these symmetric.}''
%% }
The boundary conditions in \eqref{Eq:DomainOfL_+}  and \eqref{Eq:DomainOfL_-} belong to the class assuring selfadjointess of respective operators $L_+$ and $L_-$ (see equation \eqref{domH_linear}).
Correspondingly, one needs to consider $\bfA$ as the operator
acting as
$
 \begin{bmatrix}
0&D_m-\omega I_2
\\-D_m+\omega I_2
&0
\end{bmatrix}
$
on $H^1(\R\setminus\{0\},\C^2\times\C^2)$
and with domain
\begin{equation}
\label{Eq:DomainOfA}
\dom(\bfA)=
\left\{
(r,s)\in
H^1(\R\setminus\{0\},\C^2\times\C^2):
\ \jj\sigma_2[r]_0
-2\mu(\sigma_3+2\kappa\Pi_1)\hat r=0,
\ \,\,\jj\sigma_2 [s]_0-2\mu\sigma_3\hat s=0  \right\},
\end{equation}
with
$\hat r$, $[r]_0$, $\hat s$, and $[s]_0$ as before.
\end{remark}

%% \ac{Answering Referee Y, comment 7
%%  (``{\it In formula (2.34), $r$ and $s$
%%  should be $L^2$-integrable over $\R$
%%  }'')
%%  [note that \eqref{Eq:DomainOfA} was (2.34) before]
%%  There could be some confusion:
%%  since
%%  $(r,s)\in H^1(\R\setminus\{0\},...)$,
%%  we automatically know that $(r,s)\subset L^2(\R,...)$;
%%  at the same time, having *only* $(r,s)\in L^2(\R)$
%%  would not be enough since
%%  we need $D_m$ to be able to act on them.
%% }

Before we formulate the results,
let us mention that a \emph{virtual level}
can be defined as a limit point of an eigenvalue family (dependent on a parameter)
when this limit point
belongs to the essential spectrum
but not necessarily to the point spectrum
(that is, it does not necessarily correspond
to a square-integrable eigenfunction).
The virtual levels usually occur
at thresholds of the essential spectrum
(the endpoints of the essential spectrum or
the points where the continuous spectrum
changes its multiplicity),
when they are referred to as \emph{threshold resonances}.
For more on the phenomenon of virtual levels,
see e.g. \cite{MR544248,MR1841744,MR2598115,erdogan2019dispersive}.
The general theory
of virtual levels of operators in Banach spaces
is developed in \cite{virtual-levels,virtual-levels-review}.

\medskip

We start with the spectra of $L_\pm$.
We consider the closed densely defined operator
\begin{eqnarray}\label{l-kappa}
L(\omega,\kappa)
=
D_m-\omega I_2-2\mu\delta(x)\sigma_3-4\mu\kappa\delta(x)\Pi_1,
\qquad
\dom{(L(\omega,\kappa))}=\dom{(L_+)},
\end{eqnarray}
with $\dom{(L(\omega,\kappa))}=\dom{(L_+)}$ given by
\eqref{Eq:DomainOfL_+};
we note that one has
$L_{-}=L(\omega,0)$,
$L_{+}=L(\omega,\kappa)$.
\\
Denote
\begin{eqnarray}\label{def-x-e-o-o-e}
\begin{array}{l}
\bfX_{\mbox{\rm\footnotesize even-odd}}
=L^2_{\mbox{\rm\footnotesize even}}(\R,\C)
\times L^2_{\mbox{\rm\footnotesize odd}}(\R,\C)\subset L^2(\R,\C^2),
\\[2ex]
\bfX_{\mbox{\rm\footnotesize odd-even}}
=L^2_{\mbox{\rm\footnotesize odd}}(\R,\C)
\times L^2_{\mbox{\rm\footnotesize even}}(\R,\C)\subset L^2(\R,\C^2),
\end{array}
\end{eqnarray}
with
$L^2_{\mbox{\rm\footnotesize odd}}$
and
$L^2_{\mbox{\rm\footnotesize even}}$
the subspaces of $L^2$ consisting of odd and even functions of $x\in\R$,
respectively.
We note that
$L^2(\R,\C^2)
=
\bfX_{\mbox{\rm\footnotesize odd-even}}\oplus
\bfX_{\mbox{\rm\footnotesize even-odd}}$
and also that
$\bfX_{\mbox{\rm\footnotesize odd-even}}$
and $\bfX_{\mbox{\rm\footnotesize even-odd}}$
are invariant subspaces for the operator $L(\omega,\kappa)$,
so it suffices to study the spectra of the restrictions of $L(\omega,\kappa)$ onto these subspaces.

%% \ac{Answering Referee Y, comment 8:
%%  We added the following remark
%%  on the invariant subspaces.
%% }

\begin{remark}
\label{remark-inv}
Let us provide some detail
why 
$\bfX_{\mbox{\rm\footnotesize even-odd}}$ and
$\bfX_{\mbox{\rm\footnotesize odd-even}}$
are invariant subspaces of $L(\omega,\kappa)$.
Let us consider the parity operator
\[
\calP:\;L^2(\R,\C^2)\to L^2(\R,\C^2),
\qquad
\calP\Psi(x)=\sigma_3\Psi(-x).
\]
The operator $\calP$ commutes with $D_m-\omega I_2$ on its domain
$\dom(D_m-\omega I_2)=H^1(\R,\C^2)$.
It also commutes with its singular perturbation given by
$L=L(\omega,\kappa)$ from \eqref{l-kappa},
or, equivalently, $L\calP=\calP L$.
The identity clearly holds when localized to $\R_-\cup\R_+$,
when considering the action of the operator outside of the origin.
Proving the identity on the full operator domain amounts
to proving that $\calP(\dom(L))\subset \dom(L)$;
we claim that this holds true.
Indeed, let $\Psi\in \dom(L)$;
thus,
we assume that $\Psi\in H^1(\R\setminus\{0\},\C^2)$ and that
$\jj\sigma_2[\Psi]_0-2\mu(\sigma_3+2\kappa\Pi_1)\hat\Psi=0$.
For $\calP\Psi$,
taking into account that
$\left[\calP\Psi \right]_0=-\sigma_3\left[\Psi\right]_0$
and
$\widehat{\calP\Psi}=\sigma_3\hat{\Psi}$
(in the notations from \eqref{psi-hat} and \eqref{psi-zero}),
%Replacing in the boundary condition and using
%$\sigma_2\sigma_3=-\sigma_3\sigma_2$ and $\Pi_1\sigma_3=\Pi_1$,
one derives:
\[
\jj\sigma_2\left[\calP\Psi\right]_0-2\mu(\sigma_3+2\kappa\Pi_1)
\widehat{\calP\Psi}
=\sigma_3\big(\jj\sigma_2\left[\Psi\right]_0-2\mu(\sigma_3+2\kappa\Pi_1)\hat\Psi\big)=0,\]
that is, $\calP\Psi\in \dom(L)$.
(We took into account the relations
$
%%\{\sigma_2,\sigma_3\}:=
\sigma_2\sigma_3+\sigma_3\sigma_2=0$ and
$%%[\Pi_1,\sigma_3]:=
\Pi_1\sigma_3-\sigma_3\Pi_1=0$).
It follows that
$L$ is invariant in the eigenspaces
corresponding to eigenvalues $\pm 1$ of the operator $\calP$.
These eigenspaces are given by
\[
\bfX_\pm=2^{-1}(I\pm\calP)L^2(\R,\C^2),
\]
so that
$\bfX_+=\bfX_{\mbox{\rm\footnotesize even-odd}}$
and
$\bfX_-=\bfX_{\mbox{\rm\footnotesize odd-even}}$;
the domains of $L$ restricted to these subspaces are given by
\[
\dom(L\at{\bfX_\pm})
=\{
\Psi\in 2^{-1}(I\pm\calP)H^1(\R\setminus\{0\},\C^2):\,
\jj\sigma_2[\Psi]_0-2\mu(\sigma_3+2\kappa\Pi_1)\hat\Psi=0
\}
=\dom(L)\cap\bfX_\pm.
\]
%% \ac{PERHAPS NOT NEEDED FOR THE ARTICLE:
%%  Notice that if we consider the operator
%%  $L_-= D_m-\omega I_2-f\delta(x)(\sigma_3+\epsilon\sigma_1)$
%%  appearing in the linearization of equation
%%  \eqref{nld-point-perturbed-2}
%%  in Section~\ref{sect-nld-perturbation-broken-parity},
%%  we face with an example where
%%  $\calP(\dom(L_-))\nsubseteq\dom(L_-)$
%%  (and also $\calP(\dom(L_+))\nsubseteq\dom(L_+)$).
%%  More precisely, the operator domain is
%%  $\dom(L_{-})=
%%  \left\{s\in H^1(\R\setminus\{0\},\C^2):
%%  \,
%%  \jj\sigma_2[s]_0
%%  =2f(\sigma_3+\epsilon\sigma_1)\hat s \right\}$,
%%  and being $\sigma_1\sigma_3=-\sigma_3\sigma_1$
%%  commutation and invariance of the even-odd and
%%  odd-even subspaces fail.
%% }
%% \ac{PERHAPS NOT NEEDED FOR THE ARTICLE:
%%  \dn{
%%  One has:
%%  $
%%  \dom(L_\pm)=\left\{ r\in H^1(\R\setminus\{0\},\C^2):
%%  \,
%%  \jj\sigma_2[r]_0=M\hat r  \right\}$,
%%  where $M_\pm$
%%  are  certain hermitian matrices (to get selfadjointness;
%%  and this could answer in a more complete way to question 8
%%  of Referee Y).
%%  When $M_\pm$ commute with $\sigma_3$,
%%  the operators $L_\pm$ are parity preserving,
%%  and then even-odd, odd-even subspaces are invariant.
%%  This should extend more or less automatically
%%  to the vector linearization $\mathbf{J}\mathbf{L}$.
%% }}
\end{remark}

\begin{lemma}\label{lemma-l}
Let $\omega\in(-m,m)$, $\kappa\in\R$.
\begin{enumerate}
\item
\label{lemma-l-1}
$\sigma_{\mathrm{ess}}\big(L(\omega,\kappa)\big)
=\R\setminus(-m-\omega,m-\omega)$,
$
\ \sigma_{\mathrm{ess}}\big(L(\omega,\kappa)\big)
\cap
\sigma_{\mathrm{p}}\big(L(\omega,\kappa)\big)=\emptyset$\textup;
\item
\label{lemma-l-2}
$
\sigma_{\mathrm{p}}\big(
L(\omega,\kappa)\at{\bfX_{\mbox{\rm\tiny{odd-even}}}}
\big)
=\{-2\omega\},
$
and eigenvalue $\lambda=-2\omega$ is
of geometric multiplicity one\textup;
\item
\label{lemma-l-3}
If $\kappa>-1/2$, then
$\sigma_{\mathrm{p}}\big(
L(\omega,\kappa)\at{\bfX_{\mbox{\rm\tiny{even-odd}}}}
\big)
=
\big\{Z(\omega,\kappa)\big\},
$
where the eigenvalue
\[
Z(\omega,\kappa)
=
-4(m-\omega)
\frac{\kappa(\kappa+1)}{1+(1+2\kappa)^2\mu^2}\in(-m-\omega,m-\omega)
\]
is of geometric multiplicity one.
If
$\kappa\leq- 1/2$,
then
$\sigma_{\mathrm{p}}\big(
L(\omega,\kappa)\at{\bfX_{\mbox{\rm\tiny{even-odd}}}}
\big)
=\varnothing$.
\end{enumerate}
\end{lemma}

%% \ac{Answering Referee Y, comment 9:
%% Indeed there was our mistake,
%% the erroneous statement about the symmetry
%% with respect to $\kappa=-1/2$ is now excluded,
%% Part 3 of the Lemma is corrected.
%% }

\begin{remark}\label{Rem:OnZ}
We note
that
for each $\omega\in(-m,m)$,
one has
$\p_\kappa Z(\omega,\kappa)
=-4(m-\omega)(1+2\kappa)(1+\mu^2)/(1+(1+2\kappa)^2\mu^2)^2<0$ for
$\kappa>-1/2$,
$Z(\omega,\kappa)\to m-\omega-0$ when $\kappa\to -1/2+0$
(corresponding to a virtual level at the threshold
$\lambda=m-\omega$ when $\kappa=-1/2$),
$Z(\omega,0)=0$,
and 
$Z(\omega,\kappa)\to -m-\omega+0$ when
$\kappa\to +\infty$.
\end{remark}

\begin{proof}
The conclusion about the essential spectrum is standard:
since $L(\omega,\kappa)$ is selfadjoint,
its essential spectrum can be characterized by
the Weyl sequences,
which 
do not depend on the jump condition in~\eqref{Eq:DomainOfA}.
%% \ac{Answering Referee Y, comment 10:
%% }
Let us provide more detail.
We recall that for a closed operator $A$
in the Hilbert space $\bfH$, with domain $\dom(A)$,
a sequence $\psi_j\in\bfH$
is called a Weyl sequence (or a singular sequence)
corresponding to
$\lambda\in\C$
if $\psi_j\in\dom(A)$,
$\norm{\psi_j}=1$,
$\psi_j\to 0$ weakly in $\bfH$,
$(A-\lambda I_\bfH)\psi_j\to 0$.
We also recall that
the essential spectrum
$\sigma\sb{\mathrm{ess}}(A)$
of a selfadjoint operator $A$,
can be characterized as the set of values $\lambda\in\C$
for which there are Weyl sequences
(see e.g. \cite{edmunds2018spectral}).
Given the Weyl sequence $\psi_j\in L^2(\R,\C^2)$ for
$L(\omega,\kappa)$ corresponding to some $\lambda\in\C$,
we can assume that all $\psi_j$
have supports outside of $x=0$.
Indeed, let us
fix $\rho\in C^\infty\sb{\mathrm{comp}}(\R)$,
$\rho\at{[-1,1]}=1$.
Since $(\psi_j)_{j\in\N}$ is bounded in $H^1(\R\setminus\{0\},\C^2)$,
the sequence
$(\rho\psi_j)_{j\in\N}$ is compact in $L^2(\R,\C^2)$
and converges weakly
and thus strongly to $0$ in $L^2(\R,\C^2)$
and hence also
in $H^1(\R\setminus\{0\},\C^2)$.
Consequently, we can substitute
the sequence
$\psi_j(x)$, $j\in\N$, by $\tilde\psi_j(x)
=(1-\rho(x))\psi_j(x)/\norm{(1-\rho)\psi_j}$;
we drop finitely many terms $\psi_j$
with $\supp\psi_j\subset[-1,1]$.
Therefore, Weyl sequences for $L(\omega,\kappa)$
also yield the Weyl sequences for
$D_m-\omega I_2$ (and vice versa),
resulting in the same essential spectrum
$\sigma\sb{\mathrm{ess}}(L(\omega,\kappa))$:
\[
\sigma\sb{\mathrm{ess}}(L(\omega,\kappa))
=\sigma\sb{\mathrm{ess}}(D_m-\omega I_2)
=\R\setminus(-m-\omega,m-\omega),
\qquad
\forall\omega\in(-m,m),
\quad
\forall\kappa\in\R.
\]

The conclusion on absence of embedded eigenvalues
will follow from Parts~\itref{lemma-l-2} and~\itref{lemma-l-3}.

Let us now prove Part~\itref{lemma-l-2}.
To find the point spectrum of the restriction of $L(\omega,\kappa)$
onto
$\bfX_{\mbox{\rm\footnotesize odd-even}}$,
we need to consider the spectral problem
\[
\left(
\begin{bmatrix}
m-\omega-\lambda&\p_x\\-\p_x&-m-\omega-\lambda
\end{bmatrix}
-2\mu\delta(x)\sigma_3
-4\mu\kappa\delta(x)\Pi_1
\right)
\psi(x)=0,
\qquad
x\in\R,
\]
with
$
\psi(x)=
\begin{bmatrix}
c_1\sgn x\\c_2
\end{bmatrix}e^{-\gamma\abs{x}},
$
$x\in\R$,
and with
\begin{eqnarray}\label{here-nu-a-b}
\gamma=\sqrt{(m-\omega-\lambda)(m+\omega+\lambda)},
\qquad
c_1=m+\omega+\lambda,
\qquad
c_2=\gamma
\end{eqnarray}
(the  values of $c_1$ and $c_2$ are defined up to a nonzero coefficient),
and the jump condition
\begin{eqnarray}\label{ttc}
-2c_1+2\mu c_2=0.
\end{eqnarray}
We note that
since $\psi$ is square-integrable, we need $\Re\gamma>0$
and so
\begin{eqnarray}\label{spectral-gap}
-m-\omega<\lambda<m-\omega.
\end{eqnarray}
The relation \eqref{ttc} takes the form
$
-(m+\omega+\lambda)+\mu
\sqrt{(m-\omega-\lambda)
(m+\omega+\lambda)}
=0;
$
it is satisfied only for $\lambda=-2\omega$.
We note that \eqref{ttc} implies that
the geometric multiplicity of eigenvalue $\lambda=0$ equals one;
its algebraic multiplicity also equals one
since $L(\omega,\kappa)$ is selfadjoint.

Let us prove Part~\itref{lemma-l-3}.
For the spectrum of the restriction of $L(\omega,\kappa)$
onto
$\bfX_{\mbox{\rm\footnotesize even-odd}}$,
we consider the spectral problem
\[
\left(
\begin{bmatrix}
m-\omega-\lambda&\p_x\\-\p_x&-m-\omega-\lambda
\end{bmatrix}
-2\mu\delta(x)\sigma_3
-4\mu\kappa\delta(x)\Pi_1
\right)\psi=0,
\qquad
x\in\R,
\]
with
$\psi(x)=
\begin{bmatrix}
c_1\\c_2\sgn x
\end{bmatrix}e^{-\gamma\abs{x}}$,
$x\in\R$.
This again leads to the values \eqref{here-nu-a-b}.
Substituting these values
into the jump condition
$
2c_2-2(1+2\kappa)\mu c_1=0
$
results in 
\begin{eqnarray}\label{sdf}
\sqrt{(m-\omega-\lambda)(m+\omega+\lambda)}
=(1+2\kappa)\mu(m+\omega+\lambda).
\end{eqnarray}
By \eqref{spectral-gap}, $\lambda>-m-\omega$,
hence
the above relation can hold only for $1+2\kappa>0$,
that is, for $\kappa>-1/2$;
squaring \eqref{sdf},
we arrive at
\[
m-\omega-\lambda
=(1+2\kappa)^2\mu^2(m+\omega+\lambda),
\]
which leads to
\[
\lambda
=m\frac{1-(1+2\kappa)^2\mu^2}{1+(1+2\kappa)^2\mu^2}-\omega
%%=(m-\omega)\frac{1-(1+2\kappa)^2}{1+(1+2\kappa)^2\mu^2}
=-4(m-\omega)\frac{\kappa(1+\kappa)}{1+(1+2\kappa)^2\mu^2}.
% \lambda=\frac{2m}{1+(1+2\kappa)^2\mu^2}-m-\omega.
\qedhere
\]
\end{proof}

We now consider the operator $\bfA$ from \eqref{def-a-ch2}
as an operator-valued function of $\omega\in(-m,m)$ and $\kappa\in\R$. In the following theorem the point spectrum and virtual levels of $\bfA(\omega,\kappa)$ are fully and explicitly described.
Similarly to
\eqref{def-x-e-o-o-e},
we introduce the subspaces
\begin{eqnarray}\label{def-oeoe}
\bfX_{\mbox{\rm\footnotesize odd-even-odd-even}}
=
L^2\sb{\mbox{\rm\footnotesize odd}}(\R,\C)
\times
L^2\sb{\mbox{\rm\footnotesize even}}(\R,\C)
\times
L^2\sb{\mbox{\rm\footnotesize odd}}(\R,\C)
\times
L^2\sb{\mbox{\rm\footnotesize even}}(\R,\C)
\subset L^2(\R,\C^4),
\\[1ex]
\label{def-eoeo}
\bfX_{\mbox{\rm\footnotesize even-odd-even-odd}}
=
L^2\sb{\mbox{\rm\footnotesize even}}(\R,\C)
\times
L^2\sb{\mbox{\rm\footnotesize odd}}(\R,\C)
\times
L^2\sb{\mbox{\rm\footnotesize even}}(\R,\C)
\times
L^2\sb{\mbox{\rm\footnotesize odd}}(\R,\C)
\subset L^2(\R,\C^4).
\end{eqnarray}
These are invariant subspaces for $\bfA(\omega,\kappa)$
%\ac{Answering Referee Y, comment 11:}
(the argument is
exactly the same as in Remark~\ref{remark-inv};
now as the parity operator one takes
$\mathscr{P}:\,L^2(\R,\C^4)\to L^2(\R,\C^4)$,
$\mathscr{P}\Psi(x)=\varSigma\Psi(-x)$,
where $\varSigma=\mathop{\mathrm{diag}}[1,-1,1,-1]$),
and there is a decomposition of
$
\bfX:=L^2(\R,\C^4)
$
into a direct sum
\begin{eqnarray}\label{def-x-x}
\bfX
=\bfX_{\mbox{\rm\footnotesize odd-even-odd-even}}
\oplus
\bfX_{\mbox{\rm\footnotesize even-odd-even-odd}}.
\end{eqnarray}
The spectral analysis
of $\bfA(\omega,\kappa)$ can be done
restricting the operator to these two subspaces.

%% \ac{Answering Referee Y, comment 12:
%% Moved  Part~(7) of Theorem 2.9
%% (the conclusion about the essential spectrum)
%% to the beginning of the theorem (after Part~(1));
%% Reformulated the description of the behavior of eigenvalues
%% as suggested by the referee;
%% answered the questions
%% about $\omega<\mathcal{T}_\kappa^{+}$
%% and multiplicities of additional eigenvalues.
%% }

\begin{theorem}\label{theorem-sect2}
Let $\omega\in(-m,m)$ and $\kappa\in\R$.
\begin{enumerate}
\item
\label{theorem-sect2-1}
The spectrum of $\bfA(\omega,\kappa)$
is symmetric with respect to $\R$ and $\jj\R$:
\begin{eqnarray}\label{symmetry}
\lambda\in\sigma\sb{\mathrm{p}}(\bfA(\omega,\kappa))
\qquad
\Leftrightarrow
\qquad
-\lambda\in\sigma\sb{\mathrm{p}}(\bfA(\omega,\kappa))
\qquad
\Leftrightarrow
\qquad
\bar\lambda\in\sigma\sb{\mathrm{p}}(\bfA(\omega,\kappa)).
\end{eqnarray}
\item
\label{theorem-sect2-7}
\[
%\ds
\sigma\sb{\mathrm{ess}}(\bfA(\omega,\kappa))
=
\begin{cases}
\jj\big(\R\setminus(-m+\abs{\omega},m-\abs{\omega})\big),
\quad
&(\omega,\kappa)\ne (0,-1)\textup;
\\
\C,
&(\omega,\kappa)=(0,-1).
\end{cases}
\]
\item
\label{theorem-sect2-5}
For all $\omega\in(-m,m)$ and $\kappa\in\R$, one has
$
0\in\sigma\sb{\mathrm{p}}(\bfA(\omega,\kappa))$.
Moreover,
\[
\dim\ker(\bfA(\omega,\kappa))
=
\begin{cases}
1,&
\omega\ne 0,\ \kappa\neq0,
\\
2,&
\omega\ne 0,\ \kappa=0,\\
3, &\omega=0,\, \kappa\ne 0,
\\
4,&
\omega=0,\ \kappa=0.
\end{cases}
\]
\item
\label{theorem-sect2-6}
Denote
\begin{eqnarray}\label{def-Omega-kappa}
\Omega_\kappa=\frac{\kappa+1}{2\kappa}m,
\qquad
\kappa\in\R\setminus\{0\}\textup;
\qquad
\mbox{
$\abs{\Omega_\kappa}<m$ as long as
$\kappa\in\R\setminus[-1/3,1]$.}
\end{eqnarray}
The algebraic multiplicity of $\lambda=0$ is given by
\begin{eqnarray}
\dim
\mathfrak{L}(\bfA(\omega,\kappa))
=\begin{cases}
2,
%=0+2
&\omega\in(-m,m)\setminus\{0,\Omega_\kappa\},
\quad
\kappa\in\R\setminus[-1/3,1]\textup;
\\
2,
%=0+2
&\omega\in(-m,m)\setminus\{0\},
\quad
\kappa\in[-1/3,1]\textup;
\\
4,
%=0+4
&\omega
=\Omega_\kappa,
\quad
\kappa\in\R\setminus(
\{-1\}\cup[-1/3,1])\textup;
\\
2,
&
\omega=0,\quad \kappa\not\in\{-1,\,0\}\textup;
\\
4,
&\omega=0,\quad\kappa=0\textup;
\\
6,
&\omega=0,\quad\kappa=-1.
\end{cases}
\end{eqnarray}
Above, $\mathfrak{L}(\bfA(\omega,\kappa))$
is the notation for the generalized eigenspace
of $\bfA(\omega,\kappa)$ corresponding to $\lambda=0$.
\item
\label{theorem-sect2-2}
For all $\omega\in(-m,m)$ and $\kappa\in\R$,
one has
\[
\pm 2\omega\jj\in\sigma\sb{\mathrm{p}}(\bfA(\omega,\kappa)\at{\bfX_{\mbox{\rm\tiny odd-even-odd-even}}})\subset\sigma\sb{\mathrm{p}}(\bfA(\omega,\kappa)).
\]
For $\omega\neq 0$,
%\textup;
eigenvalues $\pm 2\omega\jj$
of $\bfA(\omega,\kappa)\at{\bfX_{\mbox{\rm\tiny odd-even-odd-even}}}$
are of geometric multiplicity one;
they are embedded into the essential spectrum
for $\abs{\omega}\ge m/3$
and they are isolated eigenvalues
of algebraic multiplicity one
for $\omega\in(-m/3,m/3)\setminus\{0\}$.
For $\omega=0$,
these eigenvalues
are of geometric and algebraic multiplicity two.
(The invariant space $\bfX_{\mbox{\rm\tiny odd-even-odd-even}}$
of $\bfA(\omega,\kappa)$
was defined in \eqref{def-oeoe}.)
\item
\label{theorem-sect2-3}
For $\omega\ne 0$,
the virtual levels
at the thresholds
$\lambda=\pm(m-\abs{\omega})\jj$
occur
for
$\kappa<-1$, $-1<\kappa<2^{-1/2}-1$, and $\kappa>2^{-1/2}$
at
$\omega=
\mathcal{T}_\kappa$, where
\[
\mathcal{T}_\kappa=
\begin{cases}
\displaystyle
\mathcal{T}_\kappa^{-}:=\frac{(\kappa+1)^2}{(\kappa+2)\kappa}m,
&\quad
\kappa\in(-1,2^{-1/2}-1)\textup;
%%\qquad   \ \ %\lambda=\pm(m+{\omega})\jj
\\[1.5ex]
\displaystyle
\mathcal{T}_\kappa^{+}:=\frac{(\kappa+1)^2}{(3\kappa+2)\kappa}m,
&\quad
\kappa\in\R\setminus[-1,2^{-1/2}].
%\kappa<-1
%\ \ \mbox{and}\ \ \kappa>2^{-1/2}.
%    \qquad 
% \lambda=\pm(m-{\omega})\jj
\end{cases}
\]
We note that $\mathcal{T}_\kappa^{-}\in(-m,0)$
and $\mathcal{T}_\kappa^{+}\in(0,m)$
in the relevant ranges of $\kappa$.
\item
\label{theorem-sect2-4}
%Besides $\lambda=0$ and $\lambda=\pm 2\omega\jj$,
The point spectrum of $\bfA(\omega,\kappa)$
contains only the following additional eigenvalues
besides $\{0,\pm 2\omega\jj\}$:

\begin{enumerate}
\item
\label{theorem-sect2-4-a}
$\kappa<-1$\textup:
%% \begin{itemize}
%% \item
%% Two purely imaginary eigenvalues in the spectral gap for 
%% $\omega\in(\mathcal{T}_\kappa^{+},\Omega_\kappa)$\textup;
%% \item
%% Two real eigenvalues
%% (hence linear instability) for $\omega\in(\Omega_\kappa,2\Omega_\kappa)$,
%% with eigenvalues going to $\pm\infty$ as
%% $\omega\to 2\Omega_\kappa-0$.
%% \end{itemize}

There are no additional eigenvalues
%(besides $\{0,\pm 2\omega\jj\}$)
for
$-m<\omega\le\mathcal{T}_\kappa^{+}$\textup;
as
$\omega$ grows from $\omega\le\mathcal{T}_\kappa^{+}$
to $\mathcal{T}_\kappa^{+}+0$,
two simple purely imaginary eigenvalues
of opposite signs
bifurcate from the thresholds $\pm(m-|\omega|)\jj$
of the essential spectrum
and stay in the spectral gap
for 
$\omega\in(\mathcal{T}_\kappa^{+},\Omega_\kappa)$.
As $\omega\to\Omega_\kappa-0$, these eigenvalues collide at $\lambda=0$.
As $\omega$ grows to $\Omega_\kappa+0$,
two simple real eigenvalues of opposite signs bifurcate from $\lambda=0$,
stay on $\R$ for $\omega\in(\Omega_\kappa,2\Omega_\kappa)$
(hence the corresponding solitary waves are linearly unstable),
and go to $\pm\infty$ as
$\omega\to 2\Omega_\kappa-0$.
There are no additional eigenvalues
%(besides $\{0,\pm 2\omega\jj\}$)
for $2\Omega_\kappa\le\omega<m$.

\item
\label{theorem-sect2-4-b}
$\kappa=-1$\textup:
%% \begin{itemize}
%% \item
%% For $\omega=0$,
%% one has
%% %$\sigma(\bfA)=\sigma\sb{\mathrm{ess}}(\bfA)=\C$,
%% $\sigma\sb{\mathrm{p}}(\bfA)
%% =\C\setminus
%% \big(
%% \jj(-\infty,-m]\cup\jj[m,+\infty)
%% \big)$;
%% \item
%% No additional eigenvalues for $\omega\in(-m,m)\setminus\{0\}$.
%% \end{itemize}

There are
no additional eigenvalues
%(besides $\{0,\pm 2\omega\jj\}$)
for $\omega\in(-m,m)\setminus\{0\}$;

for $\omega=0$,
one has
$\sigma\sb{\mathrm{p}}(\bfA)
=\C\setminus
\big(
\jj(-\infty,-m]\cup\jj[m,+\infty)
\big)$.

%% \item
%% $-1<\kappa<2^{-1/2}-1$\textup:
%% \begin{itemize}
%% \item
%% Two real eigenvalues
%% (hence linear instability) for
%% $\omega\in(\max(2\Omega_\kappa,-m),
%% \max(\Omega_\kappa,-m))$
%% (this case is vacuous for $-1/3\le\kappa<2^{-1/2}-1$).
%% When
%% $-1<\kappa\le-2/3$
%% (so that $2\Omega_k\in[-m,0)$),
%% these eigenvalues go to $\pm\infty$ as
%% $\omega\to 2\Omega_\kappa+0$\textup;
%% \item
%% Two purely imaginary eigenvalues in the spectral gap for 
%% $\omega\in(\max(\Omega_\kappa,-m),\mathcal{T}_\kappa^{-})$.
%% %\item
%% %No additional eigenvalues for other values of $\omega$.
%% \end{itemize}

\item
\label{theorem-sect2-4-c}
$-1<\kappa<-1/3$\textup:

For $\max(2\Omega_\kappa,-m)<\omega<\Omega_\kappa$,
there are two simple real eigenvalues of opposite signs
(hence the corresponding solitary waves are linearly unstable).
% for
%$\omega\in(\max(2\Omega_\kappa,-m),
%\max(\Omega_\kappa,-m))$
%(this case is vacuous for $-1/3\le\kappa<2^{-1/2}-1$
%since one has $\Omega_\kappa\le -m$).
When
$-1<\kappa\le-2/3$
(so that $2\Omega_k\in[-m,0)$),
these eigenvalues go to $\pm\infty$ as
$\omega\to 2\Omega_\kappa+0$,
and then there are no additional eigenvalues
%(besides $\{0,\pm 2\omega\jj\}$)
for
$-m<\omega\le 2\Omega_\kappa$.
As $\omega\to\Omega_\kappa-0$, these eigenvalues 
collide at $\lambda=0$.
As $\omega$ grows to $\Omega_\kappa+0$,
two simple purely imaginary eigenvalues
of opposite signs
bifurcate from $\lambda=0$,
stay along the the imaginary axis
for $\omega\in(\Omega_\kappa,\mathcal{T}_\kappa^{-})$,
and disappear at the thresholds
$\pm(m-\abs{\omega})\jj$
as $\omega\to\mathcal{T}_\kappa^{-}-0$.
There are no additional eigenvalues
%(besides $\{0,\pm 2\omega\jj\}$)
for
$\omega\ge\mathcal{T}_\kappa^{-}$.

\item
\label{theorem-sect2-4-d}
$-1/3\le\kappa<2^{-1/2}-1$\textup:

There are two simple purely imaginary eigenvalues
of opposite signs
for $-m<\omega<\mathcal{T}_\kappa^{-}$,
which disappear at the thresholds
$\pm(m-\abs{\omega})\jj$
as $\omega\to\mathcal{T}_\kappa^{-}-0$.
There are no additional eigenvalues
%(besides $\{0,\pm 2\omega\jj\}$)
for
$\omega\ge\mathcal{T}_\kappa^{-}$.

\item
\label{theorem-sect2-4-e}
$2^{-1/2}-1\le\kappa\le 2^{-1/2}$\textup:
No additional eigenvalues
%(besides $\{0,\pm 2\omega\jj\}$)
for any $\omega\in(-m,m)$.

\item
\label{theorem-sect2-4-f}
%%$2^{-1/2}<\kappa\le 1$\textup:
$\kappa>2^{-1/2}$\textup:

There are no eigenvalues
%(besides $\{0,\pm 2\omega\jj\}$)
for 
$\omega\le\mathcal{T}_\kappa^{+}$.
As $\omega$ increases to
$\mathcal{T}_\kappa^{+}+0$,
two simple purely imaginary eigenvalues
of opposite signs
bifurcate from the thresholds
$\pm(m-\abs{\omega})\jj$
and stay on the imaginary axis.
If $\kappa>1$ (so that $\Omega_\kappa\in(-m,m)$),
these eigenvalues collide at $\lambda=0$
as $\omega\to\Omega_\kappa-0$,
and, as $\omega$ grows to $\Omega_\kappa+0$,
two simple real eigenvalues of opposite signs
bifurcate from $\lambda=0$
and stay on the real axis
for $\omega\in(\Omega_\kappa,m)$
(hence the corresponding solitary waves are linearly unstable).

\end{enumerate}
\end{enumerate}
\end{theorem}

\begin{figure}
\begin{center}

% \ifpdf

%%\ac{DISABLING THE PICTURE!!}
%\begin{comment}
%\end{comment}
\begin{tikzpicture}[>=stealth,thick, xscale=1.5, yscale=1.0 ]
    \begin{axis}[%ticks=none,
        xmin=-2,xmax=2,
        ymin=-1,ymax=1,
        axis x line=middle,
        axis y line=middle,
%        axis line style=->,
        ylabel={$\omega$},
        xlabel={$\kappa$},
x tick label style={yshift={(\tick==-1)*2em}},
y tick label style={xshift={(\tick==-1)*2em}},
        ]

\addplot[line width=0.2mm,black,dotted,-]expression[domain=-2:2,samples=10]{1};
\addplot[line width=0.2mm,black,dotted,-]expression[domain=-2:2,samples=10]{-1};
\addplot[line width=1.0mm,no marks,blue,-]expression[domain=-2:-0.333,samples=50]{(x+1)/(2*x)}; %%Omega
\addplot[line width=1.0mm,no marks,blue,-]expression[domain=1:2,samples=10]{(x+1)/(2*x)}; %%Omega
%\addplot[line width=0.75mm,no marks,,dash pattern={on 7pt off 2pt},red,-]expression[domain=-2:-0.5,samples=50]{(x+1)/x}; %%2Omega
\addplot[line width=0.75mm,no marks,,dash pattern={on 7pt off 2pt on 1 pt off 2pt },red,-]expression[domain=-2:-0.5,samples=50]{(x+1)/x}; %%2Omega

%\addplot[line width=0.75mm,no marks,dash pattern={on 7pt off 2pt on 1 pt off 2pt },green,-]expression[domain=-2:-0.707,samples=50]{(x+1)^2/x/(3*x+2)}; %%Tau plus
\addplot[line width=0.75mm,no marks,dotted,MyDarkGreen,-]expression[domain=-2:-1,samples=50]{(x+1)^2/x/(3*x+2)}; %%Tau plus
\addplot[line width=0.75mm,no marks,dotted,MyDarkGreen,-]expression[domain=0.7:2,samples=10]{(x+1)^2/x/(3*x+2)}; %%Tau plus

%\addplot[line width=0.75mm,no marks,dotted,blue,-]expression[domain=-1.71:-0.293,samples=50]{(x+1)^2/x/(x+2)}; %%Tau minus
\addplot[line width=0.75mm,no marks,dotted,MyDarkGreen,-]expression[domain=-1:-0.293,samples=50]{(x+1)^2/x/(x+2)}; %%Tau minus

\addplot[line width=0.1mm,no marks,red,-]expression[domain=-2:-0.3,samples=600]{3*(x+1)/4/x+(x+1)/4/x*sin(10000*x)}; %%2Omega
\addplot[line width=0.1mm,no marks,red,-]expression[domain=0.7:2,samples=600]{((x+1)/2/x+1)/2+((x+1)/2/x-1)/2*sin(10000*x)}; %%2Omega
\addplot[line width=0.1mm,no marks,dotted,MyDarkGreen,-]expression[domain=-1:-0.293,samples=100]
 {((x+1)/2/x+(x+1)^2/x/(x+2))/2+((x+1)/2/x-(x+1)^2/x/(x+2))/2*sin(20000*x)};
\addplot[line width=0.1mm,no marks,dotted,MyDarkGreen,-]expression[domain=-2:-1,samples=100]
 {((x+1)/2/x+(x+1)^2/x/(3*x+2))/2+((x+1)/2/x-(x+1)^2/x/(3*x+2))/2*sin(20000*x)};
\addplot[line width=0.1mm,no marks,dotted,MyDarkGreen,-]expression[domain=0.7:2,samples=200]
 {((x+1)/2/x+(x+1)^2/x/(3*x+2))/2+((x+1)/2/x-(x+1)^2/x/(3*x+2))/2*sin(20000*x)};

\filldraw[black,fill=black] (-1,0) ellipse (3pt and 4pt);

    \end{axis}
\end{tikzpicture}
%\begin{comment}
%\end{comment}

% \else
% \begin{picture}(0,210)(0,-50)
% \put(-170,50){
% \textcolor{red}{*** TO SEE THE PICTURE, PLEASE COMPILE WITH PDFLATEX ***}}
% \end{picture}
% \fi

\caption{
Location of eigenvalues
of the linearized operator $\bfA(\omega,\kappa)$
for different values of parameters
$\omega\in(-m,m)$ and $\kappa\in\R$.
In the unshaded regions, there are no eigenvalues
besides $\lambda=0$ and $\lambda=\pm 2\omega\jj$.
%%(which are present for all values of $\omega\in(-1,1)$ and $\kappa\in\R$).
The Kolokolov curve
$\omega=\Omega_\kappa$
%%(\textcolor{red}{solid red lines})
(\textcolor{blue}{\bf thick solid curves} for $\kappa<-0.5$ and $\kappa>1$)
correspond to
collision of two eigenvalues
% (each of geometric multiplicity one)
at $\lambda=0$
(as indicated by the Kolokolov condition).
The virtual level curve
$\omega=\mathcal{T}_\kappa$
%(\textcolor{blue}{blue dots} for $\kappa\in(-1,0)$)
(\textcolor{MyDarkGreen}{\bf dotted curves} for $\kappa<2^{-1/2}-1$ and $\kappa>2^{-1/2}$)
corresponds to virtual levels at thresholds $\pm\jj(m-\abs{\omega})$.
%% when $\omega<0$.
%% The virtual level curves
%% $\omega=\mathcal{T}_\kappa^{+}\in(0,m)$
%% %(\textcolor{green}{green dots} for $\kappa<-1$ and $\kappa>0$)
%% (\textcolor{green}{\bf dash-dot curves} for $\kappa<-1$ and $\kappa>2^{-1/2}$)
%% correspond to virtual levels at thresholds
%% $\pm\jj(m-\abs{\omega})$  when $\omega>0$.
The \textcolor{MyDarkGreen}{\bf dotted regions}
between the Kolokolov and virtual level curves
correspond to two
%simple
purely imaginary eigenvalues
in the spectral gap.
On the other side of the Kolokolov curve
there is a pair of real eigenvalues
(linear instability; regions on the graph
filled with \textcolor{red}{\bf thin lines}).
In the region where there are two real eigenvalues $\pm\lambda\in\R$,
one has $\lambda\to +\infty$ as $\omega$ and $\kappa$ approach
the curve $\omega=2\Omega_\kappa$
(\textcolor{red}{\bf dash-dot curve} for $\kappa<-2/3$).
%% Besides eigenvalues mentioned on this plot,
%% there are always eigenvalues
%% $\lambda=0$ and $\lambda=\pm 2\omega\jj$.
At $\kappa=-1$, $\omega=0$ ({\bf bullet} on the plot),
the spectrum of the linearized operator
$\bfA$ consists of the whole complex plane.
}
\label{fig-omegak}
\end{center}
\end{figure}
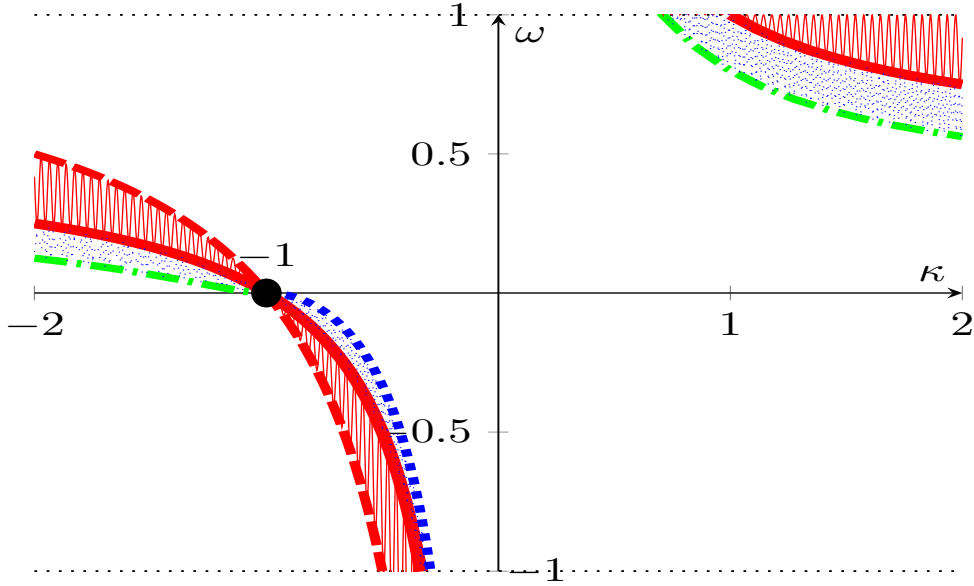

We depict these cases on Figure~\ref{fig-omegak}.
%% \begin{comment}
%% We also depict
%% a typical shape of the point spectrum of
%% $\bfA(\omega,\kappa)$
%% as described by Theorem~\ref{theorem-sect2}
%% on Figure~\ref{fig-resume2}.
%% \end{comment}

\begin{remark}
The spectrum of the linearization operator for a one-dimensional Schr\"odinger model with concentrated nonlinearity is studied in \cite{buslaev2008asymptotic} and more thoroughly in \cite{komech2012asymptotic}.
\end{remark}

\begin{remark}
For $\kappa<-1$,
there are real eigenvalues
for $\omega\in(\Omega_\kappa,2\Omega_\kappa)$
which bifurcate from zero (when $\omega=\Omega_\kappa$)
and slip off to $\pm\infty$ as $\omega\ge 2\Omega_\kappa$,
as one can see from explicit expressions for eigenvalues;
see Lemma~\ref{lemma-w-kappa} below.
As a result,
for $\omega$ between $\Omega_\kappa$ and $2\Omega_\kappa$,
there are
eigenvalue families $\pm\lambda(\omega,\kappa)\in\sigma\sb{\mathrm{p}}(\bfA(\omega,\kappa))$
such that
\[
%\qquad
\begin{cases}
\lim\limits\sb{\omega\to 2\Omega_\kappa-0}\lambda(\omega,\kappa)\to+\infty,
&\kappa<-1;
\\[2ex]
\lim\limits\sb{\omega\to 2\Omega_\kappa+0}\lambda(\omega,\kappa)\to+\infty,
&-1<\kappa<-1/2.
\end{cases}
\]
\end{remark}

%% \begin{comment}
%% \begin{figure}[ht!]
%% \begin{center}
%% %\includegraphics[width=14cm, angle=-90]{resume2b.pdf}
%% \includegraphics[width=10cm]{resume2c.eps}
%% \ac{REMEMBER TO UNCOMMENT THE PICTURE!!!}
%% \end{center}
%% \caption{Spectrum of the linearization operator as a function
%% of $\omega\in(-m,m)$, $m=1$,
%% in the case $\kappa>1$.
%% There are
%% no isolated eigenvalues for $\omega\in(-m,\mathcal{T}_\kappa^{+})$;
%% two imaginary isolated eigenvalues for $\omega\in(\mathcal{T}_\kappa^{+},\Omega_\kappa)$;
%% two real eigenvalues for $\omega\in(\Omega_\kappa,m)$.
%% }
%% \label{fig-resume2}
%% \end{figure}
%% \end{comment}

We give the proof
of Theorem~\ref{theorem-sect2}
In the next section.

\section{Proof of Theorem~\ref{theorem-sect2}}
\label{sect-theorem-sect2-proof}

\subsection{
Symmetries and the essential spectrum of $\bfA$}

For Theorem~\ref{theorem-sect2}~\itref{theorem-sect2-1},
we notice that the symmetry
$
\lambda\in\sigma\sb{\mathrm{p}}(\bfA)
\ \Leftrightarrow\ \bar\lambda\in\sigma\sb{\mathrm{p}}(\bfA)
$
follows from $\bfA$ having real coefficients,
while the symmetry
$
\bar\lambda\in\sigma\sb{\mathrm{p}}(\bfA)
\ \Leftrightarrow\ -\lambda\in\sigma\sb{\mathrm{p}}(\bfA)
$
follows from
\begin{equation}\label{Eq:Symmetries}
\bfA
=
\mathbf{J}\mathbf{L},\quad\bfA^*
=
(\mathbf{J}\mathbf{L})^*=\mathbf{L}^* \mathbf{J}^*=-\mathbf{L}\mathbf{J},
\quad
\mbox{where}
\quad
\mathbf{L}=
\begin{bmatrix}L_{+}&0\\0&L_{-}\end{bmatrix},
\quad
\mathbf{J}=\begin{bmatrix}0&I_2\\-I_2&0\end{bmatrix},
\end{equation}
with $L_\pm$ defined in \eqref{l-p-l-m-1},
while $\sigma\sb{\mathrm{p}}(\mathbf{L}\mathbf{J})
=\sigma\sb{\mathrm{p}}(\mathbf{J}\mathbf{L})$
since $\mathbf{J}$ is bounded and invertible.
This proves
Theorem~\ref{theorem-sect2}~\itref{theorem-sect2-1}.

Let us consider the essential spectrum
(Theorem~\ref{theorem-sect2}~\itref{theorem-sect2-7}).
For $\kappa=0$,
the essential spectrum
can be obtained from Lemma~\ref{lemma-l}:
\begin{eqnarray}\label{mm}
\sigma\sb{\mathrm{ess}}(\bfA(\omega,0))
=
\sigma\sb{\mathrm{ess}}
\left(
\begin{bmatrix}0&1\\-1&0\end{bmatrix}
\otimes
(D_m-\omega I_2-2\mu\delta(x))
\right)
=
\jj\big(\R\setminus(-m+\abs{\omega},m-\abs{\omega})\big).
\end{eqnarray}
Since the Weyl sequences do not depend on the value of
$\kappa\in\R$ 
%\ac{Answering Referee Y, comment 14:}
(the argument is the same as in the proof of Lemma~\ref{lemma-l}~\itref{lemma-l-1}),
the Weyl spectrum
$\sigma\sb{\mathrm{ess},2}(\bfA)$
(we recall that for a closed operator
$A$ in the Hilbert space $\bfH$,
$\sigma\sb{\mathrm{ess},2}(\bfA)$
is defined as $\lambda\in\C$
such that either the range of $A-\lambda I_\bfH$
is not closed
or $\ker(A-\lambda I_\bfH)=\infty$;
see \cite[\S I.4]{edmunds2018spectral})
coincides with \eqref{mm}.
Its complement in the complex plane
consists of a single connected component:
\begin{eqnarray}
\label{component}
\C\setminus\sigma\sb{\mathrm{ess},2}(\bfA)
=
\C\setminus\big(
\jj(-\infty,-(m-\abs{\omega})]
\cup
\jj[(m-\abs{\omega}),+\infty)
\big).
\end{eqnarray}
If $(\omega,\kappa)\ne(0,-1)$,
then, as we will see below
in the proof of Part~\itref{theorem-sect2-4},
the component \eqref{component}
contains at most discrete spectrum,
and we deduce that 
the essential spectrum
$\sigma\sb{\mathrm{ess}}(\bfA):=\sigma\sb{\mathrm{ess},5}(\bfA)$
also coincides with \eqref{mm}.
Let us mention here that $\sigma_{\mathrm{ess,5}}(A)$
for a closed operator $A$
is known as the Browder spectrum \cite[Definition 11]{browder1961spectral}
and that
$\sigma(A)$ consists of a disjoint union of
$\sigma_{\mathrm{ess},5}(A)$
and the discrete spectrum $\sigma_{\mathrm{d}}(A)$,
which in turn is the set of isolated points of
the spectrum with corresponding Riesz projections
being of finite rank;
for more details, see
\cite{edmunds2018spectral,opus}.

If $(\omega,\kappa)=(0,-1)$, then,
as we will see in the proof of
Part~\itref{theorem-sect2-4-b},
this entire connected component consists of the point spectrum;
it follows that in this case one has
$
\sigma(\bfA(0,-1))
=\sigma\sb{\mathrm{ess}}(\bfA(0,-1))=\C$.
This proves
Part~\itref{theorem-sect2-7}.

\subsection{
Discrete spectrum of $\bfA$}
\label{sect-spectrum-discrete}

We proceed to the discrete spectrum.

Part~\itref{theorem-sect2-5} of Theorem~\ref{theorem-sect2}
follows from Lemma~\ref{lemma-l} (see also Remark~\ref{Rem:OnZ}).
Namely, the kernel of $L_{-}=L(\omega,0)$
is generically one-dimensional
%%(with the kernel in $\bfX_{\mbox{\rm\footnotesize even-odd-even-odd}}$)
except at $\omega=0$, when it is two-dimensional.
The kernel of $L_{+}=L(\omega,\kappa)$
is generically zero-dimensional;
it is one-dimensional when only one of $\omega$, $\kappa$
vanishes
and it is two-dimensional when both
$\omega=0$ and $\kappa=0$.
This immediately adds up to the conclusion about
the dimension of the kernel of
$\bfA(\omega,\kappa)
=
\begin{bmatrix}0&L(\omega,0)\\-L(\omega,\kappa)&0\end{bmatrix}$
stated in Part~\itref{theorem-sect2-5}.

For all other cases of Theorem~\ref{theorem-sect2},
it is convenient
to start with the exceptional case $\kappa=0$.

%% \ac{Answering Referee Y, comment 6:
%%   added
%% Remark~\ref{remark-bad}
%% and Remark~\ref{remark-nonzero}.
%% }

\begin{remark}\label{remark-bad}
Let us point out that the case $\kappa=0$ is possible:
given $\alpha>0$ and $\omega\in(-m,m)$
which satisfy
\begin{eqnarray}\label{ff}
f(\alpha^2)=2\mu(\omega)
\end{eqnarray}
(see \eqref{a-b-f}),
one can always vary the nonlinearity
$f(\tau)$ so that \eqref{ff} is satisfied
while $f'(\alpha^2)=0$;
now \eqref{def-kappa}
gives $\kappa=0$.
At the same time,
according to
\eqref{a-b-f},
now $\alpha$ is no longer
differentiable with respect to
$\omega$,
and then the derivative
$\p_\omega\phi_\omega$ is undefined.
We will return to this in Section~\ref{Sec:Kolokolov}.
\end{remark}

In this case, one has
$
\bfA(\omega,0)
=
\begin{bmatrix}0&1\\-1&0\end{bmatrix}
\otimes
L(\omega,0)$,
with $L$ from \eqref{l-kappa},
and the statements of Theorem~\ref{theorem-sect2}
follow from Lemma~\ref{lemma-l}
which gives
$\sigma\sb{\mathrm{p}}(L(\omega,0))=\{-2\omega,\,0\}$,
with both eigenvalues
of geometric multiplicity
$1$ when $\omega\ne 0$
(with the kernel in $\bfX_{\mbox{\rm\footnotesize even-odd}}$)
and $\lambda=0$ of geometric multiplicity $2$ when $\omega=0$
(with the kernel
in $\bfX_{\mbox{\rm\footnotesize odd-even}}$
and in $\bfX_{\mbox{\rm\footnotesize even-odd}}$).
It follows that $\sigma\sb{\mathrm{p}}(\bfA(\omega,0))=\{0,\,\pm 2\omega\jj\}$,
with
$\lambda=\pm 2\omega\jj$ of geometric multiplicity $1$ and
$\lambda=0$ of geometric multiplicity $2$ when
$\omega\ne 0$;
$\lambda=0$ 
of geometric multiplicity $4$ when
$\omega=0$.
We note that since $L(\omega,0)$ is selfadjoint,
there could be no nonzero $\psi,\,\theta\in\dom(L(\omega,0))$
such that $L(\omega,0)\psi=0$ and $L(\omega,0)\theta=\psi$,
hence the algebraic multiplicity
of eigenvalue $\lambda=0$ of $\bfA(\omega,0)$
coincides with its geometric multiplicity.
%% computed in Part~\itref{theorem-sect2-5}.
This completes the proof of 
Theorem~\ref{theorem-sect2} in the case $\kappa=0$.

\medskip

In the rest of the argument, we make the assumption
that $\kappa$ is nonzero:
\begin{eqnarray}\label{kappa-nonzero}
\kappa\in\R\setminus\{0\}.
\end{eqnarray}
Below, we will use the fact that the operator $\bfA(\omega,\kappa)$
from \eqref{def-a-ch2}
is invariant in the subspaces
$\bfX_{\mbox{\rm\footnotesize odd-even-odd-even}}$
and
$\bfX_{\mbox{\rm\footnotesize even-odd-even-odd}}$
(see \eqref{def-oeoe} and \eqref{def-eoeo}) of $L^2(\R,\C^4)$,
so the search for eigenvalues and eigenvectors can be restricted to the analysis
of the spectrum of $\bfA(\omega,\kappa)$ in these two subspaces.

\subsubsection{Discrete spectrum of $\bfA$
in odd-even-odd-even subspace
and eigenvalue $2\omega\jj$}
\label{sect-spectrum-oe}

To prove Theorem~\ref{theorem-sect2}~\itref{theorem-sect2-2},
we restrict $\bfA$ to the subspace
$\bfX_{\mbox{\rm\footnotesize odd-even-odd-even}}$.

For $x\ne 0$,
a representation for an $L^2$-solution of the equation
\begin{eqnarray}\label{a-psi-lambda}
\bfA\Psi=\lambda\Psi,
\qquad
\lambda\in\C,
\end{eqnarray}
belonging to the subspace
$\bfX_{\mbox{\rm\footnotesize odd-even-odd-even}}$
(see \eqref{def-x-x})
is given by
\begin{eqnarray}\label{psi-o-e}
\Psi(x)
=
c_1
\begin{bmatrix}
\nu_{+}(\omega,\Lambda)\sgn x\\S_{+}(\omega,\Lambda)
\\-\jj \nu_{+}(\omega,\Lambda)\sgn x\\-\jj S_{+}(\omega,\Lambda)
\end{bmatrix}e^{-\nu_{+}(\omega,\Lambda)\abs{x}}
+
c_2
\begin{bmatrix}
\nu_{-}(\omega,\Lambda)\sgn x\\S_{-}(\omega,\Lambda)
\\\jj \nu_{-}(\omega,\Lambda)\sgn x\\\jj S_{-}(\omega,\Lambda)
\end{bmatrix}e^{-\nu_{-}(\omega,\Lambda)\abs{x}},
\qquad
c_1,\,c_2\in\C;
\end{eqnarray}
above, the value $\Lambda\in\C$
is defined by
\begin{eqnarray}\label{lambda-lambda}
\lambda=\jj\Lambda,
\end{eqnarray}
with $\lambda$ an eigenvalue from \eqref{a-psi-lambda}.
We used the notations
\begin{eqnarray}\label{def-m-n}
\begin{array}{l}
\nu_{+}(\omega,\Lambda):=\sqrt{m^2-(\omega+\jj\lambda)^2}=\sqrt{m^2-(\omega-\Lambda)^2},
\qquad
\Re\nu_{+}(\omega,\Lambda)\ge 0,
\\[2ex]
\nu_{-}(\omega,\Lambda):=\sqrt{m^2-(\omega-\jj\lambda)^2}=\sqrt{m^2-(\omega+\Lambda)^2},
\qquad
\Re\nu_{-}(\omega,\Lambda)\ge 0
\end{array}
\end{eqnarray}
(these expressions
come from considering the characteristic equation of the homogeneous system
with constant coefficients,
$(\bfA-\jj\Lambda I_4)\Psi=0$,
with $\Psi\in C^{\infty}(\R\setminus\{0\})$)
and
\begin{eqnarray}\label{def-r-s}
S_{+}(\omega,\Lambda)=m-\omega+\Lambda,
\qquad
S_{-}(\omega,\Lambda)=m-\omega-\Lambda.
\end{eqnarray}
We assume that
$\nu_{-}$ and $\nu_{+}$ are non-vanishing
and with positive real part
as long as the corresponding coefficient
in \eqref{psi-o-e} is nonzero,
so that $\Psi\in L^2(\R,\C^4)$.

\begin{remark}\label{remark-sometimes}
We note that
the vectors in \eqref{psi-o-e}
are linearly independent
unless either $S_{+}=\nu_{+}=0$
or $S_{-}=\nu_{-}=0$.
This degeneracy takes place at the threshold point
$\Lambda=m-\omega$
(where $S_{-}=\nu_{-}=0$)
and at the threshold point $\Lambda=-(m-\omega)$
(where $S_{+}=\nu_{+}=0$);
note that for $\omega<0$ these threshold points
are embedded into the essential spectrum.
At
$\Lambda=m-\omega$,
in place of \eqref{psi-o-e},
one needs to consider
\begin{eqnarray}\label{psi-o-e-1}
\Psi(x)
=
c_1
\begin{bmatrix}
\nu_{+}(\omega,m-\omega)\sgn x\\S_{+}(\omega,m-\omega)
\\-\jj \nu_{+}(\omega,m-\omega)\sgn x\\-\jj S_{+}(\omega,m-\omega)
\end{bmatrix}e^{-\nu_{+}(\omega,m-\omega)\abs{x}}
+
c_2
\begin{bmatrix}
\nu_{-}^{\mathrm{reg}}(\omega,m-\omega)\sgn x
\\S_{-}^{\mathrm{reg}}(\omega,m-\omega)
\\\jj \nu_{-}^{\mathrm{reg}}(\omega,m-\omega)\sgn x
\\\jj S_{-}^{\mathrm{reg}}(\omega,m-\omega)
\end{bmatrix},
%%e^{-\nu_{-}\abs{x}},
\qquad
c_1,\,c_2\in\C,
\end{eqnarray}
%instead of the Ansatz~\eqref{psi-x-e-o-e-o}, one can use this expression
with
\begin{eqnarray}
\label{reg-1}
&&
\nu_{-}^{\mathrm{reg}}(\omega,m-\omega):=
\lim\sb{\Lambda\to m-\omega+0}
\frac{\nu_{-}(\omega,\Lambda)}{\sqrt{m-\omega-\Lambda}}
%=\sqrt{m+\omega+\Lambda}
=\sqrt{2m+2\omega},
\\
&&
\label{reg-2}
S_{-}^{\mathrm{reg}}(\omega,m-\omega):=
\lim\sb{\Lambda\to m-\omega+0}
\frac{S_{-}(\omega,\Lambda)}{\sqrt{m-\omega-\Lambda}}
=0.
\end{eqnarray}
The case $\Lambda=-(m-\omega)$
is treated similarly.
\end{remark}

We note that
if $\lambda\in\jj\R$, $\abs{\lambda}\ge m+\abs{\omega}$
(so that $\lambda$ is beyond the embedded
thresholds at $\pm \jj(m+\abs{\omega})$),
then
both
$m^2-(\omega-\jj\lambda)^2\le 0$ and $m^2-(\omega+\jj\lambda)^2\le 0$,
hence
there is no corresponding square-integrable
function $\Psi\ne 0$ of the form \eqref{psi-o-e}.

An eigenvector is an element of the domain of  $\dom(\bfA)$ given in \eqref{Eq:DomainOfA}; it has
to satisfy the jump condition at the origin,
which is given by
\begin{eqnarray}\label{nu-a-b}
\begin{cases}
2\jj \nu_{+} c_1-2\jj \nu_{-} c_2
+2\mu(-\jj S_{+} c_1+\jj S_{-} c_2)=0,
\\2\nu_{+} c_1+2\nu_{-} c_2
-2\mu(S_{+} c_1+S_{-} c_2)=0.
\end{cases}
\end{eqnarray}
Since $c_1,\,c_2\in\C$ are not simultaneously zeros,
the compatibility condition leads to
\begin{eqnarray}\label{det-zero}
\det\begin{bmatrix}
\jj\nu_{+}-\jj S_{+} \mu&-\jj \nu_{-}+\jj S_{-}\mu
\\\nu_{+}-S_{+}\mu&\nu_{-}-S_{-}\mu
\end{bmatrix}
=
2\jj (\nu_{-}-S_{-} \mu)(\nu_{+}-S_{+} \mu)
=0.
\end{eqnarray}
The relation $\nu_{+}-S_{+}\mu=0$
results in
$m^2-(\omega-\Lambda)^2
=(m-\omega+\Lambda)^2
\frac{m-\omega}{m+\omega}$;
canceling $m-\omega+\Lambda\ne 0$, we have
\[
m+\omega-\Lambda
=(m-\omega+\Lambda)\frac{m-\omega}{m+\omega},
\]
which leads to $\Lambda=2\omega$.
%% $
%% m^2-(\omega+\Lambda)^2
%% =
%% \frac{m-\omega}{m+\omega}(m-\omega-\Lambda)^2,
%% $
%% which takes the form
%% \[
%% (m+\omega)^2-(m-\omega)^2
%% =-2m\Lambda,
%% \]
%% leading to $\Lambda=-2\omega$.
%% Similarly, the relation
%% $\nu_{+}-S_{+}\mu=0$
%% leads to $\Lambda=2\omega$.
Similarly, the relation
$\nu_{-}-S_{-}\mu=0$
leads to $\Lambda=-2\omega$.
Thus, if $\omega\neq 0$,
the factors $\nu_{+}-S_{+}\mu$ and $\nu_{-}-S_{-}\mu$
in \eqref{det-zero}
cannot vanish simultaneously;
one can see from \eqref{nu-a-b}
that either $c_2$ or $c_1$ vanishes
(depending on whether
$\nu_{+}-S_{+}\mu=0$
or $\nu_{-}-S_{-}\mu=0$, respectively),
hence the geometric multiplicity of
each of eigenvalues $\pm 2\jj \omega$ is equal to one.

In the case $\omega=0$,
the jump condition \eqref{nu-a-b} becomes trivial,
and $c_1$ and $c_2$ in \eqref{psi-o-e}
could take arbitrary values.
Then the two terms in the right-hand side of
\eqref{psi-o-e}
are linearly independent eigenvectors
corresponding to eigenvalue $\lambda=2\omega\jj=0$
of $\bfA$ restricted to
$\bfX\sb{\mbox{\rm\footnotesize odd-even-odd-even}}$,
proving that its geometric multiplicity equals two.

We note that
for $m/3\le \abs{\omega}<m$
the eigenvalue $\lambda=\Lambda\jj=2\omega\jj$ is
embedded into the essential spectrum of $\bfA$;
the same is true for $\lambda=-2\omega\jj$.
For example, if $m/3\le\omega<m$,
then
$\nu_{+}=\varkappa$,
the value
$\nu_{-}=\sqrt{m^2-9\omega^2}$ is purely imaginary,
$S_{+}(\omega,2\omega)=m+\omega$,
$S_{-}(\omega,2\omega)=m-3\omega$,
$\nu_{+}(\omega,2\omega)=S_{+}(\omega,2\omega)\mu$,
the system \eqref{nu-a-b}
takes the form
$\nu_{-}(\omega,2\omega) c_2-S_{-}(\omega,2\omega)\mu c_2=0$,
which results in
$c_2=0$ and arbitrary $c_1\in\C$;
due to $\nu_{+}(\omega,2\omega)>0$, one can see that
$\Psi$ from \eqref{psi-o-e}
belongs to $L^2$.
%% If $m^2-(\omega-\Lambda)^2\le 0$, then we have to take $a=0$,
%% and hence the condition for $b\neq 0$
%% from \eqref{nu-a-b}
%% is $\nu_{-}-S_{-}\mu=0$.
%% If instead $m^2-(\omega+\Lambda)^2\in \R_{-}$,
%% then we have to take $b=0$,
%% and hence the condition for $a\neq 0$
%% from \eqref{nu-a-b}
%% is $\nu_{+}+S_{+}\mu=0$.
%% Therefore, the previous reasoning also covers the embedded eigenvalues.

Let us consider the 
algebraic multiplicity
of eigenvalues $\pm 2\omega\jj$ when they are isolated
(that is, when $\abs{\omega}<m/3$).
We recall that $\bfA^\ast=-\mathbf{J}^\ast \bfA \mathbf{J}$
(see~\eqref{Eq:Symmetries})
and hence if $\Psi\in \range(\bfA-\lambda I_4)$ then
$\Psi\in \mathbf{J}\ker(\bfA-\lambda I_4)^\bot$.
So if the algebraic multiplicity were larger than one
(while the geometric multiplicity equals one),
then necessarily $\Psi$ and $\mathbf{J} \Psi$
would be orthogonal.
At the same time,
for $\Lambda=2\omega$,
when in \eqref{psi-o-e}
we can take $c_1=1$ and $c_2=0$,
one has
$\Psi^*\mathbf{J}\Psi=-2\jj S_{+}(\omega,2\omega)^2
=-2\jj(m+\omega)^2$,
which is nonzero, showing that the algebraic multiplicity
of $\lambda=\jj\Lambda=2\omega\jj$
coincides with the geometric multiplicity.
The case $\Lambda=-2\omega$ is treated similarly.
We thus conclude that
for $\omega\in(-m/3,m/3)\setminus\{0\}$
the eigenvalues $\lambda=\pm 2 \omega\jj$ are
of algebraic multiplicity one
while for $\omega=0$ the eigenvalue $\lambda=0$
is of algebraic multiplicity two.
This completes the proof of
Theorem~\ref{theorem-sect2}~\itref{theorem-sect2-2}.

\begin{remark}
One has
$\lambda=\pm 2\omega\jj\in\sigma\sb{\mathrm{p}}(\bfA)$
due to the $\mathbf{SU}(1,1)$-invariance of the Soler model
\cite{boussaid2018spectral}.
\end{remark}

\subsubsection{Discrete spectrum of $\bfA$
in even-odd-even-odd subspace
and virtual levels at thresholds}
\label{sect-spectrum-eo}

In this Section we prove 
Theorem~\ref{theorem-sect2}~\itref{theorem-sect2-3}
and
Theorem~\ref{theorem-sect2}~\itref{theorem-sect2-4}.
Similarly to our approach in Section~\ref{sect-spectrum-oe},
any square-integrable solution of $\bfA\Psi=\lambda\Psi $  with
$\lambda=\jj\Lambda$ in the subspace
$\bfX_{\mbox{\rm\footnotesize even-odd-even-odd}}$
of $L^2(\R,\C^4)$
(see \eqref{def-x-x}) can be represented as
(cf. \eqref{psi-o-e})
\begin{eqnarray}\label{psi-x-e-o-e-o}
\Psi(x)
=
c_1
\begin{bmatrix}
\nu_{+}(\omega,\Lambda)\\S_{+}(\omega,\Lambda)\sgn x\\-\jj \nu_{+}(\omega,\Lambda)\\-\jj S_{+}(\omega,\Lambda)\sgn x
\end{bmatrix}e^{-\nu_{+}(\omega,\Lambda)\abs{x}}
+
c_2
\begin{bmatrix}
\nu_{-}(\omega,\Lambda)\\S_{-}(\omega,\Lambda)\sgn x\\\jj \nu_{-}(\omega,\Lambda)\\\jj S_{-}(\omega,\Lambda)\sgn x
\end{bmatrix}e^{-\nu_{-}(\omega,\Lambda)\abs{x}},
\quad
c_1,\,c_2\in\C,
\end{eqnarray}
with
$\nu_\pm(\omega,\Lambda)$, $S_\pm(\omega,\Lambda)$
from \eqref{def-m-n} and \eqref{def-r-s},
where we will assume
that both $\nu_{-}$ and $\nu_{+}$ are non-vanishing
and with positive real part, so that $\Psi\in L^2(\R,\C^4)$.
Again, by Remark~\ref{remark-sometimes},
at $\Lambda=m-\omega$,
the vectors in \eqref{psi-x-e-o-e-o}
become linearly dependent
(the second vector vanishes)
and one can use
the following decomposition
(see \eqref{psi-o-e-1}, \eqref{reg-1}, \eqref{reg-2}):
\begin{eqnarray}\label{psi-e-o-1}
\Psi(x)
=
c_1
\begin{bmatrix}
\nu_{+}(\omega,m-\omega)
\\S_{+}(\omega,m-\omega)\sgn x\\-\jj \nu_{+}(\omega,m-\omega)
\\-\jj S_{+}(\omega,m-\omega)\sgn x
\end{bmatrix}e^{-\nu_{+}(\omega,m-\omega)\abs{x}}
+
c_2
\begin{bmatrix}
\nu_{-}^{\mathrm{reg}}(\omega,m-\omega)
\\S_{-}^{\mathrm{reg}}(\omega,m-\omega)\sgn x
\\\jj \nu_{-}^{\mathrm{reg}}(\omega,m-\omega)
\\\jj S_{-}^{\mathrm{reg}}(\omega,m-\omega)\sgn x
\end{bmatrix},
%%e^{-\nu_{-}\abs{x}},
\qquad
c_1,\,c_2\in\C.
\end{eqnarray}
The jump condition for $\Psi$ at the origin
takes the form 
\begin{eqnarray}\label{Eq:SystemSandR-e-o-e-o}
\begin{cases}
-2\jj S_{+} c_1+2\jj S_{-} c_2
-2(
-\jj \nu_{+} c_1
+\jj \nu_{-} c_2)\mu
=0,
\\-(2 S_{+} c_1+2 S_{-} c_2)
+2(\nu_{+} c_1+\nu_{-} c_2)(1+2\kappa)\mu
=0.
\end{cases}
\end{eqnarray}
To find eigenvalues,
we need to consider the compatibility condition
for the system \eqref{Eq:SystemSandR-e-o-e-o},
\begin{eqnarray}\label{c-cond}
\det
\begin{bmatrix}
\mu\nu_{+}-S_{+}&-\mu\nu_{-}+S_{-}
\\(1+2\kappa)\mu\nu_{+}-S_{+}&(1+2\kappa)\mu\nu_{-}-S_{-}
\end{bmatrix}
=0.
\end{eqnarray}

\begin{lemma}\label{lemma-no}
For
$\omega\in(-m,m)$ and $\kappa\in\R\setminus\{0\}$,
the operator $\bfA(\omega,\kappa)$
restricted onto
$\bfX_{\mbox{\rm\footnotesize even-odd-even-odd}}$
has no embedded eigenvalues.
\end{lemma}

\begin{proof}
%First let us notice that there can be no embedded eigenvalues.
To have square-integrable solutions,
the values of $c_1,\,c_2$ in \eqref{psi-x-e-o-e-o}
corresponding to
purely imaginary or zero values of $\nu_{+}$, $\nu_{-}$
have to vanish.
If $\Re\nu_{-}=0$, then $c_2=0$ in \eqref{psi-x-e-o-e-o}
(if $\Lambda=m-\omega$, then
$c_2=0$ in \eqref{psi-e-o-1}).
Then the jump condition \eqref{Eq:SystemSandR-e-o-e-o} yields
\[
S_{+}-\nu_{+} \mu=0,
\qquad
S_{+}-(1+2\kappa)\mu\nu_{+}=0.
\]
Since $\mu(\omega)>0$,
the assumption $\kappa\ne 0$ leads to
$\nu_{+}=0$, but then
\eqref{psi-x-e-o-e-o} would not be in $L^2$ unless $c_1=0$.
The case when $\Re\nu_{+}=0$,
so that $c_1=0$
%% while $c_2\ne 0$,
is treated similarly.
%% then from \eqref{Eq:SystemSandR-e-o-e-o}
%% we would obtain:
%% \[
%% S_{-}-\nu_{-}\mu=0,
%% \qquad
%% S_{-}-(1+2\kappa)\mu\nu_{-}=0,
%% \]
%% and again $\kappa\ne 0$
%% leads to $\nu_{-}=0$,
%% and then $\Psi$ could not be in $L^2$ since $c_2\ne 0$.
One concludes that there are no embedded eigenvalues
corresponding to eigenfunctions from $\bfX_{\mbox{\rm\footnotesize even-odd-even-odd}}$.
\end{proof}

Let us now study isolated eigenvalues.
We rewrite
the compatibility condition \eqref{c-cond}
as
\begin{eqnarray}\label{condition-e-o-e-o}
\Gamma(\Lambda)=0,
\end{eqnarray}
where
\begin{eqnarray}\label{def-gamma}
\Gamma(\Lambda):&=&-\nu_{-} \nu_{+}\mu^2
(2\kappa+1)
+\mu \nu_{-}(\kappa+1)(m-\omega+\Lambda)
\\
\nonumber
&&
+
(m-\omega-\Lambda)(\kappa+1)\mu \nu_{+}
-(m-\omega-\Lambda)(m-\omega+\Lambda),
\end{eqnarray}
with
$\mu=\mu(\omega)=\sqrt{(m-\omega)/(m+\omega)}$
and
$\nu_\pm=\nu_\pm(\omega,\Lambda)$ introduced in \eqref{def-m-n}.
One can see 
that $\nu_{-}(\omega,\Lambda)$
vanishes at
$\Lambda=m-\omega$ and $\Lambda=-m-\omega$
while
$\nu_{+}(\omega,\Lambda)$ vanishes at $\Lambda=m+\omega$ and at $\Lambda=-m+\omega$;
it follows that
$\Gamma(\Lambda)$ vanishes at $\Lambda=m-\omega$ and $\Lambda=-m+\omega$.

\begin{definition}
We define the first, or ``physical'', sheet of the  Riemann surface
of the function $\Gamma(\Lambda)$ to be the one where $\Re \nu_{+}\ge 0$
and $\Re\nu_{-}\ge 0$.
Below, we will call it the $(+,+)$ Riemann sheet.
\end{definition}

We now consider $\Lambda$ outside of the thresholds:
\begin{eqnarray}\label{Lambda-no-threshold}
\Lambda\not\in\left\{ -m-\omega, -m+\omega, m-\omega, m+\omega\right\},
\end{eqnarray}
so that $\nu_{-}\nu_{+}$ does not vanish. Let us find the solutions of $\Gamma(\Lambda)=0$
on the first Riemann sheet.
We divide 
\eqref{condition-e-o-e-o}
by $\nu_{-} \nu_{+}$
(this corresponds to ``normalizing'' the vectors from
\eqref{psi-x-e-o-e-o}
near $\nu_\pm\to 0$; now the resulting function will not vanish identically
near $\Lambda=m-\omega$ and $\Lambda=m+\omega$).
Taking into account the fact that $z = (\sqrt z)^2$ for all $z\in \C\backslash \R_-$,
and that $\sqrt{c z} = \sqrt c \sqrt z$ for all $c>0$ and $z\in\C\backslash \R_-$,
after some manipulations (dividing by $\mu$ and factorizing), we end up with the equation 
\begin{equation}\label{EQ1prov}
\kappa^2
=
\left(\kappa+1
-
\frac{\sqrt{1-\frac{\Lambda}{m-\omega}}}{\sqrt{1+\frac{\Lambda}{m+\omega}}}
\right)
\left(\kappa+1
-
\frac{\sqrt{1+\frac{\Lambda}{m-\omega}}}{\sqrt{1-\frac{\Lambda}{m+\omega}}}
\right).
\end{equation}
In this formula,
we choose the branch of $\sqrt z$,
$z\in \C\backslash \R_-$,
such that
$\Re \sqrt z\ge 0$. 

\begin{remark}\label{remark-root}
One has $\sqrt{z w} = \sqrt z \sqrt w$ by analytical extension from $\R_{+}$ to $z\in\C\setminus\R_{-}$ and $w\in\C\setminus\R_{-}$,
with $\arg(z)\neq \arg(w) +\pi \mod 2\pi$ for instance.
\end{remark}

%\comment{NEW:}
Let us first consider the case
$\kappa=-1$.
In this case, \eqref{EQ1prov}
leads to
$
1-\frac{\Lambda^2}{(m-\omega)^2}
=
1-\frac{\Lambda^2}{(m+\omega)^2},
$
and thus
$\kappa=-1$
corresponds to the following two cases:
\begin{enumerate}
\item
$\omega\in(-m,m)\setminus\{0\}$
and then $\Lambda=0$,
so that
$
\sigma\sb{\mathrm{p}}\big(\bfA(\omega,-1)
\at{\bfX_{\mbox{\rm\tiny even-odd-even-odd}}}\big)
=\{0\};
$
\item
$\omega=0$ and $\Lambda\in\C$ is arbitrary; 
the values corresponding to the point spectrum
are $\Lambda\in\C\setminus((-\infty,-m]\cup[m,+\infty))$,
corresponding to $\Re\nu_\pm(0,\Lambda)=\sqrt{m^2-\Lambda^2}>0$.
We conclude that
\[
\sigma\sb{\mathrm{p}}(\bfA(0,-1))
=\C\setminus\big(\jj(-\infty,-m]\cup\jj[m,+\infty)\big),
\qquad
\sigma(\bfA(0,-1))
=\sigma\sb{\mathrm{ess}}(\bfA(0,-1))
=\C.
\]
\end{enumerate}
This proves Theorem~\ref{theorem-sect2}~\itref{theorem-sect2-4-b}.

\medskip

In the rest of this subsection,
we only consider
the spectrum of the restriction of $\bfA$
onto $\bfX\sb{\mbox{\footnotesize even-odd-even-odd}}$
in the case
$\kappa\ne -1$;
together with \eqref{kappa-nonzero},
this reduces our consideration to the situation
\begin{eqnarray}\label{kappa-nonzero-1}
\kappa\in\R\setminus\{-1,\,0\}.
\end{eqnarray}

Due to Remark~\ref{remark-root}, we claim that
on the first Riemann sheet of $\Gamma(\Lambda)$ one has
\begin{equation}\label{claim}
\frac{\sqrt{1-\frac{\Lambda}{m-\omega}}}{\sqrt{1+\frac{\Lambda}{m+\omega}}} = 
\sqrt{\frac{1-\frac{\Lambda}{m-\omega}}{1+\frac{\Lambda}{m+\omega}}}
\qquad \textrm{and} \qquad 
\frac{\sqrt{1+\frac{\Lambda}{m-\omega}}}{\sqrt{1-\frac{\Lambda}{m+\omega}}} =
\sqrt{\frac{1+\frac{\Lambda}{m-\omega}}{1-\frac{\Lambda}{m+\omega}}}.
\end{equation}
Let us prove the first relation in \eqref{claim}.
Note that 
$
\frac{\sqrt{1-\frac{\Lambda}{m-\omega}}}{\sqrt{1+\frac{\Lambda}{m+\omega}}}=
\frac{\sqrt{1-\frac{\Lambda}{m-\omega}}\sqrt{1+\frac{\overline\Lambda}{m+\omega}}}{\left|\sqrt{1+\frac{\Lambda}{m+\omega}}\right|^2},
$
where we used $\overline{\sqrt{z}} = \sqrt{\overline z}$,
which holds true for all $z\in \C\backslash\R_-$.
It is enough to prove that 
\[
\sqrt{1-\frac{
\phantom{\int}
\Lambda
\phantom{\int}
}{m-\omega}}
\,
\sqrt{1+\frac{\overline\Lambda}{m+\omega}}
=
\sqrt{\left(1-\frac{\Lambda}{m-\omega}\right)\left(1+\frac{\overline\Lambda}{m+\omega}\right)}.
\]
Since 
$
\Im \left(1-\frac{\Lambda}{m-\omega} \right) \left(1+\frac{\overline\Lambda}{m+\omega}\right)
= - \frac{2 m}{m^2-\omega^2} \Im \Lambda,
$
using Remark~\ref{remark-root},
we arrive at the first relation \eqref{claim}.
Similarly, to prove the second relation in \eqref{claim}, it is enough to note that 
$
\Im\left(1+\frac{\Lambda}{m-\omega}\right) \left(1-\frac{\overline \Lambda}{m+\omega}\right)
=\frac{2 m}{m^2-\omega^2} \Im \Lambda
$
and again use Remark~\ref{remark-root}.
The conclusion is that on the first Riemann sheet of $\Gamma(\Lambda)$,
%%equation (\ref{condition-e-o-e-o}) 
equation \eqref{EQ1prov}
can be rewritten equivalently as 
\begin{equation}\label{EQ1}
\kappa^2
=
\left(\kappa+1
-
\sqrt{\frac{1-\frac{\Lambda}{m-\omega}}{1+\frac{\Lambda}{m+\omega}}}
\right)
\left(\kappa+1
-
\sqrt{\frac{1+\frac{\Lambda}{m-\omega}}{1-\frac{\Lambda}{m+\omega}}}
\right),
\qquad \Lambda \in \C. 
\end{equation}
To solve this equation, we set
\begin{equation}
\label{X-vs-Lambda}
X =
\sqrt{\frac{1-\frac{\Lambda}{m-\omega}}{1+\frac{\Lambda}{m+\omega}}},
\qquad \Re X \geq 0.
\end{equation}
Notice that $X=1$ if and only if $\Lambda=0$.
The relation \eqref{X-vs-Lambda}
leads to
$X^2 = \frac{1-\frac{\Lambda}{m-\omega}}{1+\frac{\Lambda}{m+\omega}}$,
and then
\begin{eqnarray}\label{Lambda-vs-X}
\Lambda = \frac{1-X^2}{\frac{1}{m-\omega} +\frac{X^2}{m+\omega}}.
\end{eqnarray}
For $\Lambda\neq m+\omega$ (equivalently, for $X^2\neq -\frac{\omega}{m-\omega}$),
we also have 
\[
\frac{1+\frac{\Lambda}{m-\omega}}{1-\frac{\Lambda}{m+\omega}} = \frac{m+\omega - \omega X^2}{\omega+(m-\omega)X^2},
\]
so that we rewrite equation \eqref{EQ1} as 
\begin{equation}\label{across}
\kappa^2
=
\left(\kappa+1 - X
\right)
\left(\kappa+1
-
\sqrt{\frac{m+\omega - \omega X^2}{\omega+(m-\omega)X^2}}
\right).
\end{equation}
If $\kappa+1 - X = 0$,
then $\kappa =0$ and hence $X=1$, leading to $\Lambda=0$.
Now we need to consider the case
\begin{eqnarray}\label{kappa-1-x}
\kappa+1 - X \ne 0.
\end{eqnarray}
Under this condition, the relation \eqref{across} is equivalent to
%\begin{equation}\label{EQ2}
$\kappa+1 - \frac{\kappa^2}{\kappa+1 - X}
=
\sqrt{\frac{m+\omega - \omega X^2}{\omega+(m-\omega)X^2}}\,$,
which leads to
\begin{equation}\label{EQ4}
\left(\kappa+1 - \frac{\kappa^2}{\kappa+1 - X}\right)^2
=
\frac{m+\omega - \omega X^2}{\omega+(m-\omega)X^2}.
\end{equation}
Recall that we consider the situation when the following conditions are satisfied:
$X^2\neq -\frac{\omega}{m-\omega}$, $X^2\neq -\frac{m+\omega}{m-\omega}$ (which are equivalent to
$\Lambda\neq m+\omega$ and $\Lambda\neq -m-\omega$, respectively; see \eqref{Lambda-no-threshold}),
and $ X\neq \kappa +1$ (see \eqref{kappa-1-x}).
Equation \eqref{EQ4} can be rewritten as 
%\begin{equation}\label{Eq:equation0}
$((\kappa+1)^2 - \kappa^2  -(\kappa+1)X)^2/(\kappa+1 - X)^2
=
(m+\omega - \omega X^2)/(\omega+(m-\omega)X^2)$.
%\rafd{(see SageMath file ``equation01'')} 
Taking into account that $X^2\neq -\frac{\omega}{m-\omega}$ and $ X\neq \kappa +1$,
the preceding relation can be rewritten as 
%%\begin{equation}\label{Eq:equation1}
\[
((\kappa+1)^2 - \kappa^2  -(\kappa+1)X)^2(\omega+(m-\omega)X^2)  
=
(m+\omega - \omega X^2)(\kappa+1 - X)^2,
\]
hence
\begin{equation}\label{equation}
(X-1)^2
\big(
a(\omega,\kappa) X^2 - 2b(\omega,\kappa) X - c(\omega,\kappa)
\big) = 0,
\end{equation}
with
\begin{eqnarray}\label{a-b-c}
\begin{array}{c}
a(\omega,\kappa)=
m (\kappa+1)^2 -\omega  \kappa(\kappa+ 2),
\qquad
b(\omega,\kappa)
=\kappa\big( m (\kappa+ 1) -  \omega \kappa\big),
\\[1.5ex]
c(\omega,\kappa)
= m (\kappa+1)^2 - \omega \kappa (3\kappa +2).
\end{array}
\end{eqnarray}
Denote
\begin{eqnarray}
\label{def-tau-minus}
&&
\mathcal{T}_\kappa^{-}:=\frac{(\kappa+1)^2}{\kappa(\kappa+2)}m,
\qquad
\kappa\in
(-2^{-1/2}-1,2^{-1/2}-1);
\\[0.5ex]
&&
\label{def-tau-plus}
\mathcal{T}_\kappa^{+}:=\frac{(\kappa+1)^2}{\kappa(3\kappa+2)}m,
\qquad
\kappa\in\R\setminus[-2^{-1/2},2^{-1/2}];
\\[1ex]
&&
\label{def-tau}
\mathcal{T}_\kappa=
\begin{cases}
\mathcal{T}_\kappa^{-},
\qquad
&\kappa\in
(-1,2^{-1/2}-1);
\\[1ex]
\mathcal{T}_\kappa^{+},
\qquad
&\kappa\in\R\setminus[-1,2^{-1/2}].
\end{cases}
\end{eqnarray}
We note that the intervals in
\eqref{def-tau-minus} and \eqref{def-tau-plus}
are such that the values $\mathcal{T}_\kappa^{-}$ and $\mathcal{T}_\kappa^{+}$
remain inside $(-m,m)$;
we also note that on these intervals
one has:
\[
\mathcal{T}_\kappa^{-}\le 0,
\qquad
\mathcal{T}_\kappa^{+}\ge 0.
\]
The regions of the strip $-m<\omega<m$
in the $(\kappa,\omega)$-plane
where $a$, $b$, and $c$ take particular signs
or vanish are characterized in the following lemma.

\begin{lemma}\label{lemma-abc}
Let $\kappa\in\R$, $\omega\in(-m,m)$.
We have:
\begin{itemize}
\item
%% $a(\omega,\kappa) = 0$
%% if and only if
%% $\kappa\in(-2^{-1/2}-1,2^{-1/2}-1)$,
%% $\omega=\mathcal{T}_\kappa^{-}$\textup;
%%
$a(\omega,\kappa)<0$
if and only if
$\kappa\in(-2^{-1/2}-1,2^{-1/2}-1)$,
$\omega\in(-m,\mathcal{T}_\kappa^{-})$\textup;

\item
%% $b(\omega,\kappa) = 0$
%% if and only if
%% $\kappa<-1/2$,
%% $\omega=2\Omega_\kappa$
%% or $\kappa=0$, $\omega\in(-m,m)$\textup;
%%
$b(\omega,\kappa)<0$
if and only if
$\kappa<-1/2$,
$\omega\in(2\Omega_\kappa,m)$
or
$\kappa\in[-1/2,0)$, $\omega\in(-m,m)$\textup;

\item
%% $c(\omega,\kappa) = 0$
%% if and only if
%% $\kappa\in\R\setminus[-2^{-1/2},2^{-1/2}]$,
%% $\omega=\mathcal{T}_\kappa^{+}$\textup;
%%
$c(\omega,\kappa)<0$
if and only if
$\kappa\in\R\setminus[-2^{-1/2},2^{-1/2}]$,
$\omega\in(\mathcal{T}_\kappa^{+},m)$.
\end{itemize}

\end{lemma}

The proof of Lemma~\ref{lemma-abc} follows from \eqref{a-b-c}
by inspection.
We recall that $\Omega_\kappa=\frac{\kappa+1}{2\kappa}m$ was defined in
\eqref{def-Omega-kappa}.

Equation \eqref{equation} has 
root $X_0=1$ of multiplicity two
(by \eqref{Lambda-vs-X}, it corresponds to $\Lambda=0$).

Let us consider the case $a(\omega,\kappa)=0$.
In this case, by Lemma~\ref{lemma-abc},
$\omega=\mathcal{T}_\kappa^{-}$ (and also $\kappa\ne -2$);
one has:
$$b(\mathcal{T}_\kappa^{-},\kappa) = m\frac{\kappa(\kappa+1)}{\kappa+2},
\qquad
c(\mathcal{T}_\kappa^{-},\kappa) = -2m\frac{\kappa(\kappa+1)^2}{\kappa+2}.
$$
We note that
$b(\mathcal{T}_\kappa^{-},\kappa)\ne 0$
since
$\kappa\ne -1$, $\kappa\ne 0$
(see \eqref{kappa-nonzero-1}),
and $\omega=\mathcal{T}_\kappa^{-}\ne 2\Omega_\kappa$,
hence
equation~\eqref{equation} has the root
\[
X_-(\mathcal{T}_\kappa^{-},\kappa) = -\frac{c(\mathcal{T}_\kappa^{-},\kappa) }{2b(\mathcal{T}_\kappa^{-},\kappa)}=\kappa+1.
\]
In this case, we have
$X_-(\mathcal{T}_\kappa^{-},\kappa)^2
=(\kappa+1)^2
=-\frac{\mathcal{T}_\kappa^{-}}{m-\mathcal{T}_\kappa^{-}}$;
by \eqref{Lambda-vs-X}, this gives the value
\[
\Lambda
=\frac{1-X_{-}^2}{\frac{1}{m-\omega}+\frac{X_{-}^2}{m+\omega}}
=\frac{1+\frac{\mathcal{T}^{-}_\kappa}{m-\mathcal{T}^{-}_\kappa}}
{\frac{1}{m-\mathcal{T}^{-}_\kappa}
-\frac{\mathcal{T}^{-}_\kappa}
{(m+\mathcal{T}^{-}_\kappa)(m-\mathcal{T}^{-}_\kappa)}
}
%% =\frac{\frac{m}{m-\mathcal{T}^{-}_\kappa}}
%% {\frac{m}
%% {(m+\mathcal{T}^{-}_\kappa)(m-\mathcal{T}^{-}_\kappa)}
%% }
=m+\mathcal{T}^{-}_\kappa=m+\omega,
\]
%% \[
%% =\frac{\kappa^2+2\kappa}{\frac{1}{m-\omega}+\frac{(\kappa+1)^2}{m+\omega}}
%% =
%% \frac{(m^2-\omega^2)(\kappa^2+2\kappa)}{(m+\omega)+(m-\omega)(\kappa+1)^2}
%% =
%% \frac{(m^2-\omega^2)(\kappa^2+2\kappa)}
%% {(m+\mathcal{T}^{-}_\kappa)+(m-\mathcal{T}^{-}_\kappa)(\kappa+1)^2}
%% \]
%% \[
%% =
%% (m+\omega)\frac{(m-\mathcal{T}^{-}_\kappa)(\kappa^2+2\kappa)}
%% {(m+\mathcal{T}^{-}_\kappa)+(m-\mathcal{T}^{-}_\kappa)(\kappa+1)^2}
%% =
%% (m+\omega)\frac{-m}
%% {
%% m\frac{\kappa(2+\kappa)+(\kappa+1)^2}{\kappa(2+\kappa)}
%% -\frac{m}{\kappa(2+\kappa)}
%% (\kappa+1)^2
%% }
%% \]
%% \[
%% =
%% (m+\omega)\frac{-m}
%% {
%% m\frac{2\kappa^2+4\kappa+1}{\kappa(2+\kappa)}
%% -\frac{m}{\kappa(2+\kappa)}
%% (\kappa+1)^2
%% }
%% \]
%% $\Lambda=m+\omega$,
which corresponds to a threshold and
which we do not consider (see \eqref{Lambda-no-threshold}).

Thus, we can assume that
$a(\omega,\kappa) \ne 0$.
In this case, besides root $X_0=1$, equation~\eqref{equation} has the roots
\begin{equation}\label{EQ3}
X_{\pm}(\omega,\kappa) = \frac{b(\omega,\kappa) \pm \sqrt{b^2(\omega,\kappa) + a(\omega,\kappa) c(\omega,\kappa) }}{ a(\omega,\kappa)},
\qquad \Re \sqrt{b^2(\omega,\kappa) + a(\omega,\kappa) c(\omega,\kappa) }\geq 0.
\end{equation}
We need
to make sure that the values of the parameters the functions
$X_+(\omega,\kappa) $ and $X_-(\omega,\kappa) $ are
\emph{admissible solutions},
in the sense that they correspond to either eigenvalues
or virtual levels of the linearized operator.
To simplify the reasoning,
we note that if $\Lambda$ is a solution of the original equation \eqref{EQ1},
then so is $-\Lambda$. This symmetry has a counterpart in terms of
the variable $X$:
if $X$, with $\Re X\geq 0$,
is a solution of equation \eqref{across}, then so is 
\begin{equation}\label{mountains}
Y = \sqrt{\frac{m+\omega - \omega X^2}{\omega+(m-\omega)X^2}},
\qquad 
\Re Y\ge 0;
\end{equation}
the same formula expresses $X$ in terms of $Y$.
We claim that nonzero values of
$\Lambda$ correspond to $X\ne Y$. Indeed, if $X=Y$, then
the relation
$X^2=\frac{m+\omega - \omega X^2}{\omega+(m-\omega)X^2}$
implies that
\[
(m-\omega)X^4+2\omega X^2
-m-\omega=0,
\qquad
X^2=\frac{-\omega\pm\sqrt{\omega^2+m^2-\omega^2}}{m-\omega}
=\frac{-\omega\pm m}{m-\omega}.
\]
Substituting
$X^2=1$ into \eqref{Lambda-vs-X} we see that it corresponds to $\Lambda=0$,
while $X^2=\frac{-m-\omega}{m-\omega}$
does not correspond to a finite value of $\Lambda$.
Thus,
the functions $X_+(\omega,\kappa)$ and $X_-(\omega,\kappa)$
which correspond to nonzero values $\Lambda$
are conjugated by \eqref{mountains},
satisfying the following relations:
\begin{eqnarray}\label{for-1}
&
X_+^2 = \frac{m+\omega - \omega X_-^2}{\omega+(m-\omega)X_-^2},
\qquad
X_-^2 = \frac{m+\omega - \omega X_+^2}{\omega+(m-\omega)X_+^2};
\\[1ex]
\label{for-2}
&
\Re X_{+}\ge 0,
\qquad
\Re X_{-}\ge 0.
\end{eqnarray}
By \eqref{EQ3},
for all $\kappa\in\R$ and $\omega\in(-m,m)$,
as long as $a(\omega,\kappa)\ne 0$,
there are the relations
\begin{equation}\label{Eq:Dune2}
X_++X_- =2b/a,
\qquad X_+X_- = - c/a.  
\end{equation}
By \eqref{EQ3} and \eqref{for-2},
the roots $X_\pm$
corresponding to nonzero eigenvalues or thresholds
are either both real and nonnegative,
or are mutually complex conjugate with nonnegative real part.
This takes place
in the following two cases:
\begin{eqnarray}
\mbox{\it either}\qquad
a(\omega,\kappa)>0, \quad b(\omega,\kappa)\ge 0,
\quad c(\omega,\kappa)\le 0
\label{case-1}
\\
\mbox{\it or}\qquad
a(\omega,\kappa)<0, \quad b(\omega,\kappa)\le 0,
\quad c(\omega,\kappa)\ge 0.
\label{case-2}
\end{eqnarray}
The corresponding values of parameters
in the $(\kappa,\omega)$-plane can be found with the aid
of Lemma~\ref{lemma-abc}.
The case \eqref{case-1} corresponds to the values $\kappa<-1$,
$0<\mathcal{T}_\kappa<\omega<2\Omega_\kappa$
and $\kappa>2^{-1/2}$, $0<\mathcal{T}_\kappa<\omega<m$;
the case \eqref{case-2}
corresponds to $-1<\kappa<2^{-1/2}-1$, $\max(-m,2\Omega_\kappa)<\omega<\mathcal{T}_\kappa<0$ (these are the shaded regions on Figure~\ref{fig-omegak}).
For these values of $\kappa$ and $\omega$,
equation \eqref{across} has not only root $X_0=1$
(corresponding to the eigenvalue $\lambda_0 =\jj\Lambda_0 = 0$),
but also the the roots
$X_+(\omega,\kappa)$ and  $X_-(\omega,\kappa)$
which lead to eigenvalues
$\lambda_{+}=\jj\Lambda_{+}$ and $\lambda_{-}=\jj\Lambda_{-}=-\lambda_{+}$,
with
\begin{eqnarray}\label{lambda-pm}
\Lambda_+ = \frac{1-X_+^2}{\frac{1}{m-\omega} +\frac{X_+^2}{m+\omega}},  \qquad 
\Lambda_- = \frac{1-X_-^2}{\frac{1}{m-\omega} +\frac{X_-^2}{m+\omega}}=-\Lambda_+.
\end{eqnarray}

\begin{remark}\label{remark-continuous}
One can determine for which values of $\kappa$
and $\omega$ the eigenvalues $\lambda=\pm\jj\Lambda$
are continuous functions of these parameters.
By \eqref{EQ3},
$X_{+}$ depends continuously on $\kappa$ and $\omega$
as long as $a(\omega,\kappa)\ne 0$;
that is, away from the set
$\omega=\mathcal{T}^{-}_\kappa$
(defined in \eqref{def-tau-minus}).
For $\omega\ne\mathcal{T}^{-}_\kappa$,
by \eqref{lambda-pm},
$\Lambda_{+}$ is not a continuous function
of $\kappa$ and $\omega$
when
$X_{+}^2=-(m+\omega)/(m-\omega)$.
This implies that $X_{+}$ is purely imaginary;
by \eqref{EQ3},
this means that $b(\omega,\kappa)=0$
and thus $\omega=2\Omega_\kappa$;
see \eqref{def-Omega-kappa}.
Thus,
the dependence of eigenvalues
$\lambda=\pm\jj\Lambda_{+}$
on $\omega,\,\kappa$
is continuous
except perhaps at the curves
$\omega=\mathcal{T}^{-}_\kappa$
and $\omega=2\Omega_\kappa$.
\end{remark}

Due to Theorem~\ref{theorem-sect2}~\itref{theorem-sect2-1},
the values $\Lambda_+$ and $\Lambda_- = - \Lambda_+$
in \eqref{lambda-pm}
are either real or purely imaginary.
Let us derive an explicit expression for $\Lambda_{+}$.
One has:
\[\begin{aligned}
\Lambda_+ =&\frac12 \left( \frac{1-X_+^2}{\frac{1}{m-\omega} +\frac{X_+^2}{m+\omega}} - \frac{1-X_-^2}{\frac{1}{m-\omega} +\frac{X_-^2}{m+\omega}} \right) = 
- \frac12
\frac{ \frac{X_+^2-X_-^2}{m+\omega}  +\frac{X_+^2-X_-^2}{m-\omega}}{\frac{1}{(m-\omega)^2} + \frac{X_+^2+X_-^2}{m^2-\omega^2} + \frac{X_+^2X_-^2}{(m+\omega)^2} }
\\=&- \frac{m}{m^2-\omega^2} 
\frac{ X_+^2-X_-^2}{\frac{1}{(m-\omega)^2} + \frac{1}{m^2-\omega^2}  \frac{1}{\omega} (m+\omega - (m-\omega) X_+^2X_-^2)  + \frac{X_+^2X_-^2}{(m+\omega)^2} };
\end{aligned}
\]
we used the identity
$\omega(X_+^2+X_-^2) = (m+\omega) - (m-\omega)X_+^2X_-^2$
which follows from \eqref{for-1}.
Using \eqref{EQ3}, we derive:
\begin{eqnarray}\label{lambda-plus}
\Lambda_+ = - \frac{4 b\, \omega  \sqrt{b^2 + ac}}{\frac{m+\omega}{m-\omega} a^2 - \frac{m-\omega}{m+\omega} c^2},
\end{eqnarray}
with $a=a(\omega,\kappa)$, $b=b(\omega,\kappa)$, and $c=c(\omega,\kappa)$
from \eqref{a-b-c}.
Taking into account that 
\begin{equation}\label{b2+ac}
b^2(\omega,\kappa) + a(\omega,\kappa) c(\omega,\kappa)
=(\kappa+1)\big(m(\kappa+1)-2\kappa\omega\big)
\big(m(2\kappa^2+2\kappa+1)-2\kappa(\kappa+1)\omega\big)
\end{equation}
and noticing that for $\omega\in[-m,m]$
one has  $\R\ni m(2\kappa^2+2\kappa +1)-2\kappa(\kappa+1)\omega>0\,$,
one derives:
\begin{align}\label{lambda-plus-minus}
\Lambda_\pm
=\pm\Lambda_{+}
&=
\pm
\frac{m^2-\omega^2}{2\Omega_\kappa-\omega}\frac{\sqrt{2\kappa(\kappa+1)(\Omega_\kappa -\omega)(m(2\kappa^2+2\kappa +1)-2\kappa(\kappa+1)\omega)}}
{m(2\kappa^2+2\kappa +1)-2\kappa(\kappa+1)\omega}
\nonumber
\\
&=\pm\frac{m^2-\omega^2}{2\Omega_\kappa-\omega}\sqrt{\frac{2\kappa(\kappa+1)(\Omega_\kappa -\omega)}
{m(2\kappa^2+2\kappa +1)-2\kappa(\kappa+1)\omega}}
=\pm\frac{m^2-\omega^2}{2\Omega_\kappa-\omega}\sqrt{\frac{\Omega_\kappa -\omega}{W_\kappa-\omega}},
\end{align}
where $\Omega_\kappa$ is from \eqref{def-Omega-kappa}
and
\begin{equation}\label{def-W-kappa}
W_\kappa:=m\frac{2\kappa^2+2\kappa+1}{2\kappa(\kappa+1)}.
\end{equation}
Above, we factored out $\kappa(\kappa+1)$,
which is nonzero due to \eqref{kappa-nonzero-1}.
Taking into account that $\mathcal{T}_\kappa^{+} \leq W_\kappa$,
it then follows that for $\mathcal{T}_\kappa^{+} \leq \omega \leq \Omega_\kappa $
the values $\Lambda_\pm$ from \eqref{lambda-plus-minus} are real
(hence the corresponding eigenvalues $\lambda_\pm=\jj\Lambda_\pm$ are purely imaginary),
while for $\omega >\Omega_\kappa$ they are purely imaginary
(with the corresponding
eigenvalues $\lambda_\pm=\jj\Lambda_\pm$ being real).
For $\omega=\Omega_\kappa$, the two eigenvalues coincide and are both equal to zero, and $\lambda =0$ is an eigenvalue with total algebraic  multiplicity four. Notice also that $\Lambda_\pm$ are going to infinity as $\omega$ approaches $2\Omega_\kappa$ if $\kappa<-1$.

Since $W_\kappa<-m$ for $\kappa\in(-1,0)$ and $W_\kappa>m$
for $\kappa\in(-\infty,-1)\cup(0,+\infty)$,
while $\Omega_\kappa\in(-m,m)$ if and only if
$\kappa\not\in[-1/3,1]$,
we can summarize the location of eigenvalues as follows:
\begin{itemize}
\item
For $\kappa\in(-\infty,-1)\cup(1,\infty)$,
one has:\footnote{We note that for $\kappa=-1/2$, one has $W_\kappa=-m$;
for $\kappa=1$, one has $\Omega_\kappa=m$.}
\quad
$\ds
\begin{cases}
\Lambda_\pm\in\R\setminus\{0\},&\omega\in(-m,\Omega_\kappa);
\\
\Lambda_\pm=0,&\omega=\Omega_\kappa;
\\
\Lambda_\pm\in\jj\R\setminus\{0\},&\omega\in(\Omega_\kappa,m).
\end{cases}
$

\item For $\kappa\in\{-1,0\}$ and any $\omega\in(-m,m)$,
one has $\Lambda_\pm=0$.

\item For $\kappa\in(-1,-1/3)$,\footnote{For $\kappa=-1/3$,
one has $\Omega_\kappa=-m$.}
one has:
\quad
$\ds
\begin{cases}
\Lambda_\pm\in\jj\R\setminus\{0\},&\omega\in(-m,\Omega_\kappa);
\\
\Lambda_\pm=0,&\omega=\Omega_\kappa;
\\
\Lambda_\pm\in\R\setminus\{0\},&\omega\in(\Omega_\kappa,m).
\\
\end{cases}
$

\item
For $\kappa\in[-1/3,0)$ and any $\omega\in(-m,m)$,
one has $\Lambda_\pm\in\jj\R\setminus\{0\}$.

\item
For $\kappa\in(0,1]$ and any $\omega\in(-m,m)$,
one has $\Lambda_\pm\in\R\setminus\{0\}$.
\end{itemize}

The analysis just given covers the cases of Theorem~\ref{theorem-sect2}~\itref{theorem-sect2-4-a}
and
Theorem~\ref{theorem-sect2}~\itref{theorem-sect2-4-f}.
All the remaining cases can be treated exactly the same way.
The explicit expression of eigenvalues \eqref{lambda-plus-minus}
remains the same; the ranges of $\kappa$ and $\omega$
are determined as before from the requirement that eigenvalues belong to the first Riemann sheet of $\Gamma(\Lambda)$
and that they are in $\jj(-m+\abs{\omega},m-\abs{\omega})$.

Once we know that there are exactly two zeros of the function $\Gamma(\Lambda)$
which could correspond to eigenvalues,
and therefore located on the real or on the imaginary axis
(let us mention that
the spectrum of $\bfA$
remains symmetric with respect to $\R$ and $\jj\R$
even after its restriction
onto $\bfX_{\mbox{\rm\footnotesize even-odd-even-odd}}$),
we can use
a simpler argument to locate the eigenvalues and threshold resonances. 
Due to continuous dependence of eigenvalues
on the parameters $\omega$ and $\kappa$
(except perhaps at $\omega=\mathcal{T}^{-}_\kappa$
and $\omega=2\Omega_\kappa$; see
Remark~\ref{remark-continuous}),
we know that for each
$\kappa\in\R\setminus(\{-1\}\cup[-1/3,1])$
at $\omega=\Omega_\kappa\in(-m,m)$
there is a collision of two eigenvalues
since the dimension of the generalized null space of $\bfA$
jumps at this value of $\omega$.
All we need to do is to find when these eigenvalues disappear
from the spectral gap
$\jj(-m+\abs{\omega},m-\abs{\omega})$;
that is, when the points
$\lambda=\pm\jj(m-\abs{\omega})$
become threshold eigenvalues or virtual levels.
Moreover, if these points become virtual levels, this means that
the corresponding $\nu_\pm$ changes the sign at this point,
so the zeros of $\Gamma(\Lambda)$ move onto
one of the unphysical sheets of its Riemann surface
(where at least one of $\Re\nu_{+}(\omega,\Lambda)$,
$\Re\nu_{-}(\omega,\Lambda)$ is negative),
becoming resonances (corresponding to antibound states).

\begin{lemma}\label{lemma-t-kappa}
The restriction of $\bfA(\omega,\kappa)$
onto $\bfX_{\mbox{\rm\footnotesize even-odd-even-odd}}$
has virtual levels at the thresholds of the essential spectrum
$\lambda=\pm\jj(m-\abs{\omega})$
at the following values of
$\omega\in(-m,m)\setminus\{0\}$
and $\kappa\in\R$:
\begin{enumerate}
\item
For $\kappa<-1$
or $\kappa>2^{-1/2}$,
there are virtual levels at $\lambda=\pm\jj(m-\omega)$
when $\omega=\mathcal{T}_\kappa^{+}:=\frac{(k+1)^2 m}{(3\kappa+2)\kappa}>0$.
\item
For $-1<\kappa<2^{-1/2}-1$,
there are virtual levels at $\lambda=\pm\jj(m+\omega)$
when $\omega=\mathcal{T}_\kappa^{-}:=\frac{(k+1)^2 m}{(\kappa+2)\kappa}<0$.
\end{enumerate}
\end{lemma}

\begin{proof}
We first consider the case $\omega>0$.
Let us find when an imaginary eigenvalue touches the
essential spectrum at the endpoint
$\lambda=\jj(m-\omega)$.
(By Lemma~\ref{lemma-no},
the endpoints never correspond to eigenvalues.)
One needs
\[
\kappa^2
=
(\kappa+1)
\left(\kappa+1
-
\sqrt{\frac{m+\omega}{m-\omega}}
\sqrt{2\frac{m-\omega}{2\omega}}
\right)
=
(\kappa+1)
\left(\kappa+1
-
\sqrt{\frac{m+\omega}{\omega}}
\right).
\]
That is,
\begin{eqnarray}\label{bad-kappa}
2\kappa+1
=
(\kappa+1)
\sqrt{\frac{m+\omega}{\omega}}
,
\qquad
\frac{2\kappa+1}{\kappa+1}
=
\sqrt{1+\frac{m}{\omega}}.
\end{eqnarray}
We point out that if $\kappa\in(-1,-1/2)$,
the fraction on the left is strictly negative
while the square root is nonnegative; these values of $\kappa$
cannot correspond to virtual levels at the threshold
$\jj(m-\omega)$ with $\omega>0$.
The condition to have a virtual level or an eigenvalue
at some value $0<\omega<m$
takes the form
\[
\frac{2\kappa+1}{\kappa+1}>\sqrt{2},
\]
which leads to $\kappa>2^{-1/2}$.
Let us compute the value of $\omega$
corresponding to a virtual level:
\[
\frac{m}{\omega}=
\frac{(2\kappa+1)^2}{(\kappa+1)^2}-1
=\frac{(3\kappa+2)\kappa}{(\kappa+1)^2},
\]
hence the critical value of $\omega$ which corresponds to virtual levels
at the thresholds $\lambda=\pm \jj(m-\omega)$ is given by
$\omega=\mathcal{T}_\kappa^{+}$ from \eqref{def-tau-plus}.
We point out that these critical values correspond to the collision of eigenvalues
with the threshold points only if $\kappa<-1$ and $\kappa>2^{-1/2}$.

Now we consider the case $\omega<0$.
To find the value of $\omega$ which corresponds
to a bifurcation of an eigenvalue from the threshold
of the essential spectrum (that is, when there is
a virtual level or an eigenvalue at the threshold
$\lambda=\jj(m-\abs{\omega})=\jj(m+\omega)$),
we consider the equation $\Gamma(\Lambda)=0$,
with $\Lambda=m+\omega$
and with $\Gamma(\Lambda)$ defined in \eqref{def-gamma}.
Taking into account that $\nu_{+}(\omega,\Lambda)=0$
at $\Lambda=m+\omega$,
we are to solve the equation
$0=\Gamma(m+\omega)=2(\kappa+1)m\mu\nu_{-}+4m\omega$.
This leads to
$(\kappa+1)\mu\nu_{-}(\omega,\Lambda)=-2\omega$.
Since $\omega<0$,
we conclude that
the eigenvalue touches the threshold
$\lambda=\jj(m+\omega)$
(that is, there is a virtual level at threshold)
if $\kappa>-1$.
Squaring the above equation and substituting $\nu_{-}$, we arrive at
\[
(\kappa+1)^2\frac{m-\omega}{m+\omega}
(-4\omega m-4\omega^2)=4\omega^2,
\]
hence
$
(\kappa+1)^2(m-\omega)=-\omega,
$
which leads to the critical values
$\omega=\mathcal{T}^{-}_\kappa$ from \eqref{def-tau-minus}.
\end{proof}

\begin{lemma}\label{lemma-w-kappa}
One has
$\lambda=\jj\Lambda_\pm\to\pm\infty$
as
$\ds
\omega\to
\begin{cases}
2\Omega_\kappa-0,&\kappa<-1,
\\
2\Omega_\kappa+0,&-1<\kappa\le -1/2.
\end{cases}
$
\quad
\end{lemma}

\begin{proof}
The statement of the lemma
follows from \eqref{lambda-plus-minus},
\[
\Lambda_\pm
=\pm\frac{m^2-\omega^2}{2\Omega_\kappa-\omega}
\sqrt{\frac{\Omega_\kappa-\omega}{W_\kappa-\omega}}
=\pm\jj\frac{m^2-\omega^2}{2\Omega_\kappa-\omega}
\sqrt{\frac{\omega-\Omega_\kappa}{W_\kappa-\omega}}.
\]
We note that
$\Omega_\kappa$ from \eqref{def-Omega-kappa}
and $W_\kappa$ from \eqref{def-W-kappa}
satisfy
\[
\begin{cases}
0<\Omega_\kappa<2\Omega_\kappa<m,\qquad W_\kappa>m, &\kappa<-1,
\\[0.5ex]
-m\le 2\Omega_\kappa<\Omega_\kappa<0,\qquad W_\kappa\le -m,& -1<\kappa\le -1/2,
\end{cases}
\]
so that
$\lambda=\jj\Lambda_{\pm}$
are real and approach $\pm\infty$
as
$\omega\to 2\Omega_\kappa$
(cf. Remark~\ref{remark-continuous}).
\end{proof}

Lemmata~\ref{lemma-t-kappa} and~\ref{lemma-w-kappa}
complete the proof of
Theorem~\ref{theorem-sect2}~\itref{theorem-sect2-3} and \itref{theorem-sect2-4}.

\subsection{Multiplicity of zero eigenvalue and the Kolokolov condition}\label{Sec:Kolokolov}

Finally, let us prove
Theorem~\ref{theorem-sect2}~\itref{theorem-sect2-6}.
Let us consider
$\bfA(\omega,\kappa)$
in the invariant subspace
$\bfX\sb{\mbox{\rm\footnotesize even-odd-even-odd}}$
(see \eqref{def-eoeo}).
We notice that for the restriction of $\bfA(\omega,\kappa)$
onto this subspace
one has:
\begin{eqnarray}\label{a-jordan}
\bfA(\omega,\kappa)
\begin{bmatrix}0\\\phi_\omega\end{bmatrix}= \begin{bmatrix}0\\0\end{bmatrix},
\qquad
\bfA(\omega,\kappa)
\begin{bmatrix}\p_\omega\phi_\omega\\0\end{bmatrix}=\begin{bmatrix}0\\\phi_\omega \end{bmatrix}.
\end{eqnarray}

\begin{remark}
Let us verify the second relation in \eqref{a-jordan}.
First we have to check that $\p_\omega\phi_\omega\in \dom(L_+)$.
Using \eqref{solitary-wave}, we compute:
\[
\p_\omega\phi_\omega(x)
=\p_\omega
\left(
\alpha\begin{bmatrix}1\\\mu\sgn x\end{bmatrix}e^{-\varkappa(\omega)\abs{x}}
\right)
=
\left(
\begin{bmatrix}\p_\omega\alpha\\\p_\omega(\alpha\mu)\sgn x\end{bmatrix}
-\alpha\begin{bmatrix}1\\\mu\sgn x\end{bmatrix}\abs{x}\p_\omega\varkappa(\omega)
\right)
e^{-\varkappa(\omega)\abs{x}}.
\]
Notice that 
the second term in the brackets in the right-hand side
of the above
does not contribute to the validity of the boundary condition because it is continuous and vanishes at the origin,
while the contribution of the first term
yields the following:
\begin{eqnarray}
\label{dn}
\widehat{\p_\omega\phi_\omega}
=\begin{bmatrix}\p_\omega\alpha
\\ 0\end{bmatrix}
=\frac{\alpha}{2\kappa\mu}\p_\omega\mu\begin{bmatrix}1\\0\end{bmatrix},
\qquad
\left[\p_\omega\phi_\omega \right]_0
=\begin{bmatrix}0\\ 2\p_\omega(\alpha\mu)\end{bmatrix}
=\frac{\alpha(2\kappa+1)}{\kappa}\p_\omega\mu\begin{bmatrix}0\\1\end{bmatrix};
\end{eqnarray}
we used
\eqref{p-alpha-kappa}.
From these relations it follows that
\begin{align}
\jj\sigma_2[\p_\omega\phi_\omega]_0-2\mu(\sigma_3+2\kappa\Pi_1)
\widehat{\p_\omega\phi_\omega}
=\Big(
\alpha\frac{2\kappa+1}{\kappa}-\alpha\frac{2\kappa+1}{\kappa}
\Big)
\p_\omega\mu\begin{bmatrix}1\\ 0\end{bmatrix}
=\begin{bmatrix}0\\ 0\end{bmatrix}.
\end{align}
The previous relation entails that $\p_\omega\phi_\omega\in \dom(L_+)$
$\forall \kappa\neq 0$.
Given this premise, considering separately $x>0$ and $x<0$,
one directly verifies that
$L_+\phi_\omega=(D_m-\omega I_2)\p_\omega\phi_\omega=\phi_\omega$.
This ends the verification of the second relation in \eqref{a-jordan}.
\end{remark}

\begin{remark}\label{remark-nonzero}
We point out that
for $\p_\omega\phi_\omega$ to make sense,
one needs $\alpha$ to be differentiable with respect to $\omega$;
for this, as one can see from \eqref{p-alpha-kappa},
the condition $\kappa\ne 0$
imposed in \eqref{kappa-nonzero} is required;
see also Remark~\ref{remark-bad}.
\end{remark}

By \eqref{a-jordan},
we already know that
the generalized null space of $\bfA$ is at least
two-dimensional.
Whether there are more elements in the generalized null space of $\bfA$,
depends on whether there is a solution
$\theta\in L^2(\R,\C^2)$ to
\begin{eqnarray}\label{ltp}
L_{-}\theta=\p_\omega\phi_\omega,
\end{eqnarray}
so that
$\bfA(\omega,\kappa)\begin{bmatrix}0\\\theta
\end{bmatrix}=\begin{bmatrix}\p_\omega\phi_\omega\\0\end{bmatrix}$.
Since the range of $L_{-}$
is closed,
there is a solution to \eqref{ltp}
if and only if
its right-hand side,
$\p_\omega\phi_\omega$, is orthogonal to $\ker(L_{-}^*)=\ker(L_{-})$.
By Lemma~\ref{lemma-l},
the kernel of $L_{-}(\omega)=L(\omega,0)$
on
$\bfX\sb{\mbox{\rm\footnotesize odd-even}}$
is zero (since $\omega\ne 0$),
while its kernel on
$\bfX\sb{\mbox{\rm\footnotesize even-odd}}$
is spanned
by $\phi_\omega$.
Thus, the condition to have a solution to \eqref{ltp}
is given by
\[
\langle\phi_\omega,\p_\omega\phi_\omega\rangle
=\frac{1}{2}\p_\omega Q(\phi_\omega)=0.
\]
We conclude that whether there are more elements
in the generalized null space of $\bfA$,
depends on the Kolokolov condition
$\p_\omega Q(\phi_\omega)=0$ \cite{kolokolov-1973};
this condition gives the value of the threshold
$\omega=\Omega_\kappa$ at which the dimension
of the generalized null space
$\mathfrak{L}(\bfA(\omega,\kappa))$
changes.
Let us compute $\p_\omega Q(\phi_\omega)$.
For the
$L^2$-norm of a solitary wave profile
$\phi_\omega$
from \eqref{solitary-wave},
we have:
\[
Q(\phi_\omega)
=\alpha^2(1+\mu^2)
\int_{\R}
e^{-2\varkappa\abs{x}}\,dx
=\frac{\alpha^2(1+\mu^2)}{\varkappa}.
\]
Using the relations
\begin{eqnarray}\label{p-of}
\p_\omega\alpha=\frac{\alpha\p_\omega\mu}{2\kappa\mu},
\qquad
\p_\omega\varkappa=-\frac{\omega}{\varkappa},
\qquad
\p_\omega\mu=-\frac{m}{
(m+\omega)\varkappa
}
\end{eqnarray}
(see \eqref{def-def} and \eqref{p-alpha-kappa}),
we derive:
\begin{eqnarray}\label{q-prime}
&&\hskip -54pt
\p_\omega Q(\phi_\omega)
=
\frac{
2\alpha(1+\mu^2)\varkappa\p_\omega\alpha
+
2\alpha^2\varkappa\mu\p_\omega\mu
-
\alpha^2(1+\mu^2)\p_\omega\varkappa
}{\varkappa^2}
\nonumber
\\
&&
=
\frac{2m\alpha^2}{(m+\omega)\varkappa^2}
\Big(
-
\frac{m}{\varkappa \kappa}
-
\mu
+
\frac{\omega}{\varkappa}
\Big)
=
\frac{2m\alpha^2}{(m+\omega)\varkappa^3}
\Big(
-
\frac{m}{\kappa}
-
m+2\omega
\Big).
\end{eqnarray}
Thus, we reduce the Kolokolov condition
$\p_\omega Q(\phi_\omega)=0$
to the form
\begin{eqnarray}\label{omega-m-kappa}
\frac{\omega}{m}=\frac{1+\kappa}{2\kappa}.
\end{eqnarray}
We use this relation
to define the critical value $\Omega_\kappa$ in \eqref{def-Omega-kappa}
corresponding to the critical point of $Q(\phi_\omega)$.
We point out that there is no critical value
$\omega\in(-m,m)$ of $Q(\phi_\omega)$
for
$-1/3\le\kappa\le 1$
since in this case \eqref{def-Omega-kappa}
yields $\abs{\Omega_\kappa}\ge m$.

\begin{remark}
Let us point out that,
in the context of the nonlinear Dirac equation,
the sign of $\p_\omega Q(\phi_\omega)$
is not directly related to the spectral stability:
by Theorem~\ref{theorem-sect2}~\itref{theorem-sect2-4},
the spectral regions correspond to
$\omega<\Omega_\kappa$
for $\kappa<-1$ and $\kappa>2^{-1/2}-1$
(hence $\p_\omega Q(\phi_\omega)<0$ by \eqref{q-prime})
and to
$\omega>\Omega_\kappa$ for $-1<\kappa<2^{-1/2}-1$
(hence $\p_\omega Q(\phi_\omega)>0$).
\end{remark}

\begin{remark}
One may check explicitly that indeed
for $\kappa\ne 0$
the equation $L_{-}\theta=\p_\omega\phi_\omega$ has a solution $\theta\in L^2(\R,\C^2)$
if and only if $\p_\omega Q=0$.
We have:
\[
\p_\omega\phi_\omega
=\p_\omega
\left(
\alpha\begin{bmatrix}1\\\mu\sgn x\end{bmatrix}e^{-\varkappa\abs{x}}
\right)
=
\left(
\begin{bmatrix}\p_\omega\alpha\\\p_\omega(\alpha\mu)\sgn x\end{bmatrix}
-\alpha\begin{bmatrix}1\\\mu\sgn x\end{bmatrix}\abs{x}\p_\omega\varkappa
\right)
e^{-\varkappa\abs{x}}.
\]
The solution to
\begin{eqnarray}\label{l-theta}
\begin{bmatrix}
m-\omega&\p_x
\\-\p_x&-m-\omega
\end{bmatrix}
\theta(x)
=\p_\omega\phi_\omega(x)
\end{eqnarray}
has the form
\begin{eqnarray}\label{theta}
\theta(x)=
\begin{bmatrix}
B\abs{x}+C\abs{x}^2
\\
(E\abs{x}+F\abs{x}^2)\sgn x
\end{bmatrix}
e^{-\varkappa\abs{x}},
\qquad
x\in\R,
\qquad
B,\,C,\,E,\,F\in\C;
\end{eqnarray}
above, we took into account that
the operator $D_m-\omega I_2$ is invariant in
the subspace of functions with even first component
and odd second component.
There is no jump condition to worry about since
$\theta(x)$ vanishes at the origin.
Because of the spatial symmetry,
it suffices to consider $x>0$.
Substitution of \eqref{theta}
into \eqref{l-theta} leads to the system
\[
\begin{cases}
(m-\omega)(B x+C x^2)+E+2F x
-\varkappa(E x+F x^2)=\p_\omega\alpha-\alpha\p_\omega\varkappa x,
\\
-(B+2C x)+\varkappa(Bx+C x^2)-(m+\omega)(E x+F x^2)
=\p_\omega(\alpha\mu)
-\alpha\p_\omega\varkappa \mu x,
\end{cases}
\qquad
x>0.
\]
The above system allows us to express
$E=\p_\omega\alpha$
and
$F=\mu C$
(from the first equation)
and also
$B=-\p_\omega(\alpha\mu)$
(from the second equation),
and then we derive an overdetermined system
\[
\begin{cases}
-(m-\omega)\p_\omega(\alpha\mu)+2\mu C-\varkappa\p_\omega\alpha
=-\alpha\p_\omega\varkappa,
\\
-2C-\varkappa\p_\omega(\alpha\mu)-(m+\omega)\p_\omega\alpha
=-\alpha\mu\p_\omega\varkappa
\end{cases}
\]
with the only unknown $C\in\C$.
This system yields the compatibility condition
\[
2\varkappa\p_\omega\alpha
+(m-\omega)\alpha\p_\omega\mu-\frac{\alpha m}{m+\omega}\p_\omega\varkappa=0.
\]
Substituting the expression for
$\p_\omega \alpha$ from \eqref{p-alpha-kappa}
(note that
it is for the finiteness of $\p_\omega\alpha$
that we needed the condition $\kappa\ne 0$)
and using the relations \eqref{p-of},
one again arrives at \eqref{omega-m-kappa}.
\end{remark}
Notice that 
$
\begin{bmatrix}0\\\theta
\end{bmatrix}
$
is orthogonal to 
$
\begin{bmatrix}\phi_\omega\\0\end{bmatrix}
$
and hence the Jordan chain can be continued. Hence the algebraic multiplicity of $0$ jumps at least by $2$ when \eqref{omega-m-kappa} is satisfied.
As the matter of fact, it is exactly two,
since, as we have seen
in the proof of
of Theorem~\ref{theorem-sect2}~\itref{theorem-sect2-4},
when \eqref{omega-m-kappa} is not satisfied,
there are at most two
%% simple
nonzero eigenvalues
of the restriction of
$\bfA(\omega,\kappa)$
onto
$\bfX\sb{\mbox{\rm\footnotesize even-odd-even-odd}}$.
As long as $a(\omega,\kappa)\ne 0$
(see \eqref{a-b-c}),
these eigenvalues
are locally continuous functions of parameters,
moving to $\pm\infty$ along the real axis
as $X$
defined in \eqref{X-vs-Lambda}
approaches 
$\pm \jj\sqrt{\frac{m+\omega}{m-\omega}}$
(cf. \eqref{Lambda-vs-X})
or equivalently
as $\omega$ approaches $2\Omega_\kappa$
(see Theorem~\ref{theorem-sect2}~\itref{theorem-sect2-4}).

We note that if
\eqref{omega-m-kappa} is satisfied
(that is, if $\omega=\Omega_\kappa$),
then, taking into account that $\kappa\ne 0$,
we see that the function
$a(\omega,\kappa)$ from \eqref{a-b-c}
takes the following form:
\[
a(\omega,\kappa) = m (\kappa+1)^2 -\omega  \kappa(\kappa+ 2)
=m(1+\kappa)^2
-m(1+\kappa)(2+\kappa)/2
=m(\kappa+\kappa^2)/2.
\]
For $\kappa\in\R\setminus\{0\}$,
the function
$a(\omega,\kappa)$
vanishes only when
$\kappa=-1$
(then $\omega=0$ by \eqref{omega-m-kappa});
so, outside of the point
$(\omega,\kappa)=(0,-1)$,
the value of
$X$
in \eqref{X-vs-Lambda}
is a continuous function of $\omega$ and $\kappa$
in an open neighborhood of the curve $\omega=\Omega_\kappa$.
If we consider $\Lambda$
in the disc $\mathbb{D}_\delta$ of some fixed radius $\delta>0$,
then one can see
from \eqref{Lambda-vs-X}
that $\Lambda$
is also a continuous function
of $\omega$ and $\kappa$.
It follows that
there could be 
at most two
simple eigenvalues $\pm\jj\Lambda$
colliding at $\lambda=0$,
hence the algebraic multiplicity of
eigenvalue $\lambda=0$ cannot jump by more than two.

\medskip

Now we consider
$\bfA(\omega,\kappa)$
in the invariant subspace
$\bfX\sb{\mbox{\rm\footnotesize odd-even-odd-even}}$
of $L^2(\R,\C^4)$
(see~\eqref{def-oeoe}).
By Theorem~\ref{theorem-sect2}~\itref{theorem-sect2-2},
the restriction of $\bfA(\omega,\kappa)$
to this subspace contains eigenvalue
$\lambda=0$
only when $\omega=0$,
with both the geometric and algebraic multiplicities
being equal to two.

\medskip

This completes the proof of Theorem~\ref{theorem-sect2}.

\section{Parity-preserving
perturbation of the Soler model
with concentrated nonlinearity}
\label{sect-nld-perturbation}

In this section we address by perturbative analysis
the effect of changing the Soler nonlinearity
by the term which breaks the $\mathbf{SU}(1,1)$-invariance
while preserving the parity:
the equation is invariant in subspaces
$\bfX_{\mbox{\rm\footnotesize even-odd-even-odd}}$
and
$\bfX_{\mbox{\rm\footnotesize odd-even-odd-even}}$
consisting of odd-even and in even-odd wave functions.

\smallskip

\noindent{\bf Model.}
We perturb the Soler model changing the Lagrangian density
\eqref{Lagrangian-density}
so that the self-interaction is based on the quantity
$\psi\sp\ast(\sigma_3+\epsilon I_2)\psi$, $\epsilon\ne 0$
(instead of $\psi\sp\ast\sigma_3\psi$);
now formally the dynamics is governed by the equation
\begin{eqnarray}\label{nld-point-perturbed-1}
\jj \p_t\psi=(\jj\sigma_2\p_x+\sigma_3 m)\psi
-\delta(x)f(\psi\sp\ast(\sigma_3+\epsilon I_2)\psi)
(\sigma_3+\epsilon I_2)\psi,
\quad
x\in\R,
\quad
t\in\R.
\end{eqnarray}
Above,
$f\in C(\R)\cap C^1(\R\setminus\{0\})$,
$f(0)=0$,
and
the following jump condition on $\psi$ is understood
(cf. \eqref{jumpeq}):
\begin{equation}\label{jumpeq_paritypres}
\jj\sigma_2[\psi]_{0}
=f(\hat\psi\sp\ast(\sigma_3+\epsilon I_2)\hat\psi)(\sigma_3+\epsilon I_2)\hat\psi.
\end{equation}
Just like \eqref{nld-point},
this is a Hamiltonian $\mathbf{U}(1)$-invariant system,
but for $\epsilon\ne 0$
it is no longer $\mathbf{SU}(1,1)$-invariant.

\smallskip

\noindent
{\bf Solitary waves.\ }
Like in \eqref{solitary-wave},
there are solitary wave solutions
$\phi_{\omega,\epsilon}(x)e^{-\jj\omega t}$
to \eqref{nld-point-perturbed-1}
with
\[
\phi_{\omega,\epsilon}(x)
=\alpha(\omega,\epsilon)
\begin{bmatrix}1\\\mu\sgn x\end{bmatrix}
e^{-\varkappa\abs{x}}
\]
and with $\varkappa$, $\mu$ from \eqref{def-def}.
Without loss of generality,
we may assume that $\alpha(\omega,\epsilon)>0$.
The value of $\alpha(\omega,\epsilon)$
is to satisfy the jump condition
\eqref{jumpeq_paritypres}
with
$
[\phi_{\omega,\epsilon}]_0=
2\begin{bmatrix}0\\\alpha\mu(\omega)\end{bmatrix}$,
$\hat\phi_{\omega,\epsilon}=
\begin{bmatrix}\alpha\\0\end{bmatrix}$,
which leads to
$
2\jj\sigma_2\begin{bmatrix}0\\\alpha\mu(\omega)\end{bmatrix}
=
(\sigma_3+\epsilon I_2)
f
\begin{bmatrix}\alpha\\0\end{bmatrix},
$
resulting in
\begin{eqnarray}\label{alpha-mu}
2\mu(\omega)=(1+\epsilon)f(\tau),
\qquad
\tau:=\phi_{\omega,\epsilon}\sp\ast(\sigma_3+\epsilon I_2)
\phi_{\omega,\epsilon}\at{x=0}
=(1+\epsilon)\alpha^2.
\end{eqnarray}

\smallskip

\noindent
{\bf Linearization.\ }
Let us consider the linearization at a solitary wave.
Using the Ansatz
\[
\psi(t,x)=(\phi_{\omega,\epsilon}(x)+\,r(t,x)+\jj s(t,x))e^{-\jj\omega t},
\qquad
r(t,x),\,s(t,x)\in\R^2;
\]
we derive that the perturbation
$(r(t,x),\,s(t,x))$
satisfies the following system (where we omit explicit and repetitive domain definition):
\[
\begin{cases}
-\dot s=D_m r-\omega r
-f\delta(x)(\sigma_3+\epsilon I_2)r
-2g\delta(x)(\phi_{\omega,\epsilon}\sp\ast(\sigma_3+\epsilon I_2)r)(\sigma_3+\epsilon I_2)\phi_{\omega,\epsilon}
=:L_{+}(\epsilon)r,
\\[1ex]
\dot r=D_m s-\omega s
-f\delta(x)(\sigma_3+\epsilon I_2)s
=:L_{-}(\epsilon)s,
\end{cases}
\]
where
\begin{eqnarray}\label{def-f-g-ch3}
f=f(\tau),
\qquad
g=f'(\tau)
\end{eqnarray}
are evaluated at $\tau$ from
\eqref{alpha-mu}.
Explicitly,
\begin{eqnarray*}
L_{-}(\epsilon)s
&=&
D_m s-\omega s
-f\delta(x)(\sigma_3+\epsilon I_2)s,
\\[2ex]
L_{+}(\epsilon)r
&=&
(D_m-\omega)r
-f\delta(x)(\sigma_3+\epsilon I_2)r
-2g\delta(x)\phi_{\omega,\epsilon}^*(\sigma_3+\epsilon I_2)r(\sigma_3+\epsilon I_2)\phi_{\omega,\epsilon}
\\
&=&
(D_m-\omega)r
-f\delta(x)(\sigma_3+\epsilon I_2)r
-2\alpha g\delta(x)(1+\epsilon)r_1
\begin{bmatrix}
(1+\epsilon)\alpha\\0
\end{bmatrix}
\\
&=&
D_m r-\omega r
-f\delta(x)(\sigma_3+\epsilon I_2)r
-2(1+\epsilon)^2 g\alpha^2\delta(x)\Pi_1 r,
\end{eqnarray*}
with $\Pi_1$ from \eqref{def-pi1-pi2}
and with
with $f$, $g$ from \eqref{def-f-g-ch3}.
Thus, the linearization operator is given by
\begin{eqnarray}\label{def-a-not-broken}
\bfA(\epsilon)
=
\begin{bmatrix}
0&D_m-\omega I_2-f\delta(x)(\sigma_3+\epsilon I_2)
\\-D_m+\omega I_2
+
\delta(x)(f\sigma_3+f\epsilon I_2
+2(1+\epsilon)^2 g\alpha^2\Pi_1)
)
&0
\end{bmatrix}.
\end{eqnarray}

We are going to prove that
there are no unstable eigenvalues
bifurcating from $\pm 2\omega\jj$ for $\epsilon\ne 0$.
We first notice that both
$L\sb\pm$ are invariant in the subspace
of $L^2(\R,\C^2)$
consisting of odd-even (and, similarly, even-odd) functions (see Remark \ref{remark-inv} and equation \eqref{domH_linear}).
Since the eigenvalues bifurcating from $\pm 2\omega\jj$ correspond to the invariant subspace
$\bfX_{\mbox{\rm\footnotesize odd-even-odd-even}}$ of $\bfA$
(see \eqref{def-x-x}),
which is also an invariant subspace for $\bfA(\epsilon)$,
it is enough to consider this operator in this subspace only.
(As in the even-odd-even-odd subspace analysis
in Section~\ref{sect-spectrum-eo},
the spectrum of the restriction of $\bfA(\epsilon)$ 
on the invariant subspace
$\bfX_{\mbox{\rm\footnotesize even-odd-even-odd}}$
contains no eigenvalues in the vicinity of the essential spectrum
except possibly near the thresholds $\jj(\pm m\pm\omega)$.)
Moreover, the restrictions of
$L_{-}(\epsilon)$ and $L_{+}(\epsilon)$
onto odd-even spaces are equal,
therefore
\[
\bfA(\epsilon)\big\vert\sb{\bfX\sb{\mbox{\rm\footnotesize odd-even-odd-even}}}
=
\begin{bmatrix}0&L_{-}(\epsilon)\\-L_{-}(\epsilon)&0\end{bmatrix}
=\begin{bmatrix}0&1\\-1&0\end{bmatrix}\otimes L_{-}(\epsilon)
\]
has purely imaginary spectrum.
Let us give a more accurate argument.

\begin{theorem}\label{theorem-instability}
There is
$\omega_0\in(0,m)$ and an open neighborhood $U\subset\R$,
$U\ni 0$,
such that for
$\omega\in(\omega_0,m)$ and $\epsilon\in U$
the operator $\bfA(\epsilon)$ has two eigenvalues
$
\lambda(\epsilon)=\pm\jj(2\omega+\zeta(\epsilon))$,
$\zeta(\epsilon)\in\R\quad \forall\epsilon\in U$,
$\lim\sb{\epsilon\to 0}\zeta(\epsilon)=0$.
\end{theorem}

\begin{proof}
To study
whether $\lambda(\epsilon)=\jj\Lambda(\epsilon)$
is an eigenvalue of the operator $\bfA(\epsilon)$
from \eqref{def-a-not-broken},
we consider the action of $\bfA(\epsilon)-\jj\Lambda(\epsilon) I_4$
onto the superposition
\begin{eqnarray}\label{above-Ansatz}
\Psi(x)
=
a
\begin{bmatrix}
\nu_{+}\sgn x\\S_{+}\\-\jj\nu_{+}\sgn x\\-\jj S_{+}
\end{bmatrix}e^{-\nu_{+}\abs{x}}
+
b
\begin{bmatrix}
-\jj\xi\sgn x\\S_{-}\\\xi\sgn x\\\jj S_{-}
\end{bmatrix}e^{\jj\xi\abs{x}}
+
c
\begin{bmatrix}
\nu_{+}\\S_{+}\sgn x\\-\jj\nu_{+}\\-\jj S_{+}\sgn x
\end{bmatrix}e^{-\nu_{+}\abs{x}}
+
d
\begin{bmatrix}
-\jj\xi\\S_{-}\sgn x\\\xi\\\jj S_{-}\sgn x
\end{bmatrix}e^{\jj\xi\abs{x}},
\end{eqnarray}
with $S_{+}=S_{+}(\omega,\Lambda)$ and $S_{-}=S_{-}(\omega,\Lambda)$ from \eqref{def-r-s}
and with $\nu_{+}$ and $\xi$
defined by
\begin{equation}\label{def-kappa-xi}
\nu_{+}(\omega,\Lambda)=\sqrt{m^{2}-(\Lambda-\omega)^{2}},
\qquad
\xi(\omega,\Lambda)=-\sqrt{(\omega+\Lambda)^{2}-m^{2}}
\end{equation}
(cf. \eqref{def-m-n});
we consider $\lambda=\jj\Lambda$ in the first quadrant,
so that $\Lambda$ (and thus $\zeta$)
has non-positive imaginary part;
then, for $\omega$ sufficiently close to $m$,
\[
\Re\xi
=-\Re((3\omega+\zeta)^2-m^2)^{1/2}
=-\Re(9\omega^2-m^2+6\omega\zeta+\zeta^2)^{1/2}
=-\sqrt{9\omega^2-m^2}
+\mathcal{O}(\zeta)
<0,
\]
\[
\Im\xi
=-(2\sqrt{9\omega^2-m^2})^{-1}6\omega\Im\zeta
+\mathcal{O}(\zeta)\Im\zeta
\ge 0.
\]
Note that the first two terms in
\eqref{above-Ansatz}
are obtained from \eqref{psi-o-e} by substituting
$\nu_{-}(\omega,\Lambda)$ with $-\jj\xi(\omega,\Lambda)$
(both expressions have positive real part)
and correspond to perturbations from
the invariant subspace $\bfX_{\mbox{\rm\footnotesize odd-even-odd-even}}$;
the last two terms
correspond to perturbations
from the invariant subspace $\bfX_{\mbox{\rm\footnotesize even-odd-even-odd}}$.
The relation $(\bfA-\jj\Lambda I_4)\Psi=0$
leads to the following jump condition:
\begin{eqnarray}\label{j-c-0}
\begin{cases}
2(-\jj S_{+} c
+\jj S_{-} d)
-(1+\epsilon)(-\jj\nu_{+} c+\xi d)f
=0
\\-2(-\jj\nu_{+} a+\xi b)
+(1-\epsilon)(-\jj S_{+} a+\jj S_{-}b)f
=0
\\-2(S_{+} c+S_{-} d)
+\big((1+\epsilon)f+2 g\alpha^2(1+\epsilon)^2\big)(\nu_{+} c-\jj\xi d)=0
\\2(\nu_{+} a-\jj\xi b)-(1-\epsilon)(S_{+} a+S_{-} b)f=0.
\end{cases}
\end{eqnarray}
As in the case of the unperturbed operator $\bfA$ (see \eqref{def-a-ch2}),
there are two invariant subspaces
of $\bfA(\epsilon)$
defined in \eqref{def-x-x}:
$\bfX_{\mbox{\rm\footnotesize even-odd-even-odd}}$
corresponding to $a=b=0$
and
$\bfX_{\mbox{\rm\footnotesize odd-even-odd-even}}$
corresponding to $c=d=0$
(note that the system \eqref{j-c-0} does not mix $a,\,b$ and $c,\,d$).
We are interested in the deformation of eigenvalues $\pm 2\omega\jj$
corresponding to $\bfX\sb{\mbox{\rm\footnotesize odd-even-odd-even}}$.

\bigskip

\noindent
$\bullet$
The spectrum of
$\bfA(\epsilon)$ restricted onto
$\bfX\sb{\mbox{\rm\footnotesize even-odd-even-odd}}$.
We do not need to consider this case
since $\bfA(0)$ restricted onto
$\bfX\sb{\mbox{\rm\footnotesize even-odd-even-odd}}$
only has
%% simple
isolated purely imaginary eigenvalues,
which have to stay on imaginary axes because of the symmetries
\eqref{symmetry}.
For completeness, we mention that
in this case
the jump condition \eqref{j-c-0} takes the form
\[
\begin{cases}
2(-\jj S_{+} c+\jj S_{-} d)
-(1+\epsilon)(-\jj\nu_{+} c+\xi d)f
=0
\\-2(S_{+} c+S_{-} d)
+\big((1+\epsilon)f+2g\alpha^2(1+\epsilon)^2\big)(\nu_{+} c-\jj\xi d)=0,
\end{cases}
\]
and the compatibility condition
for having a nontrivial solution $c,\,d\in\C$
is given by
\[
\det
\begin{bmatrix}
-2\jj S_{+}+\jj(1+\epsilon)f\nu_{+}
&
2\jj S_{-}
-(1+\epsilon)f\xi
\\
-2 S_{+}
+\big((1+\epsilon)f+2 g\alpha^2(1+\epsilon)^2\big)\nu_{+}
&
-2 S_{-}
-\big((1+\epsilon)f+2 g\alpha^2(1+\epsilon)^2\big)\jj\xi
\end{bmatrix}
=0.
\]

\noindent
$\bullet$
The spectrum of
$\bfA(\epsilon)$ restricted onto
$\bfX\sb{\mbox{\rm\footnotesize odd-even-odd-even}}$.
The jump condition \eqref{j-c-0} takes the form
\[
\begin{cases}
-2(-\jj\nu_{+} a+\xi b)
+(1-\epsilon)(-\jj S_{+} a+\jj S_{-}b)f
=0
\\2(\nu_{+} a-\jj\xi b)-(1-\epsilon)(S_{+} a+S_{-} b)f=0.
\end{cases}
\]
The compatibility condition is:
\[
\det
\begin{bmatrix}
2\jj\nu_{+}-\jj(1-\epsilon)S_{+} f
&
-2\xi
+\jj(1-\epsilon)S_{-} f
\\2\nu_{+}-(1-\epsilon)S_{+} f
&
-2\jj\xi-(1-\epsilon)S_{-} f
\end{bmatrix}
=
-2\jj (2\nu_{+}-(1-\epsilon)S_{+} f)
(2\jj\xi+(1-\epsilon)S_{-} f)=0.
\]
The deformation of the eigenvalue $2\omega\jj$
corresponds to vanishing of the first factor;
thus,
$
\nu_{+}=\frac{1}{2}(1-\epsilon)S_{+} f;
$
squaring this relation, we arrive at
$
m^2-(\omega+\zeta)^2
=(1-\epsilon)^2
(m+\omega+\zeta)^2 f^2/4$.
This allows us to write
\[
-
\Big(
2\omega-\frac{1}{2}(1-\epsilon)^2(m+\omega)f^2+\frac14(1-\epsilon)^2\zeta^2
\Big)\zeta
=\frac{1}{4}(1-\epsilon)^2
(m+\omega)^2 f^2
-m^2+\omega^2.
\]
Using \eqref{alpha-mu},
we arrive at
\[
-
\Big(
2\omega-(1-\epsilon)^2(m+\omega)
\frac{2\mu^2}{(1+\epsilon)^2}
\Big)\zeta+\frac{(1-\epsilon)^2\zeta^2}{4}
=\frac{(1-\epsilon)^2(m+\omega)^2 \mu^2}{(1+\epsilon)^2}
-m^2+\omega^2
=-\frac{4(m^2-\omega^2)\epsilon}{(1+\epsilon)^2}.
\]
This relation shows that,
for $\abs{\epsilon}$ small enough,
there is a real-valued solution
$
\zeta=
\big(2+\mathcal{O}(\epsilon)\big)
\epsilon(m^2-\omega^2)/m.
$
This completes the proof
of Theorem~\ref{theorem-instability}.
\end{proof}

\section{Broken parity perturbation of the Soler model
with concentrated nonlinearity}
\label{sect-nld-perturbation-broken-parity}

Now
we consider the perturbation
that breaks not only the $\mathbf{SU}(1,1)$ symmetry of the 
Soler model,
but also the parity symmetry:
the linearized equation is no longer invariant
in the subspaces
$\bfX_{\mbox{\rm\footnotesize even-odd-even-odd}}$
and
$\bfX_{\mbox{\rm\footnotesize odd-even-odd-even}}$
of $L^2(\R,\C^4)$,
consisting of even-odd-even-odd and odd-even-odd-even components.
We show that under this perturbation
the weakly relativistic solitary waves
become linearly unstable:
the spectrum of the corresponding linearization
contains the eigenvalues with positive real part;
these eigenvalues bifurcate from $\pm 2\omega\jj$
(see Theorem \ref{theorem-instability-broken}).

\smallskip

\noindent{\bf Model.\ }
We perturb the Soler model
so that the self-interaction term
in the Lagrangian density \eqref{Lagrangian-density}
depends on
\begin{eqnarray}\label{perturbation-sigma-1}
\psi\sp\ast(\sigma_3+\epsilon \sigma_1)\psi,
\end{eqnarray}
$\epsilon\ne 0$,
so that the dynamics is described formally by the equation
\begin{eqnarray}\label{nld-point-perturbed-2}
\jj \p_t\psi=(\jj\sigma_2\p_x+\sigma_3 m)\psi
-\delta(x)f(\psi\sp\ast(\sigma_3+\epsilon\sigma_1)\psi)
(\sigma_3+\epsilon\sigma_1)\psi,
\qquad
x\in\R,
\quad
t\in\R,
\end{eqnarray}
with the pure power nonlinearity
\begin{eqnarray}\label{pp}
f(\tau)=|\tau|^\kappa,
\qquad
\tau\in\R,
\qquad
\kappa>0.
\end{eqnarray}
The following boundary condition for domain elements
is assumed in this section
(see \eqref{psi-hat}, \eqref{psi-zero}, and \eqref{jumpeq}):
\begin{equation}\label{jumpeq_paritymotpres}
\jj\sigma_2[\psi ]_{0}
-f(\hat\psi\sp\ast(\sigma_3+\epsilon \sigma_1)\hat\psi)(\sigma_3+\epsilon\sigma_1)\hat\psi=0.
\end{equation}
Equation \eqref{nld-point-perturbed-2}
is a Hamiltonian $\mathbf{U}(1)$-invariant system
which is no longer $\mathbf{SU}(1,1)$-invariant.
We will show that the perturbation
\eqref{perturbation-sigma-1}
breaks the parity symmetry:
components of the solitary waves are no longer
even or odd,
and the linearization operator at a solitary wave
is no longer invariant
in $\bfX\sb{\mbox{\rm\footnotesize even-odd-even-odd}}$
or $\bfX\sb{\mbox{\rm\footnotesize odd-even-odd-even}}$.

\smallskip

\noindent
{\bf Solitary waves.\ }
The first step of the analysis is to construct solitary waves.
Instead of \eqref{solitary-wave},
the amplitude is now to be
a linear combination of \eqref{s-w-1}
and \eqref{s-w-2}
(since the equation is no longer invariant
in \eqref{def-x-e-o-o-e}), that is,
the solitary waves are now of the form
\begin{equation}\label{phibroken}
\varPhi_{\omega,\epsilon}(x)
=
\left(
\alpha(\omega,\epsilon)\begin{bmatrix}1\\\mu\sgn x\end{bmatrix}
+
\beta(\omega,\epsilon)\begin{bmatrix}\sgn x\\\mu\end{bmatrix}
\right)
e^{-\varkappa\abs{x}},
\end{equation}
with $\varkappa$, $\mu$ from
\eqref{def-def}.
The conditions on $\alpha(\omega,\epsilon)$ and $\beta(\omega,\epsilon)$
come from the jump condition
(cf. \eqref{jumpeq})
\begin{eqnarray}\label{j-c}
\jj\sigma_2
\begin{bmatrix}2\beta\\2\alpha\mu
\end{bmatrix}
-(\sigma_3+\epsilon\sigma_1)f
\begin{bmatrix}\alpha\\\beta\mu\end{bmatrix}=0,
\end{eqnarray}
where
$f=f(\tau)$
with
$\tau:=\hat\varPhi_{\omega,\epsilon}\sp\ast(\sigma_3+\epsilon\sigma_1)
\hat\varPhi_{\omega,\epsilon}$
and
$\hat\varPhi_{\omega,\epsilon}=
\begin{bmatrix}
\alpha\\\beta\mu
\end{bmatrix}
$.
The jump condition \eqref{j-c} takes the form of the following system:
\begin{eqnarray}\label{f-a-b}
\begin{cases}
(f-2\mu)\alpha+f\epsilon\mu \beta=0,
\\f\epsilon \alpha+(2-f\mu)\beta=0.
\end{cases}
\end{eqnarray}
One can see from \eqref{f-a-b} that both $\alpha$ and $\beta$
can simultaneously be chosen real;
without loss of generality, we assume that
\[
\alpha(\omega,\epsilon)>0,
\qquad
\beta(\omega,\epsilon)\in\R.
\]
The compatibility condition leads to
$
0=f^{2} \epsilon^2\mu-(f-2\mu)(2-f\mu)
=f^2\mu(1+\epsilon^{2})-2(1+\mu^2)f+4\mu$,
hence
$
f=\big(1+\mu^2\pm\sqrt{1-2\mu^2+\mu^4-4\mu^2\epsilon^2}\big)/(\mu +\mu\epsilon^2).
$
We need to choose the negative sign at the square root,
so that
$f=2\mu+\mathcal{O}(\epsilon^2)$;
then we are consistent with the case
$\epsilon=0$,
$\omega\in(0,m)$ 
(see \eqref{a-b-f}).
Therefore, one has:
\begin{eqnarray}\label{f-is-f}
&&
f
=
\frac{
1+\mu^2-\sqrt{1-2\mu^2+\mu^4-4\mu^2\epsilon^2}
}{(1+\epsilon^2)\mu}
=\frac{1}{(1+\epsilon^2)\mu}
\left(
1+\mu^2-(1-\mu^2)\sqrt{1-\frac{4\mu^2\epsilon^2}{(1-\mu^2)^2}}
\right)
\nonumber
\\
&&\quad =\frac{2}{1+\epsilon^2}
\Big(
\mu
+\frac{\mu\epsilon^2}{1-\mu^2}
+\mathcal{O}(\epsilon^4\mu^3)
\Big)
=\frac{2}{(1+\epsilon^2)(1-\mu^2)}
\left(
\mu-\mu^3
+\mu\epsilon^2
+\mathcal{O}(\epsilon^4\mu^3)
\right)
\nonumber
\\
&&\quad =
\frac{2}{(1+\epsilon^2)(1-\mu^2)}
\left(
\mu-\mu^3
+\mu\epsilon^2
-\mu^3\epsilon^2
+\mu^3\epsilon^2
+\mathcal{O}(\epsilon^4\mu^3)
\right)
=
2\mu
\left(1+\mathcal{O}(\epsilon^2\mu^2)\right).
\end{eqnarray}
The second equation from \eqref{f-a-b} yields:
\begin{eqnarray}\label{bbroken}
\beta=-\frac{f\epsilon\alpha}{2-f\mu}
=
-
\frac{
2\mu(1+\mathcal{O}(\epsilon^2\mu^2)
)\epsilon\alpha}{
2-
2\mu^2(1+\mathcal{O}(\epsilon^2\mu^2))
}
=-\frac{\epsilon\mu}{1-\mu^2}(1
+\mathcal{O}(\epsilon^2\mu^2)
)\alpha.
\end{eqnarray}
For future use, we compute:
\begin{eqnarray}\label{alpha-beta}
\alpha-\frac{\beta\mu}{\epsilon}
=
\Big(1+
\frac{\mu^2}{1-\mu^2}
(1+\mathcal{O}(\epsilon^2\mu^2))
\Big)\alpha,
\end{eqnarray}
and by \eqref{phibroken} one has
\begin{eqnarray}\label{tau-ch4}
\tau
&=&\varPhi_{\omega,\epsilon}\sp\ast\at{x=0}(\sigma_3+\epsilon\sigma_1)
\varPhi_{\omega,\epsilon}\at{x=0}
=
\begin{bmatrix}\alpha&\beta\mu\end{bmatrix}
\begin{bmatrix}1&\epsilon\\\epsilon&-1\end{bmatrix}
\begin{bmatrix}\alpha\\\beta\mu\end{bmatrix}
=\alpha^2+2\epsilon\alpha\beta\mu-\beta^2\mu^2
\\
\nonumber
&=&
\Big(1
-2\epsilon^2\frac{\mu^2}{1-\mu^2}
(1+\mathcal{O}(\epsilon^2\mu^2))
-\epsilon^2\frac{\mu^2}{(1-\mu^2)^2}
(1+\mathcal{O}(\epsilon^2\mu^2))
\mu^2
\Big)\alpha^2
=
\big(1+\mathcal{O}(\epsilon^2\mu^2)
\big)\alpha^2.
\end{eqnarray}
Combining the above expression for $\tau$
with the relation \eqref{f-is-f} satisfied by $f$,
we derive:
\begin{eqnarray}\label{abroken}
2\mu\big(1+\mathcal{O}(\epsilon^2\mu^2)\big)
=f=\abs{\tau}^\kappa
=
\alpha^{2\kappa}
(1+\mathcal{O}(\epsilon^2\mu^2)),
\qquad
\alpha=
(2\mu)^{\frac{1}{2\kappa}}
\big(1+\mathcal{O}(\epsilon^2\mu^2)\big).
\end{eqnarray}
The solitary wave
is given by the expression \eqref{phibroken} with $\alpha$ and $\beta$
from \eqref{abroken} and \eqref{bbroken}.

\smallskip

\noindent{\bf Linearization.\ }
Let us consider the linearization at the solitary wave
\eqref{phibroken}.
We use the Ansatz
\begin{eqnarray}\label{Ansatz-br}
\psi(t,x)=(\varPhi_{\omega,\epsilon}(x)+r(t,x)+\jj s(t,x))e^{-\jj\omega t},
\qquad
\big(r(t,x),\,s(t,x)\big)\in \R^2\times\R^2.
\end{eqnarray}
A substitution of the Ansatz
\eqref{Ansatz-br}
into equation \eqref{nld-point-perturbed-2}
shows that the perturbation
$(r(t,x),\,s(t,x))$
satisfies the following linearized system:
\[
\begin{cases}
-\dot s=(D_m-\omega) r
-f\delta(x)(\sigma_3+\epsilon \sigma_1)r
-2g\delta(x)
(\varPhi_{\omega,\epsilon}\sp\ast(\sigma_3+\epsilon \sigma_1)r)(\sigma_3+\epsilon \sigma_1)\varPhi_{\omega,\epsilon}
=:L_{+}(\epsilon)r,
\\[1ex]
\dot r=\left(D_m -\omega\right) s
-f\delta(x)(\sigma_3+\epsilon \sigma_1)s
=:L_{-}(\epsilon)s.
\end{cases}
\]
Above (cf. \eqref{tau-ch4}),
\begin{eqnarray}\label{def-f-g-ch4}
f=f(\tau),
\qquad
g=f'(\tau),
\qquad
\tau:=\varPhi_{\omega,\epsilon}\sp\ast
(\sigma_3+\epsilon \sigma_1)\varPhi_{\omega,\epsilon}\at{x=0}
=(\alpha+\epsilon\beta\mu)^2
-\beta^2\mu^2
+\mathcal{O}(\epsilon^4).
\end{eqnarray}
Thus, we have:
\begin{eqnarray}
L_{+}r
&=&
(D_m-\omega)r
-f\delta(x)(\sigma_3+\epsilon\sigma_1)r
-2g\delta(x)(\varPhi_{\omega,\epsilon}\sp\ast(\sigma_3+\epsilon\sigma_1)r)
(\sigma_3+\epsilon\sigma_1)\varPhi_{\omega,\epsilon}
\nonumber
\\
&=&(D_m-\omega)r
-f\delta(x)\cdot(\sigma_3+\epsilon\sigma_1)r
-2g\delta(x)(\alpha (r_1+\epsilon r_2)+\beta\mu(-r_2+\epsilon r_1))
\begin{bmatrix}
\alpha+\epsilon\beta\mu
\\\epsilon\alpha-\beta\mu
\end{bmatrix}
\nonumber
\\
&=&(D_m-\omega)r
-f\delta(x)(\sigma_3+\epsilon\sigma_1)r
-2g\delta(x)((\alpha+\epsilon\beta\mu)r_1+(\alpha\epsilon-\beta\mu)r_2)
\begin{bmatrix}
\alpha+\epsilon\beta\mu
\\\epsilon\alpha-\beta\mu
\end{bmatrix}
\nonumber
\\
&=&
\Big(
D_m-\omega
-f\delta(x)(\sigma_3+\epsilon\sigma_1)
-\delta(x)
\big[
X\Pi_1 
+\epsilon Y\sigma_1 
+\epsilon^2 Z\Pi_2
\big]
\Big)r,
\end{eqnarray}
where $\Pi_1$, $\Pi_2$ are the projectors
from \eqref{def-pi1-pi2}
and the quantities
$X,\,Y,\,Z\in\R$
defined by
\begin{eqnarray}\label{def-x-y-z}
X=2(\alpha+\epsilon\beta\mu)^2g,
\qquad
Y=2
(\alpha+\epsilon\beta\mu)\Big(\alpha-\frac{\beta\mu}{\epsilon}\Big)g
,
\qquad
Z=2\Big(\alpha-\frac{\beta\mu}{\epsilon}\Big)^2g.
\end{eqnarray}
In the pure power case, by \eqref{f-is-f},
one has
\[
\tau g
=\tau f'(\tau)
=\kappa f(\tau)
=2\kappa\mu(1+\mathcal{O}(\epsilon^2\mu^2)),
\]
with $\tau$ from \eqref{tau-ch4};
hence,
using
\eqref{alpha-beta} in \eqref{def-x-y-z},
we have the following estimates:
\begin{eqnarray}\label{x-y-z-such}
X=4\kappa\mu(1+\mathcal{O}(\epsilon^2\mu^2)),
\ \quad
Y=4\kappa\mu(1+\mathcal{O}(\mu)),
\ \quad
Z=4\kappa\mu(1+\mathcal{O}(\mu)).
\end{eqnarray}
We denote
\begin{eqnarray}\label{def-f-f}
F=f+Y
=
f+4\kappa\mu(1+\mathcal{O}(\mu))
=2(1+2\kappa)\mu(1+\mathcal{O}(\mu)),
\end{eqnarray}
so that
$
\begin{array}{l}
L_{+}
=
D_m-\omega
-\delta(x)(f\sigma_3
+F\epsilon\sigma_1
+X\Pi_1 
+\epsilon^2 Z\Pi_2)
$
and
$
L_{-}
=
D_m-\omega
-f\delta(x)(\sigma_3+\epsilon\sigma_1).
\end{array}
$
Now we can write
$\bfA(\epsilon)=\begin{bmatrix}
0&L_{-}\\-L_{+}&0\end{bmatrix}$
in the explicit form
as
\begin{eqnarray}\label{def-a-broken}
\bfA(\epsilon)
=
\begin{bmatrix}
0&\!\!\!\!\!\!
D_m-\omega I_2-f\delta(x)(\sigma_3+\epsilon\sigma_1)
\\-D_m+\omega I_2
+\delta(x)
\big(
f\sigma_3+\epsilon F\sigma_1
+X\Pi_1
+Z\epsilon^2
\Pi_2
\big)
&0
\end{bmatrix},
\end{eqnarray}
with quantities
$f$, $g$ from \eqref{def-f-g-ch4},
$X$, $Z$ from \eqref{def-x-y-z},
$F$ from \eqref{def-f-f},
and with projectors
$\Pi_1$, $\Pi_2$ from \eqref{def-pi1-pi2}.\\ It is now apparent that the parity symmetry is broken (see Remark \ref{remark-inv} and equation \eqref{domH_linear}).

\subsection*{Bifurcations
of eigenvalues from the essential spectrum}

Let $\lambda(\epsilon)$
be the deformation of the eigenvalue $2\omega\jj$ of $\bfA(\epsilon)$ from \eqref{def-a-ch2}
under the perturbation \eqref{nld-point-perturbed-2}.
As before
(see \eqref{lambda-lambda} and Theorem~\ref{theorem-instability}),
let $\Lambda\in\C$ and $\zeta\in\C$
be defined by relations
\begin{eqnarray}\label{lambda-lambda-zeta}
\lambda(\epsilon)=\jj\Lambda(\epsilon),
\qquad
\Lambda(\epsilon)=2\omega+\zeta(\epsilon).
\end{eqnarray}
The condition for the eigenvalue $\lambda(\epsilon)$
bifurcating from $2\omega\jj$
to be inside the first quadrant
(that is, the linear instability condition, $\Re\lambda>0$)
is now $\Im\zeta<0$.

\begin{theorem}\label{theorem-instability-broken}
Let $f(\tau)=\abs{\tau}^\kappa$, $\tau\in\R$\textup; $\kappa>0$.
There is
$\omega_0<m$ and an open neighborhood
$U\subset\R$, $U\ni 0$,
such that for
$\omega\in(\omega_0,m)$
and $\epsilon\in U\setminus\{0\}$
the spectrum
$\sigma\sb{\mathrm{p}}(\bfA(\epsilon))$
contains
eigenvalues
$\pm\lambda(\epsilon)$
and $\pm\overline{\lambda(\epsilon)}$,
with
\[
\lambda(\epsilon)=\jj(2\omega+\zeta(\epsilon)),
\qquad
\Im\zeta(\epsilon)<0\quad \forall\epsilon\in U\setminus\{0\},
\qquad
\lim\sb{\epsilon\to 0}\zeta(\epsilon)=0.
\]
That is, the solitary waves corresponding to
$\omega\in(\omega_0,m)$
are spectrally unstable.
\end{theorem}

\begin{proof}
As in the proof of Theorem~\ref{theorem-instability},
to study
whether
$
\lambda(\epsilon)=\jj\Lambda(\epsilon)
$
is an eigenvalue
of the operator $\bfA(\epsilon)$
from \eqref{def-a-broken},
we consider the action of $\bfA(\epsilon)-\jj\Lambda(\epsilon) I_4$
onto the superposition
\[
\Psi
=
a
\begin{bmatrix}
\nu_{+}\sgn x\\S_{+}\\-\jj\nu_{+}\sgn x\\-\jj S_{+}
\end{bmatrix}e^{-\nu_{+}\abs{x}}
+
b
\begin{bmatrix}
-\jj\xi\sgn x\\S_{-}\\\xi\sgn x\\\jj S_{-}
\end{bmatrix}e^{\jj\xi\abs{x}}
+
c
\begin{bmatrix}
\nu_{+}\\S_{+}\sgn x\\-\jj\nu_{+}\\-\jj S_{+}\sgn x
\end{bmatrix}e^{-\nu_{+}\abs{x}}
+
d
\begin{bmatrix}
-\jj\xi\\S_{-}\sgn x\\\xi\\\jj S_{-}\sgn x
\end{bmatrix}e^{\jj\xi\abs{x}}.
\]
Above,
$S_{+}=S_{+}(\omega,\Lambda)$ and $S_{-}=S_{-}(\omega,\Lambda)$ are from \eqref{def-r-s}
and
$\nu_{+}=\nu_{+}(\omega,\Lambda)$ and $\xi=\xi(\omega,\Lambda)$ are given by \eqref{def-kappa-xi}.
The jump condition at $x=0$
leads to the relations
\[
\begin{cases}
2(-\jj S_{+} c+\jj S_{-} d)
-(-\jj\nu_{+} c+\xi d)f
-\epsilon (-\jj S_{+} a+\jj S_{-} b)f
=0,
\\-2(-\jj\nu_{+} a+\xi b)
+(-\jj S_{+} a+\jj S_{-} b)f
-\epsilon (-\jj\nu_{+} c+\xi d)f
=0,
\\-2(S_{+} c+S_{-} d)
+(f+X)(\nu_{+} c-\jj\xi d)
+\epsilon (S_{+} a+S_{-} b)F
=0,
\\2(\nu_{+} a-\jj\xi b)
-(f-\epsilon^2 Z)(S_{+} a+S_{-} b)
+\epsilon (\nu_{+} c-\jj\xi d)F
=0,
\end{cases}
\]
with
$f$ from \eqref{def-f-g-ch4},
$F=f+Y$ from \eqref{def-f-f},
and $X$, $Y$, $Z$ from \eqref{def-x-y-z}.
Above, the first terms in the left-hand side
correspond to the contributions from the derivative.
The assumption that
$a,\,b,\,c,\,d\in\C$ are not simultaneously zeros
leads to the condition
\[
\det
\begin{bmatrix}
\jj\epsilon S_{+} f&-\jj\epsilon S_{-} f&-2\jj S_{+}+\jj f\nu_{+}&2\jj S_{-}-f\xi
\\
2\jj\nu_{+}-\jj S_{+} f&-2\xi+\jj S_{-} f&\jj\epsilon f \nu_{+}&-\epsilon f \xi
\\
\epsilon S_{+} F&\epsilon S_{-} F&-2 S_{+}+(f+X)\nu_{+}&-2 S_{-}-\jj(f+X)\xi
\\
2\nu_{+}-(f-\epsilon^2 Z)S_{+}&-2\jj\xi-(f-\epsilon^2 Z)S_{-}&\epsilon F\nu_{+}&-\jj\epsilon F\xi
\end{bmatrix}
=0,
\]
which we rewrite as
\begin{eqnarray}\label{det-det}
\hskip -10pt
\det
\begin{bmatrix}
2\nu_{+}-S_{+} f&-2\jj\xi-S_{-} f&\epsilon f \nu_{+}&-\jj\epsilon f \xi
\\2\nu_{+}-(f-\epsilon^2 Z)S_{+}&2\jj\xi+(f-\epsilon^2 Z)S_{-}&\epsilon F\nu_{+}&\jj\epsilon F\xi
\\\epsilon S_{+} f&\epsilon S_{-} f&-2 S_{+}+f\nu_{+}&-2 S_{-}-\jj f\xi
\\\epsilon S_{+} F&-\epsilon S_{-} F&-2 S_{+}+(f+X)\nu_{+}&2 S_{-}+\jj(f+X)\xi
\end{bmatrix}
=0.
\end{eqnarray}
Let $A$, $B$, $C$, and $D$
be the $2\times  2$ matrices 
so that the above matrix
is written in the block form as
$
\begin{bmatrix}
A&\epsilon B\\\epsilon C&D
\end{bmatrix}
$;
that is,
\begin{eqnarray}
\label{matrix-a}
A=\begin{bmatrix}
2\nu_{+}-S_{+} f&-S_{-} f-2\jj\xi
\\2\nu_{+}-S_{+} f+\epsilon^2 Z S_{+}&2\jj\xi+S_{-} f-\epsilon^2 Z S_{-}
\end{bmatrix},
\qquad
B=
\begin{bmatrix}
f \nu_{+}&-\jj f \xi
\\F\nu_{+}&\jj F\xi
\end{bmatrix},
\end{eqnarray}
\begin{eqnarray}
\label{matrix-c-d}
C=
\begin{bmatrix}
S_{+} f&S_{-} f
\\S_{+} F&-S_{-} F
\end{bmatrix},
\qquad
D=
\begin{bmatrix}
-2 S_{+}+f\nu_{+}&-2 S_{-}-\jj f\xi
\\-2 S_{+}+(f+X)\nu_{+}&2 S_{-}+\jj(f+X)\xi
\end{bmatrix};
\end{eqnarray}
recall that $S_{+}$, $S_{-}$ are from \eqref{def-r-s}
and $\nu_{+}$, $\xi$ are from \eqref{def-kappa-xi}.
Since $\lim\sb{\omega\to m,\,\Lambda\to 2m}\det D=32 m^2$ (see \eqref{det-d} below),
we can use the Schur complement of $D$ to rewrite
\eqref{det-det} as
$\det(A-\epsilon^2 M)=0$,
with $M=B D^{-1}C$.
We have:
\begin{eqnarray}\label{def-m}
M:=B D^{-1}C=\frac{1}{\det D}
\begin{bmatrix}
f\nu_{+}&-\jj f\xi
\\F\nu_{+}&\jj F\xi
\end{bmatrix}
\begin{bmatrix}
2 S_{-}+\jj (f+X) \xi&2 S_{-}+\jj f \xi
\\2 S_{+}-(f+X)\nu_{+}&-2 S_{+}+f\nu_{+}
\end{bmatrix}
\begin{bmatrix}
S_{+} f&S_{-} f
\\S_{+} F&-S_{-}F
\end{bmatrix}.
\end{eqnarray}
Taking
into account that $f=\mathcal{O}(\mu)$ (see \eqref{f-is-f}),
$F=\mathcal{O}(\mu)$ (see \eqref{def-f-f}),
$S_{+}+S_{-}=2(m-\omega)=\mathcal{O}(\mu^2)$,
\begin{eqnarray*}
&&
\hskip -10pt
M=\frac{1}{\det D}
\begin{bmatrix}
f\nu_{+}&-\jj\xi f
\\
F\nu_{+}&\jj \xi F
\end{bmatrix}
\begin{bmatrix}
2 S_{-}&2 S_{-}
\\2 S_{+}&-2 S_{+}
\end{bmatrix}
\begin{bmatrix}
S_{+} f&S_{-} f
\\S_{+} F&-S_{-} F
\end{bmatrix}
+\mathcal{O}(\mu^3)
\\
&&
\hskip -10pt
=\frac{2 S_{+}^2}{\det D}
\begin{bmatrix}
f\nu_{+}&-\jj\xi f\\F\nu_{+}& \jj\xi F
\end{bmatrix}
\!
\begin{bmatrix}
-1&-1\\1&-1
\end{bmatrix}
\!
\begin{bmatrix}
f&-f\\F&F
\end{bmatrix}
+\mathcal{O}(\mu^3)
=
\frac{2 S_{+}^2}{\det D}
\begin{bmatrix}
f\nu_{+}&-\jj\xi f\\F\nu_{+}& \jj\xi F
\end{bmatrix}
\!
\begin{bmatrix}-f-F&f-F\\f-F&-f-F\end{bmatrix}
+\mathcal{O}(\mu^3)
\\
&&
\hskip -10pt
=\frac{2 S_{+}^2}{\det D}
\begin{bmatrix}
(-f-F)f\nu_{+}-\jj(f-F)f\xi
&
(f-F)f\nu_{+}-\jj(-f-F)f\xi
\\
(-f-F)F\nu_{+}+\jj(f-F)F\xi
&
(f-F)F\nu_{+}
+\jj(-f-F)F\xi
\end{bmatrix}
+\mathcal{O}(\mu^3).
\end{eqnarray*}
In the second line, we substituted $S_{-}$ by $-S_{+}$,
with the error counted in the $\mathcal{O}(\mu^3)$ term.
It follows that
\begin{eqnarray}\label{m11-m21}
M_{11}+M_{21}
=
2(\det D)^{-1}S_{+}^2
\big(
-(f+F)^2\nu_{+}
-\jj(F-f)^2\xi
\big)
+\mathcal{O}(\mu^3)
\nonumber
\\
=
-2({\det D})^{-1}S_{+}^2
\big(
(2f+Y)^2\nu_{+}
+\jj Y^2\xi
\big)
+\mathcal{O}(\mu^3),
\end{eqnarray}
with $Y$ from \eqref{def-x-y-z}.
Taking into account that
$A_{21}=A_{11}+\epsilon^2 Z S_{+}$
and
$A_{22}=-A_{12}-\epsilon^2 Z S_{-}$,
we derive:
\begin{eqnarray*}
&&\det(A-\epsilon^2 B D^{-1}C)
=
(A_{11}-\epsilon^2 M_{11})(A_{22}-\epsilon^2 M_{22})
-
(A_{12}-\epsilon^2 M_{12})(A_{21}-\epsilon^2 M_{21})
\\&&
=
(A_{11}-\epsilon^2M_{11})(-A_{12}-\epsilon^2 Z S_{-} -\epsilon^2 M_{22})
-
(A_{12}-\epsilon^2 M_{12})(A_{11}+\epsilon^2 Z S_{+}-\epsilon^2 M_{21})=0,
\end{eqnarray*}
\begin{eqnarray*}
-2A_{11}A_{12}
-A_{11}
(\epsilon^2 Z S_{-} +\epsilon^2 M_{22}-\epsilon^2 M_{12})
+A_{12}(\epsilon^2 M_{11}-\epsilon^2 Z S_{+}+\epsilon^2 M_{21})
\\
+\epsilon^4 M_{11}(Z S_{-}+M_{22})
+\epsilon^4 M_{12}(Z S_{+}-M_{21})
=0,
\end{eqnarray*}
\[
A_{11}
=
\epsilon^2
\frac{A_{12}(M_{11}+M_{21}-Z S_{+})
+\epsilon^2
(M_{11}(Z S_{-}+M_{22})+M_{12}(Z S_{+}-M_{21}))}
{2A_{12}+\epsilon^2(Z S_{-}+M_{22}-M_{12})},
\]
\[
2\nu_{+}=S_{+} f
+
\epsilon^2
\frac{A_{12}(M_{11}+ M_{21}-Z S_{+})
+\epsilon^2
(M_{11}(Z S_{-}+M_{22})+M_{12}(Z S_{+}-M_{21}))}
{2A_{12}+\epsilon^2(Z S_{-}+M_{22}-M_{12})}.
\]
Substituting
$\nu_{+}=\sqrt{m^2-(\omega-\Lambda)^2}=\sqrt{m^2-(\omega-(2\omega+\zeta))^2}$
(see \eqref{lambda-lambda-zeta}),
we arrive at
\begin{eqnarray}\label{zeta-is-0}
\hskip -15pt
\zeta^2+2\omega\zeta
=
m^2-\omega^2
-\Big(
\frac{S_{+} f}{2}
+
\epsilon^2
\frac{M_{11}+M_{21}-Z S_{+}
+\frac{\epsilon^2}{A_{12}}
(M_{11}(Z S_{-}+M_{22})+M_{12}(Z S_{+}-M_{21}))}
{4+2\epsilon^2(Z S_{-}+M_{22}-M_{12})/A_{12}}
\Big)^2.
\end{eqnarray}
Taking into account \eqref{f-is-f} and \eqref{x-y-z-such},
the entries of $M$ from \eqref{def-m}
are estimated by
$M_{i j}=\mathcal{O}(\mu^2)$,
$1\le i,\,j\le 2$;
since
$A_{12}=-S_{-} f-2\jj\xi\longrightarrow  -4\jj m\sqrt{2}$
in the limit $\epsilon\to 0$, $\omega\to m$,
$\Lambda\to 2m$,
\eqref{zeta-is-0} yields the relation
\begin{eqnarray}\label{zeta-is}
\zeta^2+2\omega\zeta
=
m^2-\omega^2
-\Big(
\frac{S_{+} f}{2}
+
\epsilon^2
\frac{M_{11}+M_{21}-S_{+}Z
+
\mathcal{O}(\epsilon^2\mu^3)
}
{4-
\mathcal{O}(\epsilon^2\mu)}
\Big)^2.
\end{eqnarray}
Writing
\begin{eqnarray}\label{large-error}
\zeta^2+2\omega\zeta
&=&
(m+\omega)^2\mu^2
-
\Big(
\frac{1}{2}S_{+} f
+\mathcal{O}(\epsilon^2\mu)
\Big)^2
\nonumber
\\&=&(m+\omega)^2\mu^2
-(m+\omega+\zeta)^2f^2/4
-S_{+} f
\epsilon^2 \mathcal{O}(\mu)
+\mathcal{O}(\epsilon^4\mu^2)
\end{eqnarray}
(note that
the largest error term,
$\mathcal{O}(\epsilon^4\mu^2)$,
is contributed by squaring
$\epsilon^2 S_{+}Z$
in the right-hand side of \eqref{zeta-is}),
we have:
\begin{eqnarray*}
\zeta^2+2\omega\zeta
&=&
(m+\omega)^2
\Big(\mu^2-\frac{f^2}{4}\Big)
-
\frac{2(m+\omega)\zeta+\zeta^2}{4}f^2
+\mathcal{O}(\epsilon^2\mu^2)
\\&=&
(m+\omega)^2\mathcal{O}(\epsilon^2\mu^2)
-\mu^2\big(2(m+\omega)\zeta+\zeta^2\big)
+\mathcal{O}(\epsilon^2\mu^2),
\end{eqnarray*}
hence
$\zeta=\mathcal{O}(\epsilon^2\mu^2)$.
In view of this,
\begin{eqnarray}\label{r-s-m}
\begin{array}{l}
S_{+}(\omega,\Lambda)=m+\Lambda-\omega
=m+\omega+\zeta=2m+\mathcal{O}(\mu^2),
\\[1ex]
S_{-}(\omega,\Lambda)=m-\omega-\Lambda=-2m+\mathcal{O}(\mu^2),
\\[1ex]
\nu_{+}(\omega,\Lambda)
=\sqrt{m^2-(\Lambda-\omega)^2}
=\sqrt{m^2-(\omega+\zeta)^2}
=\mathcal{O}(\mu).
\end{array}
\end{eqnarray}
Now we can compute the determinant of the matrix
$D$ from \eqref{matrix-c-d}:
\begin{eqnarray}\label{det-d}
&&
\det D
=
(-2 S_{+}+f\nu_{+})(2 S_{-}+\jj(f+X)\xi)
+(2 S_{-}+\jj f\xi)(-2 S_{+}+(f+X)\nu_{+})
\nonumber
\\
&&
=
-8S_{+}S_{-}
+2 (2f+X)S_{-}\nu_{+}
+\jj
\big(
2(f+X)f\nu_{+}\xi
-2 (2f+X)S_{+}\xi
\big)
=32m^2+\mathcal{O}(\mu),
\qquad
\end{eqnarray}
with the error term being complex-valued.
Taking the imaginary part of
\eqref{zeta-is},
we obtain:
\begin{eqnarray}
\label{n-d}
2(\omega+\Re\zeta)
\Im\zeta
=
-
\epsilon^2
S_{+} f\Im
\frac{M_{11}+M_{21}-S_{+}Z
+
\mathcal{O}(\epsilon^2\mu^3)}
{4-\mathcal{O}(\epsilon^2\mu)}
+\mathcal{O}(\epsilon^4\mu^3).
\end{eqnarray}
\begin{remark}
In \eqref{n-d}, the error term is
$\mathcal{O}(\epsilon^4\mu^3)$
(instead of $\mathcal{O}(\epsilon^4\mu^2)$ as in
\eqref{large-error});
indeed, by \eqref{r-s-m},
\begin{eqnarray}\label{z-r-small}
S_{+}Z
=(m+\omega+\mathcal{O}(\epsilon^2\mu^2))Z;
\end{eqnarray}
since $Z$ from \eqref{def-x-y-z} is real-valued,
$(\epsilon^2 S_{+}Z)^2$
cannot contribute
$\mathcal{O}(\epsilon^4\mu^2)$
to the imaginary part of the right-hand side.
\end{remark}

Since the numerator in \eqref{n-d}
is $\mathcal{O}(\mu)$,
and so is the factor $S_{+} f$,
we conclude that neglecting
$\mathcal{O}(\epsilon^2\mu)$ terms from
the denominator contributes the error absorbed into
$\mathcal{O}(\epsilon^4\mu^3)$, so
\[
2(\omega+\Re\zeta)
\Im\zeta
=
-
({\epsilon^2}/{4})
S_{+} f\Im
\left(M_{11}+M_{21}-S_{+}Z
\right)
+\mathcal{O}(\epsilon^4\mu^3).
\]
Using \eqref{m11-m21}
(where in view of \eqref{r-s-m}
one has $(2f+Y)^2\nu_{+}=\mathcal{O}(\mu^3)$), \eqref{det-d},
and taking into account
the fact that
$Z=\mathcal{O}(\mu)$ is real-valued
while $S_{+}=m+\omega+\zeta$,
we continue:
\begin{eqnarray}\label{finalformula}
&&
2(\omega+\Re\zeta)\Im\zeta
=
-(\epsilon^2/4)
S_{+} f
\Im
\big(M_{11}+M_{21}-S_{+}Z\big)
+\mathcal{O}(\epsilon^4\mu^3)
\\&&
\nonumber
=
(\epsilon^2/4)
S_{+} f
\Im
\Big(
\frac{2\jj\xi S_{+}^2}{\det D}Y^2+\mathcal{O}(\mu^3)
+Z\zeta
\Big)
+\mathcal{O}(\epsilon^4\mu^3)
=
(\epsilon^2/4)
f
\frac{\xi S_{+}^3}{16 m^2}Y^2
+\mathcal{O}(\epsilon^2\mu^4)
+\mathcal{O}(\epsilon^4\mu^3).
\end{eqnarray}
Taking into account the relations
\[
\lim\sb{\omega\to m,\,\Lambda\to 2m}S_{+}(\omega,\Lambda)=2m,
\qquad
\lim\sb{\omega\to m,\,\Lambda\to 2m}\xi(\omega,\Lambda)=-2m\sqrt{2},
\qquad
Y=4\kappa\mu(1+\mathcal{O}(\mu))
\]
(see \eqref{def-r-s}, 
\eqref{def-kappa-xi},
\eqref{x-y-z-such}),
we conclude from \eqref{finalformula} that
there is $c>0$ such that
$
\Im\zeta
<-c\epsilon^2\mu^3$,
as long as
$\abs{\epsilon}$ and $\mu>0$ are sufficiently small.
It follows that the eigenvalue
$\lambda=(2\omega+\zeta)\jj$ moves to the right of the imaginary axis,
becoming an eigenvalue with positive real part and indicating
the linear instability of the corresponding solitary wave.
\end{proof}

\vskip10pt
\noindent
{\bf Acknowledgments.}
\ The authors are grateful to the anonymous referees
for pointing out several omissions in the article.
D. Noja
acknowledges for funding the EC grant IPaDEGAN (MSCA-RISE-778010).
This work was supported by a grant from the Simons Foundation (851052, A.C.).

%%ADD YOUR OWN BIBLIOGRAPHY FILES IF NEEDED!
%\bibliographystyle{siamplain}
\bibliographystyle{sima-doi}
\bibliography{dirac-delta,bibcomech}
\end{document}